\documentclass[11pt]{amsart}
\usepackage[bottom=1in, right=1in, left=1in, top=1in]{geometry}

\usepackage[dvipsnames,table]{xcolor}
\usepackage{appendix}

\RequirePackage[OT1]{fontenc}
\RequirePackage{amsthm,amsmath}
\RequirePackage[colorlinks,citecolor=magenta,urlcolor=blue,linkcolor=magenta]{hyperref}
\RequirePackage{comment}
\usepackage{caption}
\usepackage{graphicx}
\usepackage{subfigure}
\RequirePackage{latexsym,amssymb}

\RequirePackage{float,epsfig,multirow,rotating,times}
\RequirePackage{upgreek,wrapfig}
\usepackage{tikz}
\usepackage{float}
\usetikzlibrary{matrix}
\usetikzlibrary{arrows}
\usetikzlibrary{decorations.pathreplacing,calligraphy}
\usepackage{nicematrix}
\usepackage{blkarray}
\usepackage{mathdots}
\usepackage{multicol}
\usepackage{bbm}

\usepackage{indentfirst}
\usepackage{textcomp}
\usepackage{lmodern}
\usepackage{mathrsfs}
\usepackage{url}
\usepackage{relsize}
\usepackage{tcolorbox}

\usepackage{ marvosym }
\usepackage[english]{babel}
\usepackage{makecell}
\usepackage{mathtools}
\usepackage{verbatim}
\usepackage[latin1]{inputenc}

\numberwithin{equation}{section}
\theoremstyle{plain}
\newtheorem*{theorem*}{Theorem}
\newtheorem{theorem}{Theorem}
\numberwithin{theorem}{section}
\newtheorem{proposition}[theorem]{Proposition}
\newtheorem{lemma}[theorem]{Lemma}
\newtheorem{corollary}[theorem]{Corollary}

\theoremstyle{definition}
\newtheorem{definition}[theorem]{Definition}
\newtheorem{example}[theorem]{Example}

\newtheorem{remark}[theorem]{Remark}

\newcommand{\Z}{{\mathbb Z}}

\newcommand{\PP}{{\mathbb P}}

\def\codim{{\rm codim}}

\newcommand{\ko}{{\mathcal O}}

\newcommand{\Mg}{{\mathcal M}_G}
\newcommand{\Vc}{{\mathcal V}_{X,G}}

\makeatletter
\renewcommand*\env@matrix[1][*\c@MaxMatrixCols c]{%
    \hskip -\arraycolsep
    \let\@ifnextchar\new@ifnextchar
    \array{#1}}
\makeatother

\def\ci{\perp\!\!\!\perp}

\title[{Totally mixed conditional independence equilibria of generic games}]{Totally mixed conditional independence equilibria \\ of generic games}

\author[Matthieu Bouyer]{Matthieu Bouyer}
\address{Ecole polytechnique and Max Planck Institute for Mathematics in the Sciences}
\email{matthieu.bouyer@polytechnique.org} 

\author[Irem Portakal]{Irem Portakal}
\address{Max Planck Institute for Mathematics in the Sciences}
\email{mail@irem-portakal.de} 

\author[Javier Sendra-Arranz]{Javier Sendra-Arranz}
\address{Department of Mathematics, CUNEF Universidad, Madrid, Spain}
\email{javier.sendraarranz@cunef.edu}

\begin{document}

\begin{abstract}
    This paper further develops the algebraic--geometric foundations of conditional independence (CI) equilibria, a refinement of dependency equilibria that integrates conditional independence relations from graphical models into strategic reasoning and thereby subsumes Nash equilibria. Extending earlier work on binary games, we analyze the structure of the associated Spohn CI varieties for generic games of arbitrary format. We show that for generic games the Spohn CI variety is either empty or has codimension equal to the sum of the players' strategy dimensions minus the number of players in the parametrized undirected graphical model. When non-empty, the set of totally mixed CI equilibria forms a smooth manifold for generic games. For cluster graphical models, we introduce the class of Nash CI varieties, prove their irreducibility, and describe their defining equations, degrees, and conditions for the existence of totally mixed CI equilibria for generic games.
\end{abstract}
\maketitle

\section{Introduction}
The concept of dependency equilibria, introduced by Spohn in 2003 \cite{spohn2003dependency}, was first illustrated through the classical Prisoner's Dilemma, where the cooperative outcome emerges as a dependency equilibrium but not as a Nash equilibrium (\cite{nash1950equilibrium}). This notion allows for communication among all players, enabling outcomes that lie beyond classical non-cooperative reasoning. \emph{Conditional independence (CI) equilibria} generalize this idea by permitting only selected subsets of dependencies among players through \emph{undirected discrete graphical models}, a viewpoint particularly natural for certain game structures (Section~\ref{subsec: examples}). While Nash equilibria generically are finitely many, that is not the case for the set of CI equilibria, which often has high dimension. In order to understand their geometry and behavior, we study their algebro-geometric model called \emph{Spohn CI varieties}. Consider, for instance, a generic $2 \times 3$ game: such a game admits no totally mixed Nash equilibrium, that is, no equilibrium where every pure strategy is played with strictly positive probability. Yet, the corresponding algebraic model of dependency equilibria forms a del Pezzo surface of degree two (\cite[Example 10]{portakal2022geometry}). Similarly, one observes a $2 \times 2 \times 2$ game (Example~\ref{ex 2 2-2}) which, when two of the players act in partial coordination, yields totally mixed CI equilibria which Pareto dominate the Nash equilibrium of this game. In this case, the set of CI equilibria forms a smooth curve. These two scenarios motivate the relevance of CI equilibria and their algebro-geometric study in this paper. 

Building on this, the present paper extends the framework of dependency equilibria to conditional independence equilibria, addressing specific limitations of classical non-cooperative game-theoretic models.
Recently in \cite{portakal2024nash,portakal2024game}, such equilibria have been examined for generic binary games, revealing structural properties of their underlying algebro-geometric model Spohn conditional independence varieties. In this paper, we broaden this analysis to games of arbitrary size, exploring their dimensions, degrees, and conditions for non-emptiness, and connect these findings to parametrized undirected graphical models. The results highlight cases where the set of conditional independence equilibria may be empty, in contrast to the binary case, and provide partial answers to open questions such as irreducibility and smoothness of these varieties.\\
\indent In Section~\ref{sec: framework and examples}, we begin with an overview of the classical notion of discrete conditional independence models from undirected graphical models in algebraic statistics. We then introduce the concept of dependency equilibria, presenting its algebraic formulation via the so-called Spohn variety $\mathcal V_X$, and explain its connections to the established notion of Nash equilibria. We recall the original definition of the Spohn CI variety (Definition~\ref{Definition 1.25}) from \cite{portakal2022geometry}, which is obtained by removing all the irreducible components lying in specific hyperplanes from the intersection of $\mathcal V_X$ with the parametrized undirected graphical model $\mathcal M_G$. We then define and illustrate the concept with several detailed examples. In Section~\ref{sec:dim}, we focus on the dimension of Spohn CI varieties for generic games, which was left open except for the case of games with binary choices (\cite[Theorem 9]{portakal2024game}).  Using the monomial parametrization map $\phi_G$ of $\mathcal{M}_G$ (Definition~\ref{def: parametrized CI model}) and removing certain components according to the definition of Spohn CI varieties (which includes a saturation), we derive generators for a variety containing the preimage of the Spohn CI variety under $\phi_G$. Moreover, we analyze its linear system (Lemma~\ref{lemma:gens lin system}) and base locus (Lemma~\ref{lem:base locus}). Finally, by applying Bertini's theorem, we determine the codimension of the Spohn CI variety in $\mathcal{M}_G$ for a generic game $X$ which is independent of $G$.

\begin{theorem}[Theorem~\ref{theo:dimension}, Theorem~\ref{theo:smooth}]
Let $G$ be an undirected discrete graphical model. Then the Spohn CI variety $\mathcal{V}_{X,G} \subseteq \mathcal M_G$ is empty for a generic $d_1 \times \cdots \times d_n$ game $X$, or it has codimension $d_1 + \cdots + d_n - n$ in $\mathcal M_G$ for a generic $d_1 \times \cdots \times d_n$ game $X$. Moreover, for a generic game the set of totally mixed conditional independence equilibria is either empty or a smooth manifold.
\end{theorem}

In Section~\ref{sec:saturation}, we explore that for generic games, the Spohn CI varieties can be defined equivalently through an alternative construction, using different hyperplanes for the saturation appearing in the definition of Spohn varieties. Section~\ref{sec:nash ci} focuses on the special family of graphical models whose associated graph is a cluster graph. A cluster graph is a disjoint union of $k$ cliques for $k\in\mathbb{N}$. We denote by $D_i$ the product of pure strategies of each player in the $i$th clique. Then the parametrized discrete undirected graphical model is the Segre variety $\Mg=\PP^{D_1-1}\times\cdots\times\PP^{D_k-1}$. This is the case where players are acting together in groups and as Nash equilibria can be seen as the trivial case (disjoint union of zero dimensional cliques), we call the Spohn CI variety in this case the \textit{Nash CI variety}. We determine the defining equations of the Nash CI variety in Proposition~\ref{prop:eq Nash CI} and prove that the Nash CI variety is irreducible for generic games in Theorem~\ref{theo: irreducibility}. The former result allows us to study their Chow classes and determine their degrees in Proposition~\ref{prop:degree Nash} and the sufficient and necessary condition when the Nash CI variety is nonempty for a generic game.
\begin{theorem}[Theorem~\ref{theo:emptyness}]\label{theorem 1.2}
    Let $\mathcal S$ be the set of isolated vertices of a cluster graph $G$. Then, the Nash CI variety is nonempty for a generic $d_1 \times \cdots \times d_n$ game if and only if
    \begin{equation}\label{eq: final nonempty1}
    d_i\leq1+\frac12 \sum_{l=1}^k(D_l-1)
    \end{equation}
    for every $i\in\mathcal S$.
\end{theorem}
Combined with the filtration of the Spohn CI variety in Lemma~\ref{lem:inclusion} and Corollary~\ref{cor:emptiness}, this theorem provides a characterization of the emptiness for graphs beyond just cluster graphs. The same inequality \eqref{eq: final nonempty1}, though with slightly different representation of the parameters, also appears in \cite[Proposition 2.3, Chapter 13]{Gelfand1994} concerning the discriminant of a multihomogeneous form defined on the product of projective spaces $\Mg$. This inequality also generalizes the condition known for totally mixed Nash equilibria, namely $d_i -1 \leq \sum_{j \neq i} (d_j -1)\text{, for all } i \in [n]$
as shown in \cite[Theorem 2.7]{abo2025vector}.
\\
\indent Next, in Example~\ref{ex: CI beats Nash}, we present a $4 \times 2 \times 2$ game that does not admit any totally mixed Nash equilibria. However, the Nash CI variety of the one-edge graph, where vertices 2 and 3 are connected, is a line that intersects the interior of the probability simplex, implying the existence of totally mixed CI equilibria. The examples along the paper show that conditional independence equilibria can exist even when Nash equilibria do not, and that CI equilibria can give better expected payoff than Nash equilibria, highlighting the value of studying both equilibrium concepts. We conclude our paper with Appendix~\ref{appendix: degree CI} on the degree formula for certain discrete conditional independence models that may be useful for the future study of the degree formula for Spohn CI varieties in full generality.

\section{Framework and examples}\label{sec: framework and examples}
The goal of this section is to introduce \emph{conditional independence (CI) equilibria} for finite normal-form games. We begin with the definitions and notations for undirected graphical models arising in algebraic statistics. We then present the concept of \emph{Nash} and \emph{dependency equilibria}, where the latter is a new equilibrium notion in game theory that forms the basis of our investigation. Finally, we define CI equilibria as a natural generalization of Nash equilibria alongside its algebro-geometric model \emph{Spohn CI variety}.

\subsection{Discrete conditional independence models}
\label{subsec CI} We refer to \cite[Chapter 4, Chapter 13]{sullivant2018algebraic} for more details. Let $\mathcal{X}_1,\mathcal{X}_2,\ldots,\mathcal{X}_n$ be $n$ discrete random variables such that each $\mathcal{X}_k$ has state space $[d_k]:=\{1, \cdots, d_k\}$, for $k \in [n]$. For $S\subseteq [n]$ we use the notation $\mathcal{X}_S=(\mathcal{X}_k)_{k\in S}$ with state space $\mathcal R_S=\prod_{k\in S}[d_k]$. We denote by $p:=(p_{i_1i_2\cdots i_n})$ the joint distribution of the random variables $\mathcal{X}_k$. For any coordinate $i\in[n]$, replacing it by a ``$+$" means that we sum over this coordinate. Moreover, if $A_1,A_2,\ldots,A_r$ are $r$ disjoint subsets of $[n]$, we denote the marginal of $(\mathcal{X}_{A_i})_{1\le i\le r}$ by $(p_{i_{A_1},i_{A_2},\ldots,i_{A_r},+})$. We regard $p$ as a point in the probability simplex $\overline\Delta_{d_1 d_2\cdots d_n-1}\subset\mathbb P^{d_1 d_2 \cdots d_n-1}$. 
We finally consider an undirected graph $G$ on $[n]$, and each random variable $\mathcal{X}_k$ is associated to the corresponding vertex $k$ of $G$.
Conditional independence typically occurs when statistics try to derive different data samples from the same (hidden) parameters we want to know more about, because there is a certain causal inference playing a role.

\begin{definition}\label{Definition 1.2}
 For three pairwise disjoint subsets $A,B,C$ of $[n]$, we say that \textit{$\mathcal{X}_A$ is conditionally independent of $\mathcal{X}_B$ given $\mathcal{X}_C$} and we write $\mathcal{X}_A\ci \mathcal{X}_B\mid \mathcal{X}_C$ when $\forall(i_A,i_B,i_C)\in\mathcal R_A \times\mathcal R_B\times\mathcal R_C$,
\begin{align*}
   \Pr(\mathcal{X}_A=i_A,\mathcal{X}_B=i_B\mid \mathcal{X}_C=i_C)=\Pr(\mathcal{X}_A=i_A\mid \mathcal{X}_C=i_C)\cdot \Pr(\mathcal{X}_B=i_B\mid \mathcal{X}_C=i_C).
\end{align*}
The statement $\mathcal{X}_A\ci \mathcal{X}_B\mid \mathcal{X}_C$ is called a \textit{conditional independence statement}.   
\end{definition}

The discrete random variables are players in our context. For a subset of vertices representing players $A \subseteq [n]$, we let $\mathcal{R}_A : = \prod_{a \in A} [d_a]$ to be the set of pure strategy profiles for $A$. By \cite[Proposition $4.1.6$]{sullivant2018algebraic}, the CI statement $\mathcal{X}_A \ci \mathcal{X}_{B} \ | \ \mathcal{X}_{C}$ holds if and only if  \begin{equation}\label{eq: independence ideal} p_{i_A i_B i_C +} p_{j_A j_B i_C +} - p_{i_A j_B i_C+} p_{j_A i_B i_C +} = 0 \end{equation} for all $i_A, j_A \in \mathcal{R}_A$, $i_B, j_B \in \mathcal{R}_B$ and $i_C, j_C \in \mathcal{R}_C$. This means the set of CI statements translates into a system of homogeneous quadratic polynomial equations in the entries of the joint probability distribution. The complex projective variety corresponding to these equations is called the \textit{discrete conditional independence (CI) model}. For an undirected graph $G$ we define the \textit{global Markov statements} as the ones of the following form: 
$\mathcal{X}_A \ci \mathcal{X}_B\mid \mathcal{X}_C$ when $A,B,C$ are disjoint subsets of $[n]$ such that ``$C$ separates $A$ and $B$", which means that all paths from $A$ to $B$ in $G$ go through $C$. We denote the set of all global Markov statements by $\mathrm{global}(G)$. We define $\mathcal{M}_{\mathrm{global}(G)} \subseteq \mathbb P^{d_1 \cdots d_n -1}$ to be the projective variety defined by all probability distributions satisfying the equations (\ref{eq: independence ideal}) for all CI statements in $\mathrm{global}(G)$.
Aside from $\mathrm{global}(G)$, there are other graphical models arising from other CI statements. Namely, the pairwise and local Markov properties. However, since we restrict our attention to strictly positive joint probability distributions i.e.\ $p \in \Delta_{d_1 \cdots d_n -1}$, these three graphical models coincide.
Another equivalent definition by the Hammersley-Clifford theorem is the parametrized undirected graphical model.

\begin{definition}\label{def: parametrized CI model}
   Let $\mathcal C(G)$ be the set of all maximal cliques of the graph $G$. A point $p\in\mathbb P^{d_1 d_2 \cdots d_n-1}$ is said to \textit{factorize according to $G$} if it can be decomposed: $$(p_{j_1 j_2\cdots j_n})_{j_1,j_2,\ldots,j_n}=\bigg(\prod_{C\in \mathcal C(G)}\sigma^{(C)}_{j_C}\bigg)_{j_1,j_2,\ldots,j_n}$$ for some $\sigma^{(C)}\in\mathbb C^{\mathcal R_C}$ for $C\in \mathcal C(G)$. The \textit{parametrized undirected graphical model} $\mathcal M_G$ is the projective toric variety in $\mathbb P^{d_1 d_2 \cdots d_n-1}$ with respect to this monomial parametrization. 
\end{definition}
The intersection $\mathcal M_G\cap\Delta_{d_1 \cdots d_n -1}$ is also called the hierarchical log-linear model associated to the simplicial complex of cliques in $G$. We investigate the degree of these varieties for certain classes of graphs in Appendix~\ref{appendix: degree CI}.

\subsection{Nash and dependency equilibria}

The notion of dependency equilibria was introduced by the philosopher Wolfgang Spohn in \cite{spohn2003dependency} and in \cite{spohn2007dependency} with several detailed examples. These are a new type of game theory equilibria generalizing the (totally) mixed Nash equilibria by adding some dependencies for the choices of the players in another way than Aumann's correlated equilibria do \cite{aumann1974subjectivity, brandenburg2025combinatorics}.

We consider a normal-form finite $n$-player game, the strategies of the $k$-th player being $[d_k]$. Each player is given a real payoff tensor $X^{(k)}\in\mathbb R^{d_1 \times \cdots \times d_n}$ rewarding them according to their strategic choice and the ones of the other players. They may not have deterministic (also called pure) strategies, and in fact, we will focus on mixed strategies, i.e., probabilistic choices of options in $[d_k]$ for each player. In our setting, the players do not act independently: the set of mixed strategies is modeled by a joint distribution $p \in \overline\Delta_{d_1 d_2\cdots d_n-1}$. Its marginals represent the subjective strategy of each player, but there may be some correlation between them. The expected payoff of the $k$-th player is given by
\[
p \cdot X^{(k)} = 
\sum_{j_1 = 1}^{d_1} \cdots \sum_{j_n = 1}^{d_n}
X^{(k)}_{j_1 \dots j_n}\, p_{j_1 \dots j_n}.
\]

\noindent A mixed strategy $p \in \overline{\Delta}_{d_1 \cdots d_n -1}$ is called a \textit{Nash equilibrium} when no player can increase their expected payoff by deviating from their own mixed strategy, assuming that the other players' mixed strategies are fixed. A Nash equilibrium exists in every finite game \cite{nash1950equilibrium}, making it the most extensively studied equilibrium concept in game theory. Despite its foundational importance, the Nash equilibrium has notable limitations and does not fully capture all aspects of strategic interaction. In particular, it assumes that players act independently, which in real-life situations is often not the case; it may induce errors of modelling, hence erratic predictions.
The prisoner's dilemma is probably the most famous example of the paradoxical social situations coming from Nash's game-theoretic perspective. In this example, both concepts of Nash equilibrium and correlated equilibrium (see \cite[Section 3]{spohn2007dependency}) fail.

\begin{example}\label{Example $1.16$}
    Two suspects committed a crime together and were arrested by the police. However, the police have no evidence against them and need their confessions. Hence, the prisoners, that are held in two separate cells, are given two options: either keep silent (cooperate), or betray their accomplice by denouncing them (defect). They know that if they both keep silent, they will be free. Otherwise, the culprits will share a fine of $2000$ euros. But in order to encourage the prisoners to speak, the police promise to reward any non-guilty prisoner speaking with $1000$ euros. Labelling ``cooperate"  with $1$ and ``defect" with $2$, we present the payoff matrices as follows. 
\[
\begin{array}{ccc}
X^{(1)}=\begin{pmatrix}
    0&-2\\1&-1
\end{pmatrix} & \text{ and } &X^{(2)}=\begin{pmatrix}
    0&1\\-2&-1
\end{pmatrix}
\end{array}
\]
\end{example}

Whatever the other player does, the reward is better if one betrays. From the point of view of a lonely player, forgetting about the existence of the other player, it seems better to betray, although both cooperating would be a better option for both if they could trust each other. There is only one Nash equilibrium where both defect. In social situations, however, trust can lead to a better option, where both cooperate. Hence, the independence of the players is not the only problematic assumption in Nash's modelling. It is also a problem that players optimize the expected payoff assuming (wrongly) that the others' strategies are fixed, whereas they will also change at the same time in real social situations. This enables players to obey a different form of (collective) rationality than the one that is assumed in Nash equilibria. Spohn introduced dependency equilibria to address this issue: \emph{cooperate-cooperate must remain a rational possibility \cite{spohn2003dependency}}.

Dependency equilibria modify the two assumptions from the definition of Nash equilibria:
players do not act independently anymore, and they maximize their \textit{conditional} expected payoff. For $p \in \Delta_{d_1 \cdots d_n -1}$, the conditional expected payoff is defined for each player $k\in[n]$ and each strategy $i\in[d_k]$ by:
\[
\mathbb E_p\big(X^{(k)}\mid  k\text{ plays }i\big)=\sum_{j\in[d_1]\times[d_2]\times\cdots\times[d_n]}\text{Pr}(\text{``all play }j\text{"}\mid k\text{ plays }i)\cdot X^{(k)}_{j_1\cdots j_n}=\frac{\big(p\cdot_k X^{(k)}\big)_i}{p_{++\cdots+i+\cdots+}},
\]
where $ p\cdot_kX^{(k)}$ is the vector
\[
\displaystyle p\cdot_kX^{(k)}=\left(   \sum_{j_1\in [d_1]}\cdots \widehat{\sum_{j_k\in[d_k]}}\cdots \sum_{j_n\in [d_n]}  X^{(k)}_{j_1\cdots i\cdots j_n}p_{j_1\cdots i\cdots j_n} \right)_{i\in[d_k]}.
\]
Here, the $\widehat{ ... }$ corresponds to the only sum or term removed in the sequence of sums.

\begin{definition}\label{Definition $1.17$}
A probability distribution $p\in \Delta_{d_1 \cdots d_n -1}$ is a \textit{dependency equilibrium} when each player maximizes their conditional expected payoff with whatever they do with positive probability according to $p$. In other words, $p$ is a dependency equilibrium if and only if for each player $k$ and each strategy $i,j\in[d_k]$ such that $k$ plays $i$ with positive probability ($p_{++\cdots+i+\cdots+}>0$), $\mathbb E_p\big(X^{(k)}\mid k\text{ plays }i\big)\ge\mathbb E_p\big(X^{(k)}\mid k\text{ plays }j\big)$.
\end{definition}

In a dependency equilibrium, for every $k\in[n]$, $\mathbb E_p\big(X^{(k)}\mid k\text{ plays }i\big)$ does not depend on $i$ (for all $i$ such that $p_{++\cdots+i+\cdots+}\ne0$), so players have no strict preference of this strategy. Each strategy with strictly positive weight from the player has to have the same conditional expected payoff, which is then also equal to the usual notion of expected payoff $p\cdot X^{(k)}$ (\cite{portakal2022geometry}, Lemma $15$). Spohn's definition has a technical flow because the notion of conditional expected payoff relies on $p_{++\cdots+j+\cdots+}\ne0$. In fact, there are several ways to extend the definition in the non-negative real numbers when probabilities vanish: for this, see \cite{portakal2024dependency}, which formalizes mathematically an idea from \cite{spohn2007dependency}. These difficulties are the reason why later in this paper, we only focus on totally mixed dependency equilibria, i.e., $p\in \Delta_{d_1 \cdots d_n -1}$.

\begin{example}\label{Example $1.18$}
 While Nash equilibria fail in enabling win-win trust in Example~\ref{Example $1.16$}, dependency equilibria do not since they are given by two components $p = \begin{pmatrix}
    p_{11} & p_{12}\\
    p_{21} & p_{22}
\end{pmatrix} \in \overline{\Delta}_3$:\begin{align*}
    &~~~\begin{pmatrix}
    \frac{t(1+t)}2 & \frac{t(1-t)}2\\
    \frac{t(1-t)}2 & \frac{(1-t)(2-t)}2
\end{pmatrix},\text{ for } 0\le t\le1 \ \ 
\text{ and }\ \ \begin{pmatrix}
    \frac{3(1-t^2)}8 & \frac{(1-t)(1-3t)}8\\
    \frac{(1+t)(1+3t)}8 & \frac{3(1-t^2)}8
\end{pmatrix},\text{ for } -\frac13\le t\le\frac13.
\end{align*}   
\end{example}

The first component contains the unique Nash equilibrium and also a dependency equilibrium where both players cooperate. Conditional expected payoff adds reflexivity in the decision, enabling, in this case, the win-win solution. Interpreting it more precisely is done in detail in \cite{spohn2003dependency}, as part of the broader philosophical field called reflexive decision theory. There is some kind of past or history known by the players, enabling them to simultaneously anticipate the others' behaviors. They both interpret themselves and the others' choices (not assumed to be fixed) and take them into account in their decision, also knowing that the others do the same. The founding paper \cite{portakal2022geometry} introduces a better mathematical framework for the study of dependency equilibria.

\begin{definition}\label{Definition $1.22$}
   The \textit{Spohn variety} $\mathcal V_X\subseteq\mathbb P^{d_1 d_2 \cdots d_n-1}$ of a game $X = (X^{(1)}, \cdots, X^{(n)})$ is defined by the vanishing set of the $2\times2$ minors of the $d_k\times2$ matrices: $$M^{(k)}=\left(\begin{matrix}
    p_{++\cdots+1+\cdots+} & &\big(p\cdot_k X^{(k)}\big)_{1}\\
    p_{++\cdots+2+\cdots+} & &\big(p\cdot_k X^{(k)}\big)_{2}\\
    . & &.\\
    .&&.\\
    .&&.\\
    p_{++\cdots+d_k+\cdots+} && \big(p\cdot_k X^{(k)}\big)_{d_k}\\
\end{matrix}\right),~\text{for all }k\in[n].$$ 
\end{definition}

\noindent In other words, the Spohn variety is the intersection of $n$ determinantal varieties (one per player), given by the maximal minors of $M^{(k)}$. By construction, the intersection of the probability simplex with the Spohn variety is the set of dependency equilibria.

\begin{theorem}[{\cite[Theorem $6$]{portakal2022geometry}, \cite[Theorem $3.18$]{portakal2024dependency}}]\label{Proposition $1.19$}\label{Proposition $1.20$}
    The Spohn variety $\mathcal V_X$ has codimension $d_1 + \cdots +d_n -n$ and degree $d_1 \cdots d_n$ for a generic game $X$. Totally mixed Nash equilibria are the dependency equilibria lying on the intersection of $\mathcal V_X$ with the Segre variety $\mathbb P^{d_1 -1} \times\cdots\times \mathbb P^{d_n -1}$ in $\Delta_{d_1 \cdots d_n -1}$. Moreover, every Nash equilibrium of $X$ lies on the Spohn variety $\mathcal V_X$.
\end{theorem}

\subsection{Conditional independence equilibria}

From the definitions of Nash and dependency equilibria, several issues naturally emerge. The Spohn variety of a generic game is high-dimensional, whereas its set of Nash equilibria is zero-dimensional. This creates a substantial dimensional gap between Nash and dependency equilibria. From a game-theoretic standpoint, the two concepts occupy opposite ends of the spectrum of player dependencies: Nash equilibria correspond to fully independent behaviors, while dependency equilibria capture complete collaboration among players. Between these extremes, however, one might encounter situations with partial dependencies where some players act collectively while others remain independent. The notion of \emph{conditional independence equilibria}, first introduced in \cite[Section 6]{portakal2022geometry}, bridges this gap in both dimension and dependency structure between Nash and dependency equilibria. Conditional independence equilibria generalize the idea behind dependency equilibria using undirected graphical models from algebraic statistics to take into account only the dependencies between the players encoded by a graph whose vertices are the players of the game. The generic case was analyzed in detail for binary games (where each player has two pure strategies) with the one-edge graphical model in \cite{portakal2024nash}, and further generalized to arbitrary graphical models in \cite{portakal2024game}.

Let $G$ be an undirected graph with $n$ vertices. We identify the vertices of $G$ with the players of the game. Moreover, we consider the players as random variables evaluating according to a joint distribution $p$. Then, the undirected graph $G$ on the players represents the conditional independence relationships between the players through the global Markov properties introduced in Section \ref{subsec CI}. We introduce the algebro-geometric model for this notion of equilibria using algebraic geometry.

\begin{definition}\label{Definition 1.25}
The \textit{Spohn conditional independence (CI) variety} $\mathcal V_{X,G}$ is defined as the variety
\[
\mathcal V_{X,G}=\overline{\mathcal V_X\cap\mathcal M_G\setminus \mathcal W}^{Zar},
\]
where $\mathcal W$ is the union of the hyperplanes
\[
\displaystyle \mathcal W:=
\bigcup_{j\in\mathcal R_{[n]}}\left\{p\in\mathbb P^{d_1 d_2 \cdots d_n-1}\mid p_j=0\right\}\cup \left\{p\in\mathbb P^{d_1 d_2 \cdots d_n-1}\mid p_{++\cdots+}=0\right\}
\]
In other words, the Spohn CI variety $\mathcal V_{X,G}$ is the variety obtained by erasing from the intersection $\mathcal V_X\cap\mathcal M_G$ all the irreducible components lying in the hyperplanes of the form $\{ p_{++\cdots+}=0\}$ or $\{p_j=0\}$ for $j\in\mathcal{R}_{[n]}$.
\end{definition}

We say that $p\in \Delta_{d_1 \cdots d_n -1}$ is a \textit{totally mixed conditional independence equilibrium} if $p$ is contained in the Spohn CI variety $\mathcal{V}_{X,G}$.
For most graphs, the intersection $\mathcal V_X\cap\mathcal M_G$ is not irreducible and has irreducible components contained in $\mathcal W$. Since we focus on totally mixed equilibria, we want to avoid components of our varieties contained in the hyperplanes $\{p_j=0\}$. In addition, motivated by Definition \ref{Definition $1.17$}, we also aim to avoid components lying on the hyperplane $\{ p_{++\cdots+}=0\}$. Therefore, erasing these extra components favors and simplifies the geometry of Spohn CI varieties while agreeing with the game-theoretic motivation behind these concepts. However,
the saturation by $\mathcal W$ becomes computationally heavy quite rapidly. 
Alternatively, we show in Theorem~\ref{theo:saturation} that for generic games one can saturate with hyperplanes of the form $\{p_{+\cdots+a+\cdots +}=0\}$. The following example shows how CI equilibria generalize Nash and dependency equilibria.

\begin{example}\label{ex: first CI equilibria}
For the complete graph, $\mathcal{M}_G=\PP^{d_1\cdots d_n-1}$. The Spohn CI variety and the Spohn variety coincide up to some possible irreducible components contained in the hyperplanes $\mathcal W$ that only appear in nongeneric games (for the saturation with the hyperplanes in Theorem~\ref{theo:saturation}, see e.g.\ \cite[Example 3.11]{portakal2024dependency}). Therefore, the set of totally mixed CI equilibria coincides with the set of totally mixed dependency equilibria for generic games. For the graph with no edge, $\mathcal{M}_G$ is the Segre variety $\PP^{d_1-1}\times \cdots \times \PP^{d_n-1}$. By Theorem~\ref{Proposition $1.20$}, the set of totally mixed CI equilibria coincides with the set of totally mixed Nash equilibria. 
\end{example}

In the set of graphs with $n$ vertices, we have a partial order given by the inclusion of the edge sets. The graph with no edge and the complete graph are placed at the two opposite extremes of this partial order. Given two graphs with $n$ vertices $G$ and $G'$ with $G\subseteq G'$, then, 
$\mathcal{V}_{X,G}\subseteq \mathcal{V}_{X,G'}$ and the same inclusion happens for the set of totally mixed CI equilibria of each graph (see \cite[Section 4]{portakal2024game}). This induces a filtration among the sets of totally mixed CI equilibria (and Spohn CI varieties) associated with each graph, with the totally mixed Nash and dependency equilibria occupying the opposite extremes respectively.

In Sections \ref{sec:nash ci} and \ref{sec:Emptyness}, we focus on cluster graphs. That is a graph whose connected components are complete. For this type of graphs, the Spohn CI variety is called \textit{Nash conditional independence variety}. In this case, the CI equilibria is called \textit{Nash CI equilibria}.
Nash CI varieties and equilibria were studied in \cite[Section 5]{portakal2024game} for binary games.

\begin{remark}
    The definition of CI equilibria in $\overline\Delta$ rather than $\Delta$ would need some more work. If we want to intersect $\mathcal V_{X,G}$ with $\overline\Delta$ rather than $\Delta$, the different sets of CI statements from $G$ will not be equivalent anymore. One needs to distinguish between the global, local, and pairwise Markov properties, and the corresponding models are not equivalent to the parametrizable model for non-chordal graphs.
\end{remark}

\subsection{Examples}\label{subsec: examples}

In order to make these concepts from algebraic game theory more concrete, this section provides explicit computations of Spohn CI varieties and CI equilibria for certain binary normal-form games. We compute some invariants of the Spohn CI varieties as their dimension and degree, using \texttt{Macaulay2} \cite{connelly2025gametheory}. All the examples yield the generic dimensions 0, 1, 2, and 4 for the 0-, 1-, 2-, and 3-edge graphs, respectively, as proven later in Section~\ref{sec:dim}. However, the degrees in these examples might differ from those of the generic Spohn CI varieties, which are 2, 8, 28, and 8, respectively.

\begin{example}[{Highly non-generic binary game}]\label{Example $1.28$}
  We consider the following $2 \times 2 \times 2$ game. We fix the payoff tensors to be 
  \[\begin{array}{cccc}
    X^{(1)}_{i_1i_2i_3}=\left\{\begin{array}{ll}
    1 & \text{ if } i_1=i_2=i_3,\\
    0 & \text{ else,}
    \end{array}
  \right.  &X^{(2)}=\pm X^{(1)} &\text{ and }& X^{(3)}=\pm X^{(1)}.
    \end{array}
  \]
  Here, some players want all of them to make the same choice (the ones with payoff tensor $+X^{(1)}$), and the others do not (the ones with payoff tensor $-X^{(1)}$). Because of the symmetry, it can be checked that either choice $\pm$ does not change the equations defining $\mathcal V_X$. Again by symmetry, the graphs we choose for the distinct CI equilibria do not depend on the labeling of the vertices. Therefore, the distinct Spohn CI varieties only depend (up to labeling) on the number of edges of the graph. It follows that $\emptyset\subsetneq\mathcal V_{X,G_0}\subsetneq\mathcal V_{X,G_1}\subsetneq\mathcal V_{X,G_2}\subsetneq\mathcal V_{X,G_3}\subsetneq\mathbb P^7$.\\
  
\emph{$G_0$ with no edges.} The players act purely independently from each other. The Spohn CI variety $\mathcal V_{X,G_0}$ has dimension $0$ and degree $2$, and hence, consists of two distinct points. However, after intersecting with $\Delta$ and renormalizing to $1$, only one of these two points is contained in $\Delta$. Therefore, 
      there is only one totally mixed Nash equilibrium: $\frac12$ for each player on each option, which means $p_{i_1 i_2 i_3}=\frac18$ for all $i_1,i_2,i_3\in[2]$. \\ 
\indent \emph{$G_1$ with one edge.} The Spohn CI variety $\mathcal V_{X,G_1}$ has now dimension $1$ and degree $2$. If the first player is the isolated vertex, the set of CI equilibria is described as the set of probabilities $p$ with
      \[
        \left\{\begin{array}{l}
       \displaystyle p_{i,j,j} = \frac{t}{2},\\ \displaystyle
        p_{i,j,3-j} = \frac{1-2t}{4},
        \end{array}
        \right. \text{ for } i,j\in[2] \text{ and } 0<t<\frac{1}{2}.
      \]
For $t=\frac14$ we recover the totally mixed Nash equilibrium ($0$-edge graph).
Unlike totally mixed Nash equilibria, the CI equilibria are now characterized by the fact that players $2$ and $3$ coordinate their way to behave randomly through the parameter $t$.\\
\indent \emph{$G_2$ with two edges.} The Spohn CI variety $\mathcal V_{X,G_2}$ has dimension $2$ and degree $6$. Assuming that players $1$ and $3$ are the players not connected by an edge, the CI equilibria are given by: $$\left\{\begin{array}{l}
p_{i,i,i}=\frac12-u_1-u_2-u_3,\\
p_{3-i,i,i}=u_1,\\
p_{i,3-i,i}=u_2,\\
p_{i,i,3-i}=u_3,
\end{array}\right.\text{ for } i\in[2],\text{ where }\left\{\begin{array}{ll}
    u_1,u_2,u_3>0,\\ u_1+u_2+u_3<\frac12,\\ (u_2+u_1)(u_2+u_3)=\dfrac{u_2}2.
\end{array}\right.$$
To recover the case of the $1$-edge graph between players $2$ and $3$, take the parameters $u_1=\frac t2$ and $u_2=u_3=\frac{1-2t}4$ for $0<t<\frac12$.\\
\indent \emph{$G_3$ the complete graph.} We have that $\dim(\mathcal V_{X,G_3})=4$ and $\deg(\mathcal V_{X,G_3})=2$. In fact, we have that $\mathcal V_{X,G_3}=\overline{\mathcal V_X\setminus\mathcal W}^{Zar}$. The CI equilibria are the probabilities $p$ such that:
$$\forall i\in[2],~\left\{\begin{array}{l}
\vspace*{1mm}p_{i,i,i}=\big(\frac12-u_1-u_2-u_3\big)\cdot\big(1+\frac{\delta(i-\frac32)}{u_1+u_2+u_3}\big),\\
\vspace*{1mm}p_{3-i,i,i}=u_1+\delta\big(i-\frac32\big),\\
\vspace*{1mm}p_{i,3-i,i}=u_2+\delta\big(i-\frac32\big),\\
p_{i,i,3-i}=u_3+\delta\big(i-\frac32\big),
\end{array}\right.$$ with parameters $u_1,u_2,u_3,\delta$ such that $$\left\{\begin{array}{l}
    u_1+u_2+u_3<\frac12,\\
    u_1,u_2,u_3>\frac{|\delta|}2.
\end{array}\right.$$
We recover the case of the $2$-edge graph by choosing $\delta=0$ and $u_1,u_2,u_3$ verifying $(u_2+u_1)(u_2+u_3)=\frac12\cdot u_2$.
      
  \end{example}

\begin{example}[Discrete El Farol Bar problem]\label{Example $1.29$}

In this game, $n$ people living in the same village are given two choices to have fun tonight: they can either stay at home (strategy 1) or go to the village's bar (strategy 2). If more than
$a \%$ of the population of $n$ players go to the bar, then the bar is overcrowded: the people in the bar have less fun than if they had stayed home. Otherwise, they have more fun going to the bar than staying home.
The payoff tensor of player $k\in[n]$ is given for all $i_1,i_2,\ldots,i_n\in[2]$ by $$X^{(k)}_{i_1i_2\cdots i_n}=\left\{\begin{array}{l}
      ~0~\text{ if }i_k=1,\\
      ~\displaystyle1~\text{ if }i_k=2\text{ and }\sum_{l=1}^n(i_l-1)\le a\cdot n,\\
      \displaystyle-1\text{ if }i_k=2\text{ and }\sum_{l=1}^n(i_l-1)>a\cdot n.
\end{array}\right.$$

\noindent In fact, since $X_{\cdots1\cdots}^{(k)}=0$ for $1$ as $k$-th index, the equations of our varieties will be linear. We fix $n=3$ for the computations. As in Example \ref{Example $1.28$}, the payoff tensors are generic enough to get the expected dimension of $\mathcal V_{X,G}$, but the obtained degree is not the generic one. By symmetry of the payoff tensors, the Spohn CI variety is the same up to relabeling if we exchange two vertices of the graph $G$. We distinguish two cases: $a=1/2$ and $a=3/4$. \\

\noindent The value $a=1/2$ means that the bar is overcrowded as soon as two players go.\\
\indent \emph{$G_0$ with no edges.} $\dim(\mathcal V_{X,G_0})=0$ and $\deg(\mathcal V_{X,G_0})=2$. However, there is only one totally mixed Nash equilibrium, given by $p_{i_1i_2i_3}=\big(\frac1{\sqrt2}\big)^y\big(1-\frac1{\sqrt2}\big)^z$, where $y$ and $z$ are the number of $1$'s and $2$'s among the indexes $i_1,i_2,i_3$.\\
\indent \emph{$G_1$ with one edge.} $\dim(\mathcal V_{X,G_1})=1$ and $\deg(\mathcal V_{X,G_1})=3$. Assuming that the edge is between the second and the third player, the totally mixed CI equilibria are given by: $$\left\{\begin{array}{l}
   \vspace*{1mm}  p_{111}=\frac1{8y}(1-2y)\text{ and }p_{122}=\frac1{8y}(1-2y)(1-4y),\\
   \vspace*{1mm}  p_{112}=p_{121}=\frac14(1-2y),\\
  \vspace*{1mm}   p_{212}=p_{221}=\frac14(6y-1),\\
      p_{211}=\frac1{8y}(6y-1)\text{ and }p_{222}=\frac1{8y}(6y-1)(1-4y),
\end{array}\right.\text{ where }\frac16<y<\frac14.$$

For $y \rightarrow \frac14$, we get that $p_{i22} \rightarrow 0$, which means that players $2$ and $3$ agree not to go together to the bar since they know that both going would be worse for both of them. Player $1$ plays at random (probability $\frac12$ of coming to the bar); the existence of such a player justifies why players $2$ and $3$ do both not come to the bar with a high probability, namely $\frac12$. But when they do, ``$2$ goes and $3$ does not" and ``$3$ goes and $2$ does not" have the same probability to occur, which is $\frac14$.\\
\indent \emph{$G_2$ with two edges.} We have that $\dim(\mathcal V_{X,G_2})=2$, $\deg(\mathcal V_{X,G_2})=4$, and there is no direct  explicit formula for the equilibrium points $p>0$ in this case. \\ 
\indent \emph{$G_3$ complete graph.} We have that $\dim(\mathcal V_{X,G_3})=4$ and $\deg(\mathcal V_{X,G_3})=1$. The equilibria $p>0$ are given by

$$\left\{\begin{array}{l}
 \vspace*{1mm}     p_{222}=x\text{ and }p_{111}=1+\dfrac x2-\dfrac32(u_1+u_2+u_3),\\
 \vspace*{1mm}     p_{211}=u_1\text{ and }p_{122}=\dfrac{u_2+u_3-u_1-x}2,\\
\vspace*{1mm}     p_{121}=u_2\text{ and }p_{212}=\dfrac{u_1+u_3-u_2-x}2,\\
      p_{112}=u_3\text{ and }p_{221}=\dfrac{u_1+u_2-u_3-x}2,
\end{array}\right.$$
        where
        parametrized by $$\left\{\begin{array}{l}
 \vspace*{1mm}     u_1,u_2,u_3,x>0, \\
\vspace*{1mm}     u_i+x<u_j+u_k\,\,\text{ for distinct }i,j,k,\\
      u_1+u_2+u_3<\frac13(2+x).
\end{array}\right.$$

\noindent The value $a=3/4$ means that the bar is overcrowded only if all the players go. \\
\indent \emph{$G_0$ with no edges.} We get that  $\dim(\mathcal V_{X,G_0})=0$ and $\deg(\mathcal V_{X,G_0})=2$. However, there is only one totally mixed Nash equilibrium, given by $$p_{i_1i_2i_3}=\big(\frac1{\sqrt2}\big)^y\big(1-\frac1{\sqrt2}\big)^z,$$ where $y$ and $z$ are the number of indexes among $i_1,i_2,i_3$ that are equal to $2$ and $1$ respectively.\\
\indent \emph{$G_1$ with one edge.} We get that  $\dim(\mathcal V_{X,G_1})=1$ and $\deg(\mathcal V_{X,G_1})=2$. Assuming that the only edge connects the second and third player, the Nash CI equilibria $p>0$ are given by 
    $$\left\{\begin{array}{l}
    \vspace*{1mm}
     p_{i11}=\big(\frac12-2x\big)\cdot\big(\frac12+(i-\frac32)x\big),\\
     \vspace*{1mm}
      p_{i12}=p_{i21}=x\cdot\big(\frac12+(i-\frac32)x\big),\\
      p_{i22}=\frac12\big(\frac12+(i-\frac32)x\big),
\end{array}\right.\text{ for } i\in[2]\text{ and }0<x<\frac14.$$

It can be interpreted as follows. Because of the symmetry of the game, every player has the same partial probability with a higher probability of going to the bar, since it remains better in most cases. However, players $2$ and $3$ can anticipate, thus better coordinate, whether they go to the bar or not. Hence, if player $1$ is more likely to go to the bar (high value of $x$), then players $2$ and $3$ not going at the same time to the bar will become a more likely event: ``only $2$ comes" or ``only $3$ comes" have an equal probability $x$. Player $1$, who can not coordinate with the others, keeps the same expected payoff (which is $0$) as in the Nash equilibrium (which is $x=\frac1{\sqrt2}-\frac12$) whereas players $2$ and $3$ get better expected payoffs $\frac x2$. For $x\rightarrow\frac14$, at least one player among $2$ and $3$ comes to the bar: this looks more rational than the Nash equilibrium since it avoids the situation where none of them benefits from going to the bar.
\\
\indent \emph{$G_2$ with two edges.} We get that $\dim(\mathcal V_{X,G_2})=2$ and $\deg(\mathcal V_{X,G_2})=4$. In this case, the formula for the equilibrium points $p>0$ is computable but hard to interpret.\\
\indent \emph{$G_3$ complete graph.} We get that  $\dim(\mathcal V_{X,G_3})=4$ and $\deg(\mathcal V_{X,G_3})=1$. The CI equilibria $p>0$ are given by
    \begin{align*}
    &\left\{\begin{array}{l}
         p_{222}=x\text{ and }p_{111}=1+u_1+u_2+u_3-4x ,\\
          p_{122}=u_1\text{ and }p_{211}=x-u_2-u_3,\\
          p_{212}=u_2\text{ and }p_{121}=x-u_1-u_3,\\
          p_{221}=u_3\text{ and }p_{112}=x-u_1-u_2,
    \end{array}\right.\\
    \text{where }&\left\{\begin{array}{l}
          x,u_1,u_2,u_3>0, \\
          u_i+u_j<x<\frac14(1+u_1+u_2+u_3)\text{ for }i\ne j.\\
    \end{array}\right.
\end{align*}
 
\end{example}

\begin{example}[Variant of the Cournot oligopoly]\label{Example $1.30$}

Fishers $1,2,\ldots,n$ respectively fish $i_1,i_2,\ldots,i_n$ fishes early in the morning, then they want to sell them during the day in the same port. However, the
price of a fish decreases when the total amount of fished fishes $m=i_1+i_2+\cdots+i_n$ increases, according to the law of supply and demand. If $m$ fishes have been fished in total, the price of one fish is given by $\max\{0,~M-m\}$. We also assume that it is humanly impossible for one fisher to fish more than $K$ fishes in the morning. Hence, the payoff tensor of fisher $k$ is for all $i_1,i_2,\ldots,i_n\in[K]$: $$X^{(k)}_{i_1i_2\cdots i_n}=i_k\cdot\max\left\{0,~M-\sum_{j=1}^ni_j\right\}.$$
In this example we fix $n=3$, $K=2$ and $M=6$. The prices of the fishes and the benefits of each fisher are gathered in Table~\ref{table: Fishers}. As in the previous examples, the symmetry among the players allow us to distinguish graphs only by the number of edges.

\begin{table}[h]
\begin{tabular}{|c|c|c|c|c|}
\hline
 & \makecell{ Everybody fishes \\ one fish} &\makecell{Fisher $1$ fishes two\\ and the others one}&\makecell{ Fishers $1$ and $2$ fishes two \\ and the other one} & \makecell{Everyone fishes\\ two fishes}\\ \hline
 \makecell{Price \\of a fish} & 3&2&1&0\\ \hline
 \makecell{Gains of \\ fishers $1$, $2$, $3$}& $3,3,3$& $4,2,2$ & $2,2,1$& $0,0,0$ \\ \hline 
 \makecell{Collective \\ gain }& 9 & 8 & 5 & 0 \\ \hline
 
\end{tabular}
\caption{Table with the prices and gains of the fisher in the game "Variant of the Cournot Oligopoly" in Example \ref{Example $1.30$}.}
\label{table: Fishers}
\end{table}

\emph{$G_0$ with no edges.} $\dim(\mathcal V_{X,G_0})=0$ and $\deg(\mathcal V_{X,G_0})=1$. Therefore, the Spohn CI variety consists of a point. There is indeed only one totally mixed Nash equilibrium $p=\big(\frac18\big)_{i_1,i_2,i_3}$: fishers play randomly because there is no information about each other's behavior.\\
\indent \emph{$G_1$ with one edge.} $\dim(\mathcal V_{X,G_1})=1$ and $\deg(\mathcal V_{X,G_1})=3$. Assuming that the edge connects players $2$ and $3$, the Nash CI equilibria $p>0$ are given by
    $$\left\{\begin{array}{l}
      p_{111}=p_{122}=x(6x-1), \\
      p_{112}=p_{121}=(\frac12-x)(6x-1),\\
      p_{211}=p_{222}=x(2-6x),\\
      p_{212}=p_{221}=(1-2x)(1-3x),
\end{array}\right.\text{ for }\frac16<x<\frac13.$$ \\
\noindent For two and three edge graphs there is no easy representation of the set of totally mixed CI equilibria. For a $2$--edge graph $G_2$, we have that $\dim(\mathcal V_{X,G_2})=2$ and $\deg(\mathcal  V_{X,G_2})=28$. For a $3$--edge graph $G_3$, we have that $\dim(\mathcal V_{X,G_3})=4$ and $\deg(\mathcal V_{X,G_3})=8$.

\end{example}

\section{Dimension of generic Spohn CI varieties}\label{sec:dim}

The dimension of Spohn varieties was carried out in \cite[Theorem 6]{portakal2022geometry}, where the authors show that the Spohn variety of a generic game has codimension $d_1+\cdots+d_n-n$. The dimension of the discrete conditional independence model $\mathcal M_G$ is computed in \cite[Corollary 2.7]{hocsten2002grobner}. Therefore, we focus on the codimension of Spohn CI varieties in $\mathcal{M}_G$.
In the case of generic binary games, it was conjectured in \cite[Conjecture 24]{portakal2022geometry} that $
\mathrm{codim}\,_{\mathcal{M}_G} \mathcal{V}_{X,G} = n$.
In \cite[Proposition 4]{portakal2024nash}, this conjecture was first proven for one-edge graphs, and in \cite[Theorem 9]{portakal2024game} for any graph. However, the case of nonbinary games remained open. In this section, we resolve this gap.

Before moving to the computation of this dimension, we highlight the slightly pathological behavior of the nonbinary case. In the setting of generic binary games,  the Spohn CI variety of any graph is nonempty. However, depending on the graph, the Spohn CI variety of a generic nonbinary game might be empty. This already appears for the Nash equilibria, i.e., for no-edge graphs. 

\begin{remark}[{\cite[Theorem 2.7]{abo2025vector}}]\label{rem: nash emptyness}
    Let $G$ be the graph with no edges, then for a generic game $X$, the Spohn CI variety $\mathcal{V}_{X,G}$ is non-empty if and only if
    \[
    d_i\leq 2-n+\sum_{j\neq i} d_j \text{ for } i\in[n].
    \]
\end{remark}

\noindent This is not the only example of this behaviour. For instance, consider a $5 \times 2 \times 2$-game where $G$ is the graph connecting the second and third vertex. Then, a \texttt{GameTheory.m2} \cite{connelly2025gametheory} computation shows that there exists a \texttt{randomGame} $X$ for which the Spohn CI variety is empty. By the upper semicontinuity of the fibers, the Spohn CI variety of $G$ and a generic game $X$ will be empty\footnote{The upper semicontinuity theorem ensures that the set of $X$ such that $\mathcal V_{X,G}=\emptyset$ is open.  But the complementary (closed) set has at least codimension $1$ since it doesn't recover the whole space of $X$. Thus, the set of $X$ for which $\mathcal V_{X,G}=\emptyset$ is also dense in the Zariski topology.}. 
The problem of analyzing when is the Spohn CI variety of a generic game empty is left for Section \ref{sec:Emptyness}. In this section we focus on the computation of the dimension of generic Spohn CI varieties whenever they are nonempty.

The strategy used in \cite[Section 3]{portakal2024game} 
is to study the linear systems defined by the determinants of the matrices $M^{(i)}$, and then apply Bertini's Theorem (see Theorem \ref{theo:Bertini}). The difference with respect to the binary case is that $M^{(i)}$ might not be a square matrix, and hence, we obtain more equations coming from its $2\times 2$ minors of these matrices. To find the equations of these minors, we use the monomial parametrization of the graphical model $\mathcal{M}_G$ (Definition~\ref{def: parametrized CI model}).
\subsection{Linear systems and their base locus}\label{subsec:base locus}

Let $G$ be a graph with $n$ vertices and let $\mathcal{C}(G)$ be the set of all maximal cliques of $G$. Given a clique $C\in\mathcal{C}(G)$, we denote the state space of $C$ by $\mathcal{R}_C:=\prod_{i\in C}[d_i]$ whose cardinality is $D_C:=\prod_{i\in C}d_i$. To each maximal clique $C\in \mathcal{C}(G)$ we associate a projective space 
\[
\mathbb{P}_C:=\mathbb{P}^{D_C-1}
\]
with coordinates $\sigma^{(C)}_{j_C}$ for $j_C\in \mathcal{R}_C$. We denote the torus of $\mathbb{P}_C$ by $\mathbb{T}_C$. In other words, 
\[
\mathbb{T}_C:=(\mathbb{C}^{*})^{D_C}/\mathbb{C}^{*}\simeq (\mathbb{C}^{*})^{D_C-1}.
\]
Here, $\mathbb{C}^{*}$ acts on $(\mathbb{C}^{*})^{D_C}$ by scalar multiplication.
We also denote the coordinates of $\mathbb{T}_C$ by $\sigma^{(C)}_{j_C}$ for  $j_C\in\mathcal{R}_C$.
Using this notation, we associate to $G$ the Segre variety and the torus 
\[ \begin{array}{ccc}
\mathbb{P}_G:= \displaystyle\prod_{C\in\mathcal{C}(G)} \mathbb{P}_C & \text{ and }& \displaystyle
\mathbb{T}_G:=\prod_{C\in\mathcal{C}(G)}\mathbb{T}_C.
\end{array}
\]
Then, the monomial map parametrizing the toric variety $\mathcal{M}_G$ given in Definition~\ref{def: parametrized CI model}  is 
\begin{equation}
    \label{eq: mon map}
    \phi_G:\mathbb{T}_G\longrightarrow\mathcal{M}_G
\end{equation}

given by 
\begin{equation}\label{eq:mon map}
p_{j_1\cdots j_n}:=\prod_{C\in\mathcal{C}(G)}\sigma^{(C)}_{j_C}.
\end{equation}
Note that one can replace $\mathbb{T}_G$ by the Segre variety $\PP_G$ as domain of the monomial map \eqref{eq: mon map}.
 
\subsubsection{Intersecting the Spohn variety with the parametrized undirected graphical model} Our goal is to calculate the dimension of  $\mathcal V_{X,G}$ by pulling it back via the monomial map $\phi_G$. To understand this pullback, we first substitute \eqref{eq:mon map} in the matrices $M^{(k)}$.  We follow the notation used in \cite{portakal2024game}. 
For a vertex $k\in[n]$, we denote by $G_k$ its connected component in $G$. In particular, we have that the set $\mathcal{C}(G_k)$ of maximal cliques of $G_k$ is a subset of $\mathcal{C}(G)$.
We also consider the sets
$$\mathscr C_k=\displaystyle\prod_{s\in G_k\setminus\{k\}}[d_s] \text{ and }\mathscr D_k=\displaystyle\prod_{s\in[n]\setminus G_k}[d_s].$$
For $j\in \mathscr C_k$ and $a\in [d_k]$, we define the index $j(a)\in [d_k]\times \mathscr C_k$ as $j(a)_k=a$ and $j(a)_i=j_i$ for every $i\in G_k\setminus\{k\}$. Using this notation, we define the monomials
\begin{equation*}
\begin{array}{cl}
\displaystyle\mathfrak S_{j,a}^{(k)}:=\prod_{C\in\mathcal C(G_k)}\sigma_{j_C(a)}^{(C)} &\text{for }j\in \mathscr C_k \text{ and } a\in [d_k],\\
\displaystyle\mathfrak O_{j}^{(k)}:=\prod_{C\not\in\mathcal C(G_k)}\sigma_{j_C}^{(C)} &\text{for }j\in \mathscr D_k,
\end{array}
\end{equation*}
and the polynomial
\begin{equation}\label{eq: poly L}
L_a^{(k)}:=
\sum_{j\in\mathscr C_k}\mathfrak S_{j,a}^{(k)}.
\end{equation}
Note that using this notation we have that 
\[
\begin{array}{ccc}
p_{j_1\ldots j_{k-1}aj_{k+1}\cdots j_n}=\displaystyle\mathfrak S_{j,a}^{(k)}\, \mathfrak O_{j}^{(k)}
&\text{and}&\displaystyle
\sum_{j_1,\ldots j_{k-1},j_{k+1}\ldots,j_n}p_{j_1\ldots j_{k-1}aj_{k+1}\cdots j_n}= L_a^{(k)}\, \left(
\sum_{j'\in\mathscr D_k}\mathfrak D_{j'}^{(k)}\right).
\end{array}
\]

\noindent Therefore, evaluating the monomial map $\phi_G$ at $M^{(k)}$ we get the matrix
 \begin{equation}\label{eq:matrix 1}
\begin{pmatrix}
\displaystyle
\sum_{(j,j')\in\mathscr C_k\times\mathscr D_k}\mathfrak S_{j,1}^{(k)}\,\mathfrak D_{j'}^{(k)} & \displaystyle
\sum_{(j,j')\in\mathscr C_k\times\mathscr D_k}X_{\cdots 1\cdots}^{(k)}\,\mathfrak S_{j,1}^{(k)}\,\mathfrak D_{j'}^{(k)}\\ \vdots &\vdots \\
\displaystyle
\sum_{(j,j')\in\mathscr C_k\times\mathscr D_k}\mathfrak S_{j,d_k}^{(k)}\,\mathfrak D_{j'}^{(k)} & \displaystyle
\sum_{(j,j')\in\mathscr C_k\times\mathscr D_k}X_{\cdots d_k\cdots}^{(k)}\,\mathfrak S_{j,d_k}^{(k)}\,\mathfrak D_{j'}^{(k)}
    
\end{pmatrix}
 \end{equation}
 where $X^{(k)}_{\cdots a\cdots}$ has $k$-th index $a\in[d_k]$ and other indices match with the ones of $j$ and $j'$.  
\subsubsection{Removing components with respect to the saturation} Note that the $a$--th entry of the first column of \eqref{eq:matrix 1} factors as the product
 \begin{equation}\label{eq:extra factor}
 \displaystyle 
 L_a^{(k)}
 \left(
\sum_{j'\in\mathscr D_k}\mathfrak D_{j'}^{(k)}\right).
 \end{equation}
In particular, the second factor appears in each entry of the first column of \eqref{eq:matrix 1}, and hence, it will appear as a factor in each minor of \eqref{eq:matrix 1}. Since the equation of the hyperplane  $\{p_{+\cdots+}=0\}$ factors as
\[
\displaystyle p_{+\cdots+}=\left( \sum_{a\in[d_k]} L_a^{(k)}\right)\left(\sum_{j\in\mathscr{D}_k}\mathfrak{D}_j^{(k)}\right),
\]
the second factor of \eqref{eq:extra factor} leads to irreducible components  contained in the hyperplane of the form $\{p_{+\cdots+}=0\}$. Therefore,  this factor can be removed, and we obtain the matrix 
 \begin{equation}\label{eq:matrix 2}
\widetilde{M}^{(k)}=\begin{pmatrix}
\displaystyle
L_1^{(k)}& \displaystyle
\sum_{(j,j')\in\mathscr C_k\times\mathscr D_k}X_{\cdots 1\cdots}^{(k)}\,\mathfrak S_{j,1}^{(k)}\,\mathfrak D_{j'}^{(k)}\\ \vdots &\vdots \\
\displaystyle
L_{d_k}^{(k)}& \displaystyle
\sum_{(j,j')\in\mathscr C_k\times\mathscr D_k}X_{\cdots d_k\cdots}^{(k)}\,\mathfrak S_{j,d_k}^{(k)}\,\mathfrak D_{j'}^{(k)}
\end{pmatrix}.
 \end{equation}
 In the case where $k$ is an isolated vertex, we can do further simplifications on this matrix. In this case, the monomial $\mathfrak S_{j,a}^{(k)}$ equals $\sigma^{(k)}_a$ for $a\in [d_k]$ and the $a$--th row of the matrix is 
 \[
\begin{pmatrix}
    \displaystyle
\sigma^{(k)}_a& \displaystyle
\sum_{j'\in\mathscr D_k}X_{\cdots a\cdots}^{(k)}\,\sigma^{(k)}_a\,\mathfrak D_{j'}^{(k)}
\end{pmatrix}.
 \]
 We note that $\sigma^{(k)}_a$ is a common factor in the row. As before, since this factor leads to components lying on the hyperplane of the form $\{p_{j_1\cdots j_{k-1} a j_{k+1}\cdots j_n}=0\}$, we may remove this factor. Therefore, for an isolated vertex $k$, we define the matrix $\widetilde{M}^{(k)}$ as the matrix
 \begin{equation}\label{eq:matrix 3}
\widetilde{M}^{(k)}=\begin{pmatrix}
\displaystyle
1& \displaystyle
\sum_{j'\in\mathscr D_k}X_{\cdots 1\cdots}^{(k)}\,\mathfrak D_{j'}^{(k)}\\ \vdots &\vdots \\
\displaystyle
1& \displaystyle
\sum_{j'\in\mathscr D_k}X_{\cdots d_k\cdots}^{(k)}\,\mathfrak D_{j'}^{(k)}
\end{pmatrix}.
 \end{equation}

\noindent For a player $k$ and two distinct strategies $a,b\in[d_k]$, we define the polynomial $F_{a,b}^{(k)}$ as the $2\times 2$ minor of $\widetilde{M}^{(k)}$ given by the $a$--th and $b$--th rows. In other words, if $k$ is not an isolated vertex, we have
\[
F_{a,b}^{(k)}=\left|\begin{array}{cc}
    \displaystyle L_a^{(k)}&\displaystyle\sum_{(j,j')\in\mathscr C_k\times\mathscr D_k}X_{\cdots a\cdots}^{(k)}\mathfrak S_{j,a}^{(k)}\,\mathfrak D_{j'}^{(k)}\\
    &\\
    \displaystyle L_b^{(k)}&\displaystyle\sum_{(j,j')\in\mathscr C_k\times\mathscr D_k}X_{\cdots b\cdots}^{(k)}\mathfrak S_{j,b}^{(k)}\,\mathfrak D_{j'}^{(k)}
\end{array}\right|.
\]
If $k$ is an isolated vertex, we have the following form: 
\[
F_{a,b}^{(k)}=\left|\begin{array}{cc}
   1~~~~~~&\displaystyle\sum_{j'\in\mathscr D_k}X_{\cdots a\cdots}^{(k)}\,\mathfrak D_{j'}^{(k)}\\
    &\\
   1~~~~~~&\displaystyle\sum_{j'\in\mathscr D_k}X_{\cdots b\cdots}^{(k)}\,\mathfrak D_{j'}^{(k)}
\end{array}\right|
\]
In particular, $F_{a,b}^{(k)}$ is a multihomogeneous polynomial in the variables $\sigma_j^{(C)}$ for $C\in \mathcal{C}(G)$, i.e.\ in the Segre variety $\mathbb{P}_G$. Let $e_C$ be the standard vector. The multidegree of $F_{a,b}^{(k)}$  is 

\[u_k:=\left\{\begin{array}{cl}
\displaystyle\sum_{C\not\in \mathcal{C}(G_k)} e_C + 2\sum_{C\in \mathcal{C}(G_k)} e_C & \text{ if }k\text{ is not an isolated vertex,}\\ & \\
\displaystyle\sum_{C\not\in \mathcal{C}(G_k)} e_C = (1,\ldots,1,\underset{(k)}{0},1,\ldots,1)&
\text{ if }k\text{ is an isolated vertex.}
\end{array}\right.\] 

\noindent For a game $X$, we define the variety $Y_{X,G}$ as the subvariety of $\mathbb{T}_G$ given by the $2\times 2$ minors of the matrices $\widetilde{M}^{(1)},\ldots,\widetilde{M}^{(n)}$. That is $Y_{X,G}$ is the variety defined by the polynomials $F^{(k)}_{a,b}$ for all $k\in[n]$ and $a,b\in[d_k]$.
Let $\widetilde{\mathcal{V}}_{X,G}$ be the preimage of the Spohn CI variety via the monomial map $\phi_G$. By construction, we have that $\widetilde{\mathcal{V}}_{X,G}\subseteq Y_{X,G}$. Moreover, $\widetilde{\mathcal{V}}_{X,G}$ is obtained by removing from $ Y_{X,G}$ the irreducible components contained in the pullback of hyperplanes of the form $\{p_{j_1\cdots j_n}=0\}$ and $\{p_{+\cdots +}=0\}$. In Proposition~\ref{prop:eq Nash CI}, we show that $\widetilde{\mathcal{V}}_{X,G}= Y_{X,G}$ for generic $X$. However, this equality does not hold for any graph.

\begin{example}\label{ex: wrong saturaion}
Consider a generic $d_1 \times d_2 \times d_3$ game and let $G$ be the path graph, i.e. the graph where the edge between the first and third vertex is removed from the complete graph. Since there is no isolated vertex nor separate connected component in the graph (it has only two maximal cliques), neither of the simplifications \eqref{eq:matrix 2} and \eqref{eq:matrix 3} take effect. Let $\overline{Y_{X,G}}$ be the closure of $Y_{X,G}$ in $\PP_G$. Then, we get $\phi_G(\overline{Y_{X,G}})=\mathcal M_G\cap\mathcal V_X$. Here, we see the monomial map $\phi_G$ as a map from $\PP_G$ instead of from $\mathbb{T}_G$, which makes $\phi_G$ surjective onto $\mathcal{M}_G$. If we assume $\widetilde{\mathcal{V}}_{X,G}= Y_{X,G}$, we get that 
\[
\mathcal M_G\cap\mathcal V_X=\phi_G\left(\overline{Y_{X,G}}\right)=\phi_G\left(\overline{\widetilde{\mathcal{V}}_{X,G}}\right)=\mathcal V_{X,G}.
\]
 However,
Table \ref{table:2 edge graph} lists the dimension and degree of the varieties $\mathcal V_{X,G}$  and $\mathcal M_G\cap\mathcal V_X$
for distinct values of $d_1$, $d_2$, and $d_3$, which shows that the equality does not hold in general. We conclude that in general $\widetilde{\mathcal{V}}_{X,G}\subsetneq Y_{X,G}$.

\begin{table}
\begin{tabular}{|ccc|cc|cc|}
\hline
\multicolumn{3}{|c|}{Strategies} & \multicolumn{2}{c|}{$\mathcal M_G\cap\mathcal V_X$} & \multicolumn{2}{c|}{$\mathcal V_{X,G}$} \\ \hline
$d_1$ & $d_2$ & $d_3$ & Dimension & Degree & Dimension & Degree \\ \hline
2 & 2 & 2 & 2 & 32 & 2 & 28 \\ \hline
2 & 2 & 3 & 4 & 2 & 3 & 87 \\ \hline
2 & 3 & 2 & 4 & 96 & 4 & 92 \\ \hline
2 & 2 & 4 & 6 & 2 & 4 & 196 \\ \hline
2 & 4 & 2 & 6 & 256 & 6 & 252 \\ \hline
\end{tabular}
\caption{Table with the degrees and dimensions of generic Spohn CI varieties for the path graph.}
\label{table:2 edge graph}
\end{table}
\end{example}
\subsubsection{Linear system in the torus and the equations of its base locus}\label{subsec:linsys}
In order to calculate the dimension of Spohn CI variety, we first study the dimension of the variety $Y_{X,G}$ through its equations. In the binary case done in \cite[Section 3]{portakal2024game}, the strategy followed is to analyze the linear system $\Lambda_{a,b}^{(k)}$ of all multihomogeneous polynomials $F_{a,b}^{(k)}$. For a player $k$ and $a,b\in[d_k]$ distinct, the linear system $\Lambda_{a,b}^{(k)}$ is the image the linear map 
\[
\mathbb{R}^{d_1\cdots d_n}\longrightarrow |\mathcal{O}_{\PP_G}(u_k)|
\] 
that associates to a payoff tensor $X^{(k)}$ the multihomogeneous polynomial $F^{(k)}_{a,b}$. Here, $|\mathcal{O}_{\PP_G}(u_k)|$ denotes the projective space of multihomogeneous polynomials of multidegree $u_k$ in $\PP_{G}$. In the nonbinary case, our strategy is to look at a slightly different linear system.
For a payoff tensor $X^{(k)}$ and a strategy $a\in [d_k]$, we define the $a$--th \textit{slice payoff tensor} of $X^{(k)}$ as the subtensor of $X^{(k)}$ given by $$X^{(k)}_{a}:=(X_{j_1\cdots a\cdots j_n})_{j_i\in[d_i]\text{ for } i\neq k}.$$
In other words, the entries of $X^{(k)}_{a}$ are the entries of $X^{(k)}$ for which $j_k=a$. Let $V_a^{(k)}:=\mathbb{R}^{d_1\cdots d_{k-1}d_{k+1}\cdots d_n}$ be the vector space of all these tensors. In particular, the vector space of payoff tensors of the player $k$ is isomorphic to the product $V_1^{(k)}\times \cdots \times V_{d_k}^{(k)}$.  Using this notation, we get that the polynomial $F^{(k)}_{a,b}$ only depends on the slice payoff tensors $X^{(k)}_a$ and $X^{(k)}_b$.
Now, fix a strategy $a\in[d_k]$ and the slice payoff tensor $X^{(k)}_a$. For $b\in[d_k]\setminus\{a\}$, we consider the map 
\begin{equation*}
    \label{eq:map linear system}
V_b^{(k)}\longrightarrow |\mathcal{O}_{\PP_G}(u_k)|
\end{equation*}
that sends the $b$--th slice payoff tensor $X^{(k)}_b$ to the polynomial $F_{a,b}^{(k)}$ defined by $X^{(k)}_a$ and $X^{(k)}_b$.  Note that the map \eqref{eq:map linear system} is linear but not projective. Therefore, the image is not a projective subspace, but its closure is.
We define the linear system $\Lambda_{a,b}^{(k)}(X^{(k)}_a)$ as the closure of the image of \eqref{eq:map linear system}.  Note that $\Lambda_{a,b}^{(k)}(X^{(k)}_a)$ depends on the slice payoff tensor $X^{(k)}_a$. Moreover, $\Lambda_{a,b}^{(k)}(X^{(k)}_a)$ is a linear subspace of $\Lambda_{a,b}^{(k)}$ and we have that
\[
\Lambda_{a,b}^{(k)}=
\bigcup_{X^{(k)}_a\in V_a^{(k)}} \Lambda_{a,b}^{(k)}(X^{(k)}_a).
\]

\noindent For a slice payoff tensor $X_a^{(k)}$, we denote the base locus of the linear system $\Lambda_{a,b}^{(k)}(X^{(k)}_a)$ by $B_{a,b}^{(k)}(X^{(k)}_a)$. 
Next, we calculate the equations of this base locus. First of all, note that if $k$ is an isolated vertex, the linear system $\Lambda_{a,b}^{(k)}(X^{(k)}_a)$ is complete, and hence, $B_{a,b}^{(k)}(X^{(k)}_a)$ is empty. Assume now that $k$ is not an isolated vertex and consider the linear system $W_{a,b}^{(k)}(X^{(k)}_a)$ of polynomials  of the form 
\begin{equation}\label{eq:det linear sys}
\left|\begin{matrix}
    \displaystyle\sum_{j\in\mathscr C_k}L_a^{(k)}&\displaystyle\sum_{j\in\mathscr C_k}X_{\cdots a\cdots}^{(k)}\mathfrak S_{j,a}^{(k)}\\
    &\\
    \displaystyle\sum_{j\in\mathscr C_k}L_b^{(k)}  &\displaystyle\sum_{j\in\mathscr C_k}X_{\cdots b\cdots}^{(k)}\mathfrak S_{j,b}^{(k)}
\end{matrix}\right|
\end{equation}
for any slice payoff tensor $X^{(k)}_b$.
As in \cite[Section 3]{portakal2024game}, the base locus of $W_{a,b}^{(k)}(X^{(k)}_a)$ and $B_{a,b}^{(k)}(X^{(k)}_a)$ coincide. To calculate this base locus we first describe the generators of $W_{a,b}^{(k)}(X^{(k)}_a)$.

\begin{lemma}\label{lemma:gens lin system}
    For a nonisolated vertex $k$ and a strategy $a\in[d_k]$, fix the slice payoff tensor $X^{(k)}_a$. Then, the linear system $W_{a,b}^{(k)}(X^{(k)}_a)$ is generated by the polynomials

    \begin{itemize}
        \item $ \mathfrak S_{j,b}^{(k)} L_{a}^{(k)}$ for $j\in\mathscr C_k
        $,
\item $\mathcal G_a^{(k)}\,L_b^{(k)}$,
    \end{itemize}
    where
    \[
 \mathcal G_a^{(k)}= \sum_{j\in\mathscr C_k} 
(X^{(k)}_{1\cdots a\cdots 1}-X^{(k)}_{j_1\cdots a\cdots j_n})
\mathfrak S_{i,a}^{(k)}.
    \]
\end{lemma}
\begin{proof}
    As in the proof of \cite[Lemma 7]{portakal2024game}, we can write the determinant \eqref{eq:det linear sys} as 
    \begin{equation}\label{eq:expand det}
    \begin{array}{c}
\displaystyle  A^{(k)}_{\mathbbm{1},\mathbbm{1}}\mathfrak S_{\mathbbm{1},a}^{(k)}\mathfrak S_{\mathbbm{1},b}^{(k)}+
\sum_{i\in\mathscr C_k\setminus\{\mathbbm{1}\}} A^{(k)}_{i,\mathbbm{1}}\mathfrak S_{i,a}^{(k)}\mathfrak S_{\mathbbm{1},b}^{(k)}+
\sum_{j\in\mathscr C_k\setminus\{\mathbbm{1}\}} A^{(k)}_{\mathbbm{1},j}\mathfrak S_{\mathbbm{1},a}^{(k)}\mathfrak S_{j,b}^{(k)}+\\ \\ \displaystyle+
\sum_{i,j\in\mathscr C_k\setminus\{\mathbbm{1}\}}(A^{(k)}_{i,\mathbbm{1}}+ A^{(k)}_{\mathbbm{1},j}-A^{(k)}_{\mathbbm{1},\mathbbm{1}}  )
\mathfrak S_{i,a}^{(k)}\mathfrak S_{j,b}^{(k)}
\end{array},
    \end{equation}
    where $\mathbbm{1}=(1,\cdots 1)$ and 
    $
    A_{i,j}^{(k)}=X^{(k)}_{j_1\cdots b\cdots j_n}-X^{(k)}_{i_1\cdots a\cdots i_n}.
    $
Since the slice payoff tensor $X^{(k)}_b$ is free of choice, the coefficients of the form $A^{(k)}_{\mathbbm{1},j}$  in \eqref{eq:expand det} are not subject to any restriction. However,
since the slice payoff tensor $X^{(k)}_a$ is fixed,  all the coefficients $A^{(k)}_{i,\mathbbm{1}}=A^{(k)}_{\mathbbm{1},\mathbbm{1}}+X_{1\cdots a\cdots1}^{(k)}-X_{i_1\cdots a\cdots i_n}^{(k)}$ with $i\neq \mathbbm{1}$ are fixed. 
Let $g$ be the polynomial defined by the fixed coefficients, that is
\[
g= 
\sum_{i\in\mathscr C_k\setminus\{\mathbbm{1}\}} (A^{(k)}_{i,\mathbbm{1}}-A^{(k)}_{\mathbbm{1},\mathbbm{1}})\mathfrak S_{i,a}^{(k)}\mathfrak S_{\mathbbm{1},b}^{(k)}
+
\sum_{i,j\in\mathscr C_k\setminus\{\mathbbm{1}\}}(A^{(k)}_{i,\mathbbm{1}}-A^{(k)}_{\mathbbm{1},\mathbbm{1}})
\mathfrak S_{i,a}^{(k)}\mathfrak S_{j,b}^{(k)}=L_b^{(k)}
\,\mathcal G_a^{(k)}(X_a^{(k)})
\]
This polynomial is achieved in the linear system by setting all the free coefficients to be zero, including $A^{(k)}_{\mathbbm{1},\mathbbm{1}}$. In other words, the polynomial $g$ is obtained when $X_{j_1\cdots b\cdots j_n}^{(k)}=X_{1\cdots a\cdots 1}^{(k)}$, and isolating $g$ in \eqref{eq:expand det} is the same as replacing the $A^{(k)}_{i,\mathbbm{1}}$ by $A^{(k)}_{\mathbbm{1},\mathbbm{1}}$.
Now, we can write the polynomial \eqref{eq:expand det} as
\[
g+A^{(k)}_{\mathbbm{1},\mathbbm{1}}L^{(k)}_a\mathfrak S_{\mathbbm{1},b}^{(k)}
+\sum_{j\in\mathscr C_k\setminus\{\mathbbm{1}\}} A^{(k)}_{\mathbbm{1},j}\mathfrak S_{\mathbbm{1},a}^{(k)}\mathfrak S_{j,b}^{(k)}+
\sum_{i,j\in\mathscr C_k\setminus\{\mathbbm{1}\}}A^{(k)}_{\mathbbm{1},j}
\mathfrak S_{i,a}^{(k)}\mathfrak S_{j,b}^{(k)}=g+\sum_{j\in\mathscr C_k} A^{(k)}_{\mathbbm{1},j}\mathfrak S_{j,b}^{(k)}L^{(k)}_a
\]
where $A^{(k)}_{\mathbbm{1},j}$ are free parameters. Since $W_{a,b}^{(k)}(X^{(k)}_a)$ is a projective subspace, the polynomials of the form
\[
\sum_{j\in\mathscr C_k} A^{(k)}_{\mathbbm{1},j}\mathfrak S_{j,b}^{(k)}L^{(k)}_a
\]
are also in the linear system for any $A^{(k)}_{\mathbbm{1},j}$. By setting all coefficients to zero except $A^{(k)}_{\mathbbm{1},j}$, we get the desired polynomials.
\end{proof}

\noindent We next calculate the equations of $B_{a,b}^{(k)}(X^{(k)}_a)$.

\begin{lemma}\label{lem:base locus}
        For a nonisolated vertex $k$ and a strategy $a\in[d_k]$, fix the slice payoff tensor $X^{(k)}_a$. Then, the base locus $B_{a,b}^{(k)}(X^{(k)}_a)$ is the union
\[
\mathbb{V}(L_a^{(k)},\mathcal G_a^{(k)})\cup\mathbb{V}(L_a^{(k)},L_b^{(k)})\cup\mathbb{V}(\mathfrak S_{j,b}^{(k)}:j\in\mathscr C_k).
\]
\end{lemma}
\begin{proof}
As in \cite[Section 3]{portakal2024game}, $B_{a,b}^{(k)}(X^{(k)}_a)$ and the base locus of $W_{a,b}^{(k)}(X^{(k)}_a)$ coincide. So it is enough to calculate the base locus of $W_{a,b}^{(k)}(X^{(k)}_a)$. Using the generators of the linear system described in Lemma \ref{lemma:gens lin system}, we distinguish two cases depending on whether $L_a^{(k)}$ is zero or not. If $L_{a}^{(k)}=0$, we obtain the variety $\mathbb{V}(L_{a}^{(k)},\mathcal G_a^{(k)}L_b^{(k)})$, which decompose into the irreducible components $\mathbb{V}(L_{a}^{(k)},\mathcal G_a^{(k)})$ and $\mathbb{V}(L_{a}^{(k)},L_b^{(k)})$. If $L_{a}^{(k)}\neq 0$, we get that $\mathfrak S_{j,b}^{(k)}=0$ for every $j\in\mathscr C_k$ which leads to the component $\mathbb{V}(\mathfrak S_{j,b}^{(k)}:j\in\mathscr C_k).$
\end{proof}

\subsection{Computation of the dimension}\label{subsec:dim}

To calculate the dimension of a generic Spohn CI variety, we use the linear systems introduced in 
Subsection \ref{subsec:linsys} and Bertini's Theorem. For the sake of clarity, we recall the version of Bertini's theorem used in this paper.

\begin{theorem}[{\cite[Theorem 8.18]{hartshorne2013algebraic}}]\label{theo:Bertini}
    Let $Z$ be a $d$--dimensional complex projective variety contained in a Segre variety $\PP$ and let $W$ be a nonzero linear system of $\PP$. Then
    \begin{itemize}
        \item for a generic divisor $D\in W$, every irreducible component of $D\cap Z$ not contained in the base locus of $W$  has dimension $d-1$.
        \item if $Z$ is irreducible and $d\geq 2$, then for a generic divisor $D\in W$, $D\cap Z$ is irreducible away from the base locus of $W$.
        \item for a generic divisor $D\in W$, $D\cap Z$ is smooth away from the base locus of $W$ and the singular locus of $Z$.
    \end{itemize}
\end{theorem}

We consider the integer vector $\mathbf{a}:=(a_k)_{k\in[n]}\in \mathcal{R}_{[n]}$. In other words, the integer vector $\mathbf{a}$ consists of a choice of strategy for each player. For such vector, we consider the variety
\[
    Y_{X,G,\mathbf{a}}:=\mathbb{V}(F_{a_k,b}^{(k)}\,:\,k\in[n]\text{ and } b\in[d_k]\setminus\{a_k\}).
    \]
    In other words, $Y_{X,G,\mathbf{a}}$ is the variety defined by the $2\times 2$ minors of the matrix $\widetilde{M}^{(k)}$ containing the $a_k$-th row for each player $k$. Therefore, we have 
    \begin{equation}\label{eq:intersection}
    Y_{X,G}:= \bigcap_{\mathbf{a}\in \mathcal{R}_G} Y_{X,G,\mathbf{a}}.
    \end{equation}

\noindent Using Lemma \ref{lem:base locus} and these varieties, we derive the dimension of Spohn CI varieties.

\begin{theorem}\label{theo:dimension}
Let $G$ be a graph with $n$ vertices. Then, either $\mathcal{V}_{X,G}$ is empty for a generic game $X$, or $\mathrm{codim}_{\mathcal{M}_G}\,\mathcal{V}_{X,G} = d_1+\cdots+d_n-n$ for a generic game $X$.
\end{theorem}
\begin{proof}
Assume that  $\mathcal{V}_{X,G}$ is nonempty for a generic game $X$.
We follow a similar structure as \cite[Theorem 9]{portakal2024game}.
By \cite[Theorem 6]{portakal2022geometry}, the codimension of $\mathcal{V}_{X}$ is $d_1+\cdots+d_n-n$. Now, since
every irreducible component of the Spohn CI variety is an irreducible component of $\mathcal{M}_G\cap \mathcal{V}_{X}$, we deduce that
\[
\mathrm{codim}_{\mathcal{M}_G}\,\mathcal{V}_{X,G} \leq d_1+\cdots + d_n-n.
\]
Recall that $\widetilde{\mathcal{V}}_{X,G}$ is the pullback of the Spohn CI variety via $\phi_G$. In particular, 
\[
\mathrm{codim}_{\mathbb{T}_G}\,\widetilde{\mathcal{V}}_{X,G}\leq\mathrm{codim}_{\mathcal{M}_G}\,\mathcal{V}_{X,G} \leq d_1+\cdots + d_n-n.
\]

Now, $\widetilde{\mathcal{V}}_{X,G}$ does not have any irreducible component contained in the pullback via $\phi_G$ of the hyperplanes of the form $\{p_{j_1\ldots j_n}=0\} $ and $\{p_{+\cdots+}=0\}$. Therefore, it is enough to show that $d_1+\cdots + d_n-n\leq\mathrm{codim}_{\mathbb{T}_G}\,\widetilde{\mathcal{V}}_{X,G}$ away from these hyperplanes.  Moreover, since the varieties  $Y_{X,G}$  and $\widetilde{\mathcal{V}}_{X,G}$ coincide away from these hyperplanes, it is enough to show that 
$d_1+\cdots + d_n-n\leq\mathrm{codim}_{\mathbb{T}_G}\,Y_{X,G}$.
By Bertini's Theorem
the variety $Y_{X,G,\mathbf{a}}$ has codimension  $d_1+\cdots + d_n-n$ away from the union of the base loci
\begin{equation}\label{eq: B a}
B_{\mathbf{a}}:=\bigcup_{k\in[n]}\,\,\bigcup_{b\in[d_k]\setminus\{a_k\}}B_{a_k,b}^{(k)}(X_{a_k}^{(k)})
\end{equation}
for every $\mathbf{a}\in\mathcal{R}_{[n]}$. By \eqref{eq:intersection}, we deduce that all irreducible components of $Y_{X,G}$ away from each $B_{\mathbf{a}}$ have codimension at least $d_1+\cdots + d_n-n$; hence, considering their union, $Y_{X,G}$ has codimension at least $d_1+\cdots + d_n-n$ away from
\begin{equation}\label{eq:big intersection}
B:=\bigcap_{\mathbf{a}\in\mathcal{R}_{[n]}}B_{\mathbf{a}}.
\end{equation}
To show that $Y_{X,G}$ has the correct codimension away from the pullback of the hyperplanes of the form $\{p_{j_1\ldots j_n}=0\} $ and $\{p_{+\cdots+}=0\}$, it is enough to check that $B$ is contained in this pullback. By Lemma \ref{lem:base locus}, $B_{a,b}^{(k)}(X_a^{(k)})$ has $3$ irreducible components. In particular, the component $\mathbb{V}(\mathfrak S_{j,b}^{(k)}:j\in\mathscr C_k)$ is already contained in a hyperplane of the form $\{p_{j_1\ldots j_n}=0\} $. Therefore, we may replace 
$B_{a_k,b}^{(k)}(X_{a_k}^{(k)})$ by 
\begin{equation}\label{eq:def new base locus}
B_{a_k,b}^{(k)}(X_{a_k}^{(k)})':= \mathbb{V}(L_{a_k}^{(k)},\mathcal G_{a_k}^{(k)})\cup\mathbb{V}(L_{a_k}^{(k)},L_b^{(k)})
\end{equation}
in the definition of $B_\mathbf{a}$ and $B$. Let $p$ be a point in $B$. We claim that $p$ is contained in $\mathbb{V}(L_a^{(k)}: a\in [d_k])$ for some player $k\in [n]$. Assume on the contrary that $p\not \in \mathbb{V}(L_a^{(k)}: a\in [d_k])$ for every $k\in [n]$. Thus, for every $k\in[n]$ there exists $a_k\in [d_k]$ such that $p\not\in \mathbb{V}(L_{a_k}^{(k)})$. Consider the integer vector $\mathbf{a}=(a_k)_{k\in[n]}$. By construction, $p$ is contained in $B_\mathbf{a}$, and hence, there exist $k\in [n]$ and $b\in[d_k]\setminus\{a_k\}$ such that 
$p\in B_{a_k,b}^{(k)}(X_{a_k}^{(k)})'$. By \eqref{eq:def new base locus}, we deduce that $p\in B_{a_k,b}^{(k)}(X_{a_k}^{(k)})'\subseteq \mathbb{V}(L_{a_k}^{(k)})$, which is a contradiction. We conclude that $p$ is contained in $\mathbb{V}(L_a^{(k)}: a\in [d_k])$ for some player $k\in [n]$. 
In particular, we get that $p\in \mathbb{V}(
\sum_{a\in [d_k]}L_a^{(k)} )$. Using that the evaluation of $\phi_G$ at $p_{+\cdots+}$ factors as:
\[
\phi_G^{*}\,p_{+\cdots+} = \left( \sum_{a\in [d_k]}L_a^{(k)}\right)\left(
\sum_{j'\in\mathscr D_k}\mathfrak D_{j'}^{(k)}\right),
\]
we deduce that $p$ is contained in the pullback of the hyperplane $\{p_{+\cdots+}=0\}$. Hence, $B$ is also contained in this pullback.
\end{proof}

\noindent Similarly to \cite[Section 3]{portakal2024game}, we can derive some partial results about the smoothness of Spohn CI varieties using the arguments of the proof of Theorem \ref{theo:dimension}. 

\begin{theorem}\label{theo:smooth}
    Let $G$ be a graph with $n$ vertices. Then, for generic games the Spohn CI variety $\mathcal{V}_{X,G}$ is smooth away from the hyperplanes of the form $\{p_{j_1\ldots j_n}=0\} $ and $\{p_{+\cdots+}=0\}$. In particular, the set of totally mixed CI equilibria is either empty for generic games or a smooth manifold of codimension $d_1+\cdots +d_n-n$ in $\mathcal{M}_G$ for generic games.
\end{theorem}
\begin{proof}
    The proof follows exactly the same idea as in \cite[Theorem 12]{portakal2024game}. Let $\mathcal{V}_{X,G}$ be the Spohn CI variety of a generic game $X$, and let $\widetilde{\mathcal{V}}_{X,G}$ be the preimage of $\mathcal{V}_{X,G}$ through the monomial map \eqref{eq:mon map}. As in the proof of Theorem \ref{theo:dimension}, consider the varieties $Y_{X,G}$ and $Y_{X,G,\mathbf{a}}$ for $\mathbf{a}\in\mathcal{R}_{[n]}$. By Bertini's Theorem,  $Y_{X,G,\mathbf{a}}$ is smooth away from the base locus $B_{\mathbf{a}}$ \eqref{eq: B a}. Moreover, away from this base locus, $Y_{X,G}$ is formed by irreducible components of $Y_{X,G,\mathbf{a}}$. Therefore, $Y_{X,G}$ is smooth away from $B_{\mathbf{a}}$. By \eqref{eq:intersection}, we deduce that   $Y_{X,G}$ is smooth away from $B$ defined in \eqref{eq:big intersection}.
     As shown in the proof of Theorem \ref{theo:dimension}, $B$ is contained in the pullbacks of hyperplanes of the form $\{p_{j_1\ldots j_n}=0\} $ and $\{p_{+\cdots+}=0\}$. Therefore, $Y_{X,G}$ is smooth away from the pullbacks of hyperplanes of the form $\{p_{j_1\ldots j_n}=0\} $ and $\{p_{+\cdots+}=0\}$. Similarly, $Y_{X,G}$ and $\widetilde{\mathcal{V}}_{X,G}$ coincide away from the pullbacks of these hyperplanes. We deduce that 
    $\widetilde{\mathcal{V}}_{X,G}$  is also smooth away from these hyperplanes. Using \cite[Tag 02KM]{stacks-project}, we conclude that the Spohn CI variety is smooth away from the hyperplanes of the form $\{p_{j_1\ldots j_n}=0\} $ and $\{p_{+\cdots+}=0\}$.

    Now, one notices that the hyperplanes of the form $\{p_{j_1\ldots j_n}=0\} $ and $\{p_{+\cdots+}=0\}$ intersect the probability simplex in the boundary. Therefore, the set of totally mixed CI equilibria is either empty or it lies in the smooth locus of $\mathcal{V}_{X,G}$. In the second case, by \cite[Theorem 2.2.9]{mangolte2020real}, the dimension of the set of totally mixed CI equilibria  is the one given in Theorem \ref{theo:dimension}.
\end{proof}

\subsection{Saturation and Spohn CI varieties}\label{sec:saturation}

 One of the main difficulties of dealing with Spohn CI varieties lies in their construction: $\mathcal{V}_{X,G}$ is obtained by removing from $\mathcal{M}_G\cap\mathcal{V}_X$ the irreducible components lying in hyperplanes of the form  $\{p_{j_1\ldots j_n}=0\} $ and $\{p_{+\cdots+}=0\}$. In terms of ideals, this process of erasing components is achieved via saturations. 
 A difficulty arising from this saturation description of Spohn CI varieties is the lack of an algebraic family of Spohn CI varieties. More precisely, if we fix a graph $G$, then the union
 \begin{equation}\label{eq:big union}
\bigcup_{\text{Games } X} \mathcal{V}_{X,G}
 \end{equation}
is not an algebraic variety but a constructible set. In order to build an algebraic family of Spohn CI varieties, one would need to take the Zariski closure of \eqref{eq:big union}. This subtle distinction creates some technical difficulties while dealing with these varieties as seen in \cite[Proposition 25]{portakal2024game}.

 However, the main disadvantage of this saturation process is the difficulty of finding a general formula for the defining equations of Spohn CI varieties. For instance, in Section \ref{sec:dim}, we introduced the variety $Y_{X,G}$, which is close to be a good description of the Spohn CI variety. 
 In \cite[Proposition 25]{portakal2024game}, it is shown that these two varieties are indeed equal for Nash CI Varieties and generic binary games.
 However, Example \ref{ex: wrong saturaion} shows that in some cases further saturation needs to be done in $Y_{X,G}$. This lack of equations leads to obstacles while calculating invariants and properties of these varieties, such as their degree, smoothness, irreducibility, etc. From a computational perspective, in order to analyze these properties for concrete examples, one has to first deal with saturation process, which becomes computationally heavy quite rapidly. This computational challenge leads to the following questions: do we need to saturate by all the hyperplanes of the form $\{p_{j_1\ldots j_n}=0\} $ and $\{p_{+\cdots+}=0\}$?  If not, what is the minimum set of hyperplanes we need in the process of saturation? Is there a better set of hyperplanes? We present a partial answer to these questions.

\begin{theorem}\label{theo:saturation}
Let $G$ be a graph with $n$ vertices. Then, for a generic games $X$, the Spohn CI variety $\mathcal{V}_{X,G}$ is obtained by removing from $\mathcal{M}_G\cap\mathcal{V}_X$  the irreducible components lying in hyperplanes of the form $\{p_{+\cdots+a+\cdots +}=0\}$.
\end{theorem}
\begin{proof}
    We need to show that for a generic game $X$,  an irreducible component of $\mathcal{M}_G\cap\mathcal{V}_X$ is contained in a hyperplane of the form $\{p_{j_1\ldots j_n}=0\}$ or $\{p_{+\cdots+}=0\}$  if and only if it is contained in a hyperplane of the form $\{p_{+\cdots+a+\cdots+}=0\}$. This statement can be checked on the pullback via the monomial map $\phi_G$. Recall that $Y_{X,G}$ is obtained by removing from $\phi_G^{-1}(\mathcal{M}_G\cap\mathcal{V}_X)$ some irreducible components contained in hyperplanes of the form $ \{p_{j_1\ldots j_n}=0\}$ and $ \{p_{+\cdots+}=0\}$. For instance, some of these irreducible components arise from the second factor of \eqref{eq:extra factor}. Such factor also divides the pullback of $p_{+\cdots+a+\cdots+}$. Therefore, these removed irreducible components are also contained in hyperplanes of the form $\{p_{+\cdots+a+\cdots+}=0\}$. We deduce that it is enough to check the statement for $Y_{X,G}$. Assume that $Z$ is an irreducible component of $Y_{X,G}$ contained in the pullback $H$ of a hyperplane of the form $\{p_{j_1\ldots j_n}=0\} $ or$\{p_{+\cdots+}=0\}$. As in Theorem \ref{theo:dimension}, $H\cap Y_{X,G}$ has codimension $d_1+\cdots+d_n-n+1$ away from $H\cap B$. Since the codimension of $Z$ is at most $d_1+\cdots+d_n-n$, we deduce that $Z$ is contained in $H\cap B$. By the construction of $B$ in \eqref{eq:big intersection}, for every $\mathbf{a}\in\mathcal{R}_{[n]}$ there exists $k\in [n]$ and $b_k\in[d_k]\setminus\{a_k\}$ such that $Z$ is contained in $B_{a_k,b_k}^{(k)}(X_{a_k}^{(k)})$. By Lemma \ref{lem:base locus}, the base locus  $B_{a_k,b_k}^{(k)}(X_{a_k}^{(k)})$ is contained in $\mathbb{V}(L_{a_k}^{(k)})\cup\mathbb{V}(L_{b_k}^{(k)})$. Therefore, $Z$ is contained in the pullback of a hyperplane of the form $\{p_{+\cdots+a+\cdots+}=0\}$.

Now,  assume that $Z$ is an irreducible component of $Y_{X,G}$ contained in the pullback $H$ of a hyperplane of the form $\{p_{+\cdots+a+\cdots +}=0\} $. As before, by Bertini's Theorem, we deduce that $Z$ is contained in $H\cap B$. As in the proof of Theorem \ref{theo:dimension}, we deduce that $Z$ is contained in the pullback of a  hyperplane of the form $ \{p_{j_1\ldots j_n}=0\}$ or $ \{p_{+\cdots+}=0\}$. Therefore, for generic $X$, the saturation of $Y_{X,G}$ by the hyperplanes of the form $ \{p_{j_1\ldots j_n}=0\}$ or $ \{p_{+\cdots+}=0\}$ equals the saturation by the hyperplanes of the form $\{p_{+\cdots+a+\cdots+}=0\}$.
\end{proof}

\section{Equations and degree of generic Nash CI varieties}\label{sec:nash ci}

In this section, we focus on the geometric properties of generic Nash CI varieties that were introduced in Section \ref{sec: framework and examples}. Let $X$ be an $n$--player game. During this section and Section~\ref{sec:Emptyness}, we fix $G$ to be a cluster graph with $n$ vertices, i.e., $G=G_1\sqcup\cdots\sqcup G_k$ is the disjoint union of $k$ complete graphs. Let $n_i$ be the number of vertices of $G_i$. 
We label the players of the game and the vertices of $G$ by $(1,1),\ldots,(1,n_1),\ldots,(k,1),\ldots,(k,n_k)$, where, $(i,1),\ldots,(i,n_i)$ are the vertices of $G_i$. Similarly, the number of pure strategies of the player $(i,j)$ is $d_{i,j}$. Using this notation we introduce the numbers
\[
D_i = \prod_{j=1}^{n_i}d_{i,j} \text{ and } S_i=\sum_{j=1}^{n_i}d_{i,j}.
\]
For this type of graphs, the graphical model associated to $G$, denoted by $\Mg$, coincides with the Segre variety
\begin{equation}
    \label{eq:Nash and Segree}
    \Mg=\PP^{D_1-1}\times\cdots\times\PP^{D_k-1}.
\end{equation}
Let $\sigma^{(i)}_j$ for $i\in[k]$ and $j\in\prod_{j=1}^{n_i}[d_{i,j}]$ be the coordinates of the $i$--th factor of $ \Mg$.
In this setting, the Spohn CI variety is called the \emph{Nash CI variety} of the game. Similarly, for a cluster graph, the Spohn CI equilibrium is called \emph{Nash CI equilibrium}.

\begin{remark}
    From a game-theoretic perspective, the Nash CI equilibria model the situation where the players are divided into $k$ groups of $n_1,\ldots,n_k$ players respectively. Each group corresponds to a connected component of the graph. Players in each group behave collectively, but there is no communication between distinct groups. From a geometric perspective, such division into groups is described through the Segre parametrization 
    \[
    p_{j_{1,1}\cdots j_{k,n_k}}=\sigma^{(1)}_{j_{11}\cdots j_{1,n_1}}\cdots \sigma^{(k)}_{j_{k1}\cdots j_{k,n_k}}
    \]
    of the Segre variety $\mathcal{M}_G$. Recall that $\sigma^{(i)}_j$ are the coordinates of the $i$--th factor of $\mathcal{M}_G$. The Segre parametrization splits the mixed strategy $p$ into a mixed strategy for each factor of the Segre variety (each connected component of $G$). Therefore, the action of deciding collectively a mixed strategy as in the dependency setting is now divided into each group deciding independently their mixed strategy, and then joining them together to create the mixed strategy of the game.
\end{remark}

In the remaining sections of the paper, we will focus on the study of the algebro-geometric properties of generic Nash CI varieties.
In the binary case, these varieties where introduced and studied in \cite[Section 5]{portakal2024game} where their dimension and degree was determined. In this section, we carry out such computation in the nonbinary case.
From Theorem \ref{theo:dimension}, we derive the dimension of Nash CI varieties for generic games:

\begin{corollary}\label{cor:dim Nash}
    Assume that for generic games $X$, the Nash CI variety is nonempty. Then, the dimension of the Nash CI variety is 
    \[
    D_1+\cdots+D_k-S_1-\cdots-S_k+n-k.
    \]
\end{corollary}

In order to apply the dimension formula given in Corollary \ref{cor:dim Nash}, we need to understand when a generic Nash CI variety is nonempty. The strategy to analyze the emptiness of Nash CI varieties is to use the Chow ring of the graphical model $\Mg$. In Subsection \ref{sec:Nash and Chow} we compute the class of a generic Nash CI variety in the Chow ring of $\Mg$ for generic games. As a consequence of this computation, we provide a formula for the degree of Nash CI varieties. In Section \ref{sec:Emptyness}, we will give sufficient and necessary conditions of a generic Nash CI variety for being nonempty.

\subsection{Defining equations}\label{sec:equations}

In order to provide a formula for the class of  generic Nash CI varieties in the Chow ring of $\mathcal{M}_G$, we need to find the equations of generic Nash CI varieties.
To do so, we set some notations first. Fix the player $(a,b)$ where $a\in[k]$ and $b\in [n_a]$. For any $l\in[d_{a,b}]$, we consider the linear form $L_l^{(a,b)}$ in $\PP^{D_a-1}$ given by 
\[
\displaystyle L_l^{(a,b)}:= \sum_{j_1\cdots \hat{j_b}\cdots j_{n_a}} \sigma_{j_1\cdots l\cdots j_{n_a}}^{(a)}.
\]
The linear form $L_l^{(a,b)}$ corresponds to the polynomial \eqref{eq: poly L} for this concrete type of graphs.
Similarly, given a game $X$ and for $l\in[d_{a,b}]$, we consider the polynomial
\[
\mathcal G_{X,l}^{(a,b)}:=
\displaystyle \sum_{j_{11}\cdots\widehat{j_{ab}} \cdots j_{kn_k}}\!\!\!\!\! X^{(i,j)}_{j_{11}\cdots l\cdots j_{k,n_k}}\sigma^{(1)}_{j_{11}\cdots  j_{1n_1}}\cdots \sigma^{(a)}_{j_{a1}\cdots l\cdots j_{an_a}}\cdots \sigma^{(k)}_{j_{k1}\cdots  j_{kn_k}}
\]
if $(a,b)$ is not an isolated vertex, and
\[
\mathcal G_{X,l}^{(a,b)}:=
\displaystyle \sum_{j_{11}\cdots\widehat{j_{ab}} \cdots j_{kn_k}}\!\!\!\!\! X^{(i,j)}_{j_{11}\cdots l\cdots j_{k,n_k}}\sigma^{(1)}_{j_{11}\cdots  j_{1n_1}}\cdots \widehat{\sigma^{(a)}_{j_{a1}\cdots l\cdots j_{an_a}}}\cdots \sigma^{(k)}_{j_{k1}\cdots  j_{kn_k}}
\]
if $(a,b)$ is an isolated vertex.
As done in Section \ref{sec:dim}, evaluating in the matrices $M^{(a,b)}$ the parametrization of the Segre variety \eqref{eq:Nash and Segree} leads to rows and columns with a common factor of the form 
\[
\displaystyle\sum_{j_{11}\cdots j_{1n_1}}\cdots \widehat{\sum_{j_{a,1}\cdots j_{a,n_a}}}\cdots \sum_{j_{k,1}\cdots j_{k,n_k}}\sigma_{j_{11}\cdots j_{1n_1}}^{(1)}\cdots\widehat{\sigma_{j_{a1}\cdots j_{an_a}}^{(a)}} \cdots \sigma_{j_{k1}\cdots j_{kn_k}}^{(k)}.
\]
If $n_a=1$, i.e. if $(a,b)$ is an isolated vertex, the common factor $\sigma_j^{(a)}$ for $j\in[d_a]$ also appears. After removing these factors we get the $d_{(a,b)}\times 2$ matrix $\widetilde{M}^{(a,b)}$ defined as 
\[
\widetilde{M}^{(a,b)}=\begin{pmatrix}
    L_1^{(a,b)} & \mathcal G_{X,1}^{(a,b)}\\
\vdots &\vdots\\ 
    L_{d_{a,b}}^{(a,b)} & \mathcal G_{X,d_{a,b}}^{(a,b)}
\end{pmatrix}
\]
if $(a,b)$ is not an isolated vertex or 
\[
\widetilde{M}^{(a,b)}=\begin{pmatrix}
    1 && \mathcal G_{X,1}^{(a,b)}\\
\vdots &&\vdots\\ 
    1 && \mathcal G_{X,d_{a,b}}^{(a,b)}
\end{pmatrix}
\]
if $(i,l)$ is an isolated vertex.
 For a game $X$, we defined the variety $N_{X,G}$ in $\mathcal{M}_G$ as the zero locus of the $2\times2$ minors of the matrices $\widetilde{M}^{(1,1)},\ldots,\widetilde{M}^{(k,n_k)}$. Note that by construction, the Nash CI variety $\Vc$  is contained in $N_{X,G}$ since the common factors removed from the matrices $M^{(a,b)}$ are factors of the equations of hyperplanes of the form $\{p_{j_{11}\cdots j_{kn_k}}=0\}$ and $\{p_{+\cdots +}=0\}$. Next, we show that for generic games, $N_{X,G}$ is actually the Nash CI variety. To do so, we translate the base locus study carried out in Lemma \ref{lem:base locus} to the setting of Nash CI varieties. For a player $(a,b)$ and two strategies $r,s\in[d_{a,b}]$, we consider the polynomial $F^{(a,b)}_{r,s}$ given by the minor 
 \[ \begin{array}{ccccc}
F^{(a,b)}_{r,s}:=\left| 
\begin{array}{cc}
L^{(a,b)}_{r}& \mathcal G^{(a,b)}_{X,r} \\
L^{(a,b)}_{s}& \mathcal G^{(a,b)}_{X,s}
\end{array}
\right| 
&\text{ if }n_a>1& \text{ and } &
F^{(a,b)}_{r,s}:=\left| 
\begin{array}{cc}
1& \mathcal G^{(a,b)}_{X,r} \\
1& \mathcal G^{(a,b)}_{X,s}
\end{array}
\right| &\text{ if }n_a=1.
\end{array}
 \]

In other words, $F^{(a,b)}_{r,s}$ is the $2\times2$ minor of $\widetilde{M}^{(a,b)}$ corresponding to the rows $r$ and $s$. So the variety $N_{X,G}$ is the zero locus of all the polynomials $F^{(a,b)}_{r,s}$. Now, consider the map $\varphi^{(a,b)}_{r,s}$ that associates to the payoff tensor $X^{(a,b)}$ the polynomial
$F^{(a,b)}_{r,s}$. The image of this map is a linear series denoted by  $\Lambda_{r,s}^{(a,b)}$. If $(a,b)$ is not an isolated vertex, then $\Lambda_{r,s}^{(a,b)}$ is a linear subseries of $$|\mathcal{O}(1,\ldots,1,\underset{(a)}{2},1,\ldots,1)|.$$ On the other hand, if $(a,b)$ is an isolated vertex, the map $\varphi^{(a,b)}_{r,s}$  is surjective and $\Lambda_{r,s}^{(a,b)}$ is the complete linear series 
$$|\mathcal{O}(1,\ldots,1,\underset{(a)}{0},1,\ldots,1)|.$$ 
The description of the generators and the base locus of the linear system $\Lambda_{r,s}^{(a,b)}$ is similar as the one given in  \cite[Lemma 7, Lemma 8]{portakal2024game}. As in Section \ref{subsec:base locus}, we are interested in the linear subsystem $\Lambda_{r,s}^{(a,b)}(X^{(a,b)}_r)$ of $\Lambda_{r,s}^{(a,b)}$ given by fixing the slice payoff tensor $X_r^{(a,b)}$. The base locus of this linear system is described in Lemma \ref{lem:base locus}. Using the notation of Nash CI varieties, these generators are described as follows. Let $(a,b)$ be a non isolated vertex, let $r,s\in[d_{a,b}]$ and fix the $r$--th slice payoff tensor $X_r^{(a,b)}$. Then, the base locus $B_{r,s}^{(a,b)}(X_r^{(a,b)})$ of the linear series  $\Lambda_{r,s}^{(a,b)}(X_r^{(a,b)})$ is the union of the varieties
\begin{equation}\label{eq:base locus nash}
\begin{array}{cccc}
\mathbb{V}(\sigma^{(a)}_{j_{a1}\cdots s\cdots j_{an_a}} :\text{ for every }\, j_{a1}\cdots \widehat{j_{ab}}\cdots j_{an_a}),&\mathbb{V}(L_r^{(a,b)},\mathcal G_r^{(a,b)}),& \text{and}&\mathbb{V}(L_r^{(a,b)},L_s^{(a,b)}).
\end{array}
\end{equation}

Using this base locus, we derive the following result.

\begin{proposition}\label{prop:eq Nash CI}
    For generic games $X$, the variety  $N_{X,G}$ equals the Nash CI variety $\Vc$.
\end{proposition}
\begin{proof}
If $N_{X,G}$  is empty, then  $\Vc$ is also empty. Assume now that $N_{X,G}$   is nonempty.
Note that by construction, the Nash CI variety $\Vc$ is obtained by removing from the variety  $N_{X,G}$ all the irreducible components contained in a hyperplane of the form  $\{p_{j_{11}\cdots j_{kn_k}}=0\}$ or $\{p_{+\cdots +}=0\}$. In order to show that  $N_{X,G}=\Vc$
for generic games, it is enough to check that $N_{X,G}$ has no irreducible component contained in a hyperplane  the form  $\{p_{j_{11}\cdots j_{kn_k}}=0\}$ or $\{p_{+\cdots +}=0\}$. 
 Consider the integer vector 
  $$\overrightarrow{m}=\left(\ldots,m_{a,b},\ldots \right)_{a\in[k],\,b\in[n_a]}$$
  whose entries are indexed by the players of the game and $2\leq m_{a,b}\leq d_{a,b}$ for each player $(a,b)$. For a player $(a,b)$, let $M^{(a,b)}_{m_{a,b}}$ be the submatrix of $\widetilde{M}^{(a,b)}$ given by the first $m_{a,b}$ rows. 
  For a game $X$, we consider the variety $N_{X,G,\overrightarrow{m}}$ defined by the $2\times 2$ minors of the matrices $M^{(a,b)}_{m_{a,b}}$ for every player $(a,b)$. Note that for $\overrightarrow{m}=(d_{1,1},\ldots,d_{k,n_k})$, we get that $N_{X,G,\overrightarrow{m}}=N_{X,G}$. We say that $\overrightarrow{m}\leq \overrightarrow{m}' $ if $m_{a,b}\leq m_{a,b}'$ for all players $(a,b)$. Then, for $\overrightarrow{m}\leq \overrightarrow{m}'$, we get that $N_{X,G,\overrightarrow{m}'}\subseteq N_{X,G,\overrightarrow{m}}$. We claim that for any integer vector 
 $$\overrightarrow{m}=\left(\ldots,m_{a,b},\ldots \right)_{a\in[k],\,b\in[n_a]}$$
with $2\leq m_{a,b}\leq d_{a,b}$ and for a generic game $X$, the variety $N_{X,G,\overrightarrow{m}}$ has codimension $|\overrightarrow{m}|-n$ in $\mathcal{M}_G$ and it has no irreducible component contained in a hyperplane of the form $\{p_{j_{11}\cdots j_{kn_k}}=0\}$ or $\{p_{+\cdots +}=0\}$. 
Here $|\overrightarrow{m}|=\sum m_{a,b}$. 
We argue by induction on the integer vector $\overrightarrow{m}$ and the number of strategies of each player.
The first case happens when $\overrightarrow{m}=(2,\ldots,2)$. In this case, the matrices $M^{(a,b)}_{m_{a,b}}$ are squared and the proof follows exactly as in the binary game case in \cite[Theorem 9]{portakal2024game} and \cite[Proposition 25]{portakal2024game}.

Now, fix an integer vector $\overrightarrow{m}>(2,\ldots,2)$ and assume that the claim holds for every integer vector lower than  $\overrightarrow{m}$. Then, there exists a player $(a,b)$ such that $m_{a,b}>2$. Consider the integer vector $\overrightarrow{m}':=\overrightarrow{m}-e_{(a,b)}$. First, we show that $N_{X,G,\overrightarrow{m}}$ has codimension $|\overrightarrow{m}|-n$ in $\mathcal{M}_G$. Using \cite[Proposition 12.2]{harris2013algebraic}, we get that
\begin{equation}\label{eq: codim induction}
\mathrm{codim}_{\mathcal{M}_G}\,N_{X,G,\overrightarrow{m}}\leq |\overrightarrow{m}|-n.
\end{equation}
 By construction we have that 
\begin{equation}\label{eq:inclusion induction}
N_{X,G,\overrightarrow{m}} = N_{X,G,\overrightarrow{m}'}\cap \mathbb{V}\left(F^{(a,b)}_{r,m_{a,b}}:\text{ for } 1\leq r< m_{a,b}\right).
\end{equation}
In particular, we get that 
\[
N_{X,G,\overrightarrow{m}} \subseteq  N_{X,G,\overrightarrow{m}'}\cap \mathbb{V}(F_{r,m_{a,b}})
\]
for $1\leq r<m_{a,b}$.
By Bertini's Theorem, for generic games $X$ we have that $N_{X,G,\overrightarrow{m}}$ has the expected codimension $|\overrightarrow{m}|-n$  away from the base locus $B_{r,m_{a,b}}^{(a,b)}(X_{r}^{(a,b)})$ of  $\Lambda_{r,m_{a,b}}^{(a,b)}(X_{r}^{(a,b)})$ for $1\leq r<m_{a,b}$. If $(a,b)$ is an isolated vertex, $\Lambda_{1,m_{a,b}}^{(a,b)}$ is complete and its base locus is empty. Assume now that $(a,b)$ is not an isolated vertex, i.e. $n_a\geq 2$.
Then, as in the proof of Theorem \ref{theo:dimension}, $N_{X,G,\overrightarrow{m}}$ has the codimension $|\overrightarrow{m}|-n$  away from
\begin{equation}\label{eq:def B}
B:= \bigcap_{1\leq r<m_{a,b}}B_{r,m_{a,b}}^{(a,b)}(X_{r}^{(a,b)}).
\end{equation}
Consider the varieties
\[
\begin{array}{l}
\displaystyle B_1=  \mathbb{V}(\sigma^{(a)}_{j_{a1}\cdots m_{a,b}\cdots j_{an_a}} :\forall\, j_{a1}\cdots \widehat{j_{ab}}\cdots j_{an_a}), \\
\displaystyle
B_2 = \mathbb{V}(L_1^{(a,b)},\ldots,L_{m_{a,b}}^{(a,b)}),\\
\displaystyle
B_3 = \mathbb{V}(L_r^{(a,b)},\mathcal G_{r}^{(a,b)}:1\leq r<m_{a,b}).
\end{array}
\]
Using \eqref{eq:base locus nash}, we get that 
\[\begin{array}{c}
\vspace*{1mm}\displaystyle B = B_1\cup \bigcap_{1\leq r<m_{a,b}} 
\left(
\mathbb{V}(L_r^{(a,b)},\mathcal G_r^{(a,b)})\cup\mathbb{V}(L_r^{(a,b)},L_{m_{a,b}}^{(a,b)})
\right) \\ 
\vspace*{1mm}\displaystyle  =B_1\cup \bigcap_{1\leq r<m_{a,b}} 
\mathbb{V}(L_r^{(a,b)})\cap\left(
\mathbb{V}(\mathcal G_r^{(a,b)})\cup\mathbb{V}(L_{m_{a,b}}^{(a,b)})
\right) \\ 
\vspace*{3mm}=\displaystyle  B_1\cup \left(\mathbb{V}(L_r^{(a,b)}:1 \leq r<m_{a,b})\cap\bigcap_{1\leq r<m_{a,b}} 
\left(
\mathbb{V}(\mathcal G_r^{(a,b)})\cup\mathbb{V}(L_{m_{a,b}}^{(a,b)})
\right) \right) \\ 
= \displaystyle  B_1\cup \mathbb{V}(L_1^{(a,b)},\ldots,L_{m_{a,b}}^{(a,b)})\cup\mathbb{V}(L_r^{(a,b)},\mathcal G_r^{(a,b)}:1\leq r<m_{a,b})=B_1\cup B_2\cup B_3
.
\end{array}
\]

Therefore, $N_{X,G,\overrightarrow{m}}$ has codimension $|\overrightarrow{m}|-n$  away from $B=B_1\cup B_2\cup B_3$. It remains to study the codimension of the irreducible components of $N_{X,G,\overrightarrow{m}}$ lying in the $B_i$ for $i\in[3]$. We show that actually, $N_{X,G,\overrightarrow{m}}$ has no irreducible component contained in the $B_i$ for $i\in[3]$.
Assume first that $N_{X,G,\overrightarrow{m}}$ has an irreducible component $Z$ contained in $B_1$. This means that $Z\subseteq N_{X,G,\overrightarrow{m}}\cap B_1$ with $\codim_{\mathcal{M}_G}Z\leq |\overrightarrow{m}|-n$ by \eqref{eq: codim induction}. On the other hand,
the intersection of $N_{X,G,\overrightarrow{m}}$ with $B_1$ leads to the case where the player $(a,b)$ has one strategy less. Mainly, the strategy $m_{a,b}$ of the player $(a,b)$ is removed. 
By induction, the codimension of $ N_{X,G,\overrightarrow{m}}\cap B_1$ in $\mathcal{M}_G$ is 
\[
|\overrightarrow{m}|-1-n+\prod_{r\neq b}^{n_a}d_{a,r}.
\]
Since the vertex $(a,b)$ is not isolated, we get that 
\[
\mathrm{codim}_{\mathcal{M}_G}  N_{X,G,\overrightarrow{m}}\cap B_1=|\overrightarrow{m}|-1-n+\prod_{r\neq b}^{n_a}d_{a,r}\geq |\overrightarrow{m}|-1-n +2=|\overrightarrow{m}|-n+1.
\]
This is a contradiction with the fact that $Z$ is an irreducible component of this intersection with $\codim_{\mathcal{M}_G}Z\leq |\overrightarrow{m}|-n$. Thus, we conclude that $N_{X,G,\overrightarrow{m}}$ does not have any irreducible component contained in $B_1$.

Now, we focus on the component $B_2$. Assume that $N_{X,G,\overrightarrow{m}}$ has an irreducible component $Z$ in $B_2$. 
We show that this is a contradiction since the intersection $N_{X,G,\overrightarrow{m}}\cap \mathbb{V}(L^{(a,b)}_1,\ldots,L^{(a,b)}_{m_{a,b}}) $  has codimension $|\overrightarrow{m}|-n+1$.
If $L^{(a,b)}_1=\cdots=L^{(a,b)}_{m_{a,b}}=0$,
 $M^{(a,b)}_{m_{a,b}}$ has rank at most $1$.
Then, 
\begin{equation*}\label{eq:tech inter}
N_{X,G,\overrightarrow{m}}\cap \mathbb{V}(L^{(a,b)}_1,\ldots,L^{(a,b)}_{m_{a,b}}) 
= N_{X,G,\overrightarrow{m}-m_{a,b}\overrightarrow{e}_{a,b}}\cap \mathbb{V}(L^{(a,b)}_1,\ldots,L^{(a,b)}_{m_{a,b}}) .
\end{equation*}
 Note that by induction $N_{X,G,\overrightarrow{m}-m_{a,b}\overrightarrow{e}_{a,b}}$ has codimension $|\overrightarrow{m}|-m_{a,b}-n+1$. Here we used that for the vector $\overrightarrow{m}-m_{a,b}\overrightarrow{e}_{a,b}$ there is no contribution coming from the player $(a,b)$, so we can treat this case as if we had a player less. Therefore, it is enough to show that $N_{X,G,\overrightarrow{m}-m_{a,b}\overrightarrow{e}_{a,b}} $ and $\mathbb{V}(L^{(a,b)}_1,\ldots,L^{(a,b)}_{m_{a,b}})$ intersect in the expected dimension 
$|\overrightarrow{m}|-m_{a,b}-n+1+m_{a,b}=|\overrightarrow{m}|-n+1$. 
Note that since $L^{(a,b)}_1,\ldots,L^{(a,b)}_{m_{a,b}}$  are linear forms in the $a$--th factor of the Segre variety $\mathcal{M}_G$, it will intersect with the expected dimension the variety defined by the contributions of the players corresponding to the other factors of $\mathcal{M}_G$ (or connected components of $G$). This follows from Bertini's Theorem and the fact that the base loci in Lemma \ref{lem:base locus} corresponding to players in other connected components of $G$ are union of linear spaces supported in other factors of the Segre variety $\mathcal{M}_G$. Therefore, we may assume that $G$ is connected, the players are labeled by $(1,1),\ldots,(1,n)$ and the player $(a,b)$ is the player $(1,n)$. In particular, we are looking to Spohn varieties and the entries of the matrices $M^{(1,i)}$ are linear forms. 

Given, $\lambda^{(b)} = [\lambda_1^{(b)},\lambda_2^{(b)}]\in\PP^1$, the generalized column $C^{(1,b)}_{\lambda^{(b)}}$ of $M^{(1,b)}$ is a linear combination 
\[
M^{(1,b)}\begin{pmatrix}
    \lambda_1^{(b)} \\ \lambda_2^{(b)}
\end{pmatrix}
\]
of the two columns of $M^{(1,b)}$. For $\lambda=(\lambda^{(1)},\ldots,\lambda^{(n)})\in (\PP^1)^n$, 
we defined $\mathcal{C}_{X,\lambda}$ as the zero locus of the linear forms in the entries of the generalized columns $C^{(1,1)}_{\lambda^{(1)}},\ldots, C^{(1,n)}_{\lambda^{(n)}}$.
Following the proof of \cite[Theorem 6.4]{eisenbud2006geometry}, we get that 
\[
\mathcal{V}_{X}=\bigcup_{\lambda\in (\PP^1)^n}\mathcal{C}_{X,\lambda}
\]
and we have a map 
\begin{equation}\label{eq:fiber map spohn}
\mathcal{V}_{X} \longrightarrow (\PP^1)^n
\end{equation}
whose fibers are the varieties $\mathcal{C}_{X,\lambda}$ defined by the generalized columns.
In the proof of \cite[Theorem 6]{portakal2022geometry} it is shown that for a generic game $X$, all the fibers of \eqref{eq:fiber map spohn} have codimension $d_1+\cdots+d_n$. In particular, $\mathcal{C}_{X,\lambda}$ has codimension $d_1+\cdots+d_n$ for every $\lambda$ of the form $\lambda=(\lambda^{(1)},\ldots,\lambda^{(n-1)},[1,0])$. Note that for such $\lambda$, the equations of $\mathcal{C}_{X,\lambda}$ coming from the player $(1,n)$ are $L_1^{(1,n)},\ldots,L_{d_{1,n}}^{(1,n)}$. Restricting the above construction to the first $m_{1,i}$ rows of the matrix $M^{(1,i)}$ for $i\in[n-1]$, we obtained the map 
\begin{equation}\label{eq:fiber map nash}
N_{X,G,\overrightarrow{m}-m_{a,b}\overrightarrow{e}_{a,b}}\cap \mathbb{V}(L^{(1,n)}_1,\ldots,L^{(1,n)}_{m_{a,b}})\longrightarrow (\PP^1)^{n-1}\times\{[1,0]\}
\end{equation}
whose fibers are the subvariety of  $\mathcal{C}_{X,\lambda}$ for  $\lambda=(\lambda^{(1)},\ldots,\lambda^{(n-1)},[1,0])$ given by the corresponding rows of the corresponding generalized columns. Since  $\mathcal{C}_{X,\lambda}$ is a linear subspace of the expected dimension, we deduce that the fibers of the map \eqref{eq:fiber map nash} have codimension $|\overrightarrow{m}|$. By Theorem \cite[Theorem 11.12]{harris2013algebraic}, we deduce that 
\[
\codim_{\mathcal{M}_G} N_{X,G,\overrightarrow{m}-m_{a,b}\overrightarrow{e}_{a,b}}\cap \mathbb{V}(L^{(1,n)}_1,\ldots,L^{(1,n)}_{m_{a,b}}) = |\overrightarrow{m}|-n+1.
\]
We conclude that $ N_{X,G,\overrightarrow{m}}$ has no irreducible component contained in 
$B_2$.

Lastly, assume that $ N_{X,G,\overrightarrow{m}}$ has an irreducible component $Z$ contained in $B_3$. The matrix $M^{(a,b)}$ restricted to $B_3$ has rank $1$. Therefore, the intersection $N_{X,G,\overrightarrow{m}} \cap B_3$ equals 
\begin{equation}\label{eq: third base loci}
N_{X,G,\overrightarrow{m}-m_{a,b}\overrightarrow{e}_{a,b}}\cap B_3 =  N_{X,G,\overrightarrow{m}-m_{a,b}\overrightarrow{e}_{a,b}}\cap\mathbb{V}(L_r^{(a,b)}:1\leq r<m_{a,b})\cap\mathbb{V}(\mathcal G_r^{(a,b)}:1\leq r<m_{a,b}).
\end{equation}
As before, $ N_{X,G,\overrightarrow{m}-m_{a,b}\overrightarrow{e}_{a,b}}$ has codimension $|\overrightarrow{m}|-m_{a,b}-n+1$. By the previous case, the intersection
\[
N_{X,G,\overrightarrow{m}-m_{a,b}\overrightarrow{e}_{a,b}}\cap\mathbb{V}(L_r^{(a,b)}:1\leq r<m_{a,b})
\]
has codimension $|\overrightarrow{m}|-m_{a,b}-n+1+m_{a,b}-1=|\overrightarrow{m}|-n$. Since 
$m_{a,b}>2$ and the polynomials $\mathcal G_r^{(a,b)}$ can be chosen generic, we get that the intersection \eqref{eq: third base loci} has codimension at least $|\overrightarrow{m}|-n+1$. We deduce that $N_{X,G,\overrightarrow{m}}$ can not have an irreducible component contained in $B_3$.

We conclude that the intersection of $N_{X,G,\overrightarrow{m}}$ with $B$ has codimension at least $|\overrightarrow{m}|-n+1$. Therefore, $N_{X,G,\overrightarrow{m}}$ does not have an irreducible component contained in $B$ and the codimension of $N_{X,G,\overrightarrow{m}}$ is $|\overrightarrow{m}|-n$.
Now, we show that for generic games  $N_{X,G,\overrightarrow{m}}$ does not have an irreducible component contained in a hyperplane of the form $\{p_{j_{11}\cdots j_{kn_k}}=0\}$ or $\{p_{+\cdots +}=0\}$.
Let $H$ be such a hyperplane and assume that there exists an irreducible component $Z$ of $N_{X,G,\overrightarrow{m}}$ contained in $H$. 
 As before, using Equation \eqref{eq:inclusion induction} and Bertini's Theorem,  we deduce that the intersection $N_{X,G,\overrightarrow{m}}\cap H$ has codimension $|\overrightarrow{m}|-n+1$ away from $B$. This implies that $Z$ is contained in $B$ which is a contradiction.
\end{proof}

In \cite[Conjecture 24]{portakal2022geometry}, it was conjectured the dimension and irreducibility of generic Spohn CI varieties of binary games. 
The dimension part of the conjecture was proven in \cite{portakal2024game}. Section \ref{sec:dim} extends this results to nonbinary games. However, the irreducibility of generic Spohn CI varieties conjectured in \cite[Conjecture 24]{portakal2022geometry} was not approached before. 
Using the arguments used in Proposition \ref{prop:eq Nash CI}, we derive the irreducibility of generic Nash CI varieties.

\begin{theorem}
    \label{theo: irreducibility}
    Let $G$ be a cluster graph with at least one edge. Then, the Nash CI variety $\mathcal{V}_{X,G}$ is irreducible for generic games $X$.
\end{theorem}
\begin{proof}
    First, note that if $\mathcal{V}_{X,G}$ is empty for generic games, then it is also irreducible. So we may assume that  $\mathcal{V}_{X,G}$ is nonempty for generic games. Then, since $G$ has at least an edge, we have that $\mathcal{V}_{X,G}$ has positive dimension for generic games by Corollary \ref{cor:dim Nash}. Therefore, we can apply the behaviour of irreducibility under Bertini's Theorem.

   As in the proof of Proposition \ref{prop:eq Nash CI}, we show that $N_{X,G,\overrightarrow{m}}$ is irreducible  by induction on $\overrightarrow{m}$ and the number of strategies.
   The initial case is the case of binary games, which is analogous to the case when $\overrightarrow{m}=(2,\ldots,2)$. In other words, assume that $d_{a,b}=2$ for any player $(a,b)$. Then, the matrices $\widetilde{M}^{(a,b)}$ are square matrices and we denote their determinant by $F^{(a,b)}=F_{1,2}^{(a,b)}$.
   For each player $(a,b)$, we consider the variety 
   \[
   N_{X,G,(a,b)}=\mathbb{V}(F^{(r,s)}:(r,s)\leq (a,b)).
   \]
   In particular, we have that $N_{X,G,(k,n_k)}=N_{X,G}=\mathcal{V}_{X,G}$. We show by induction on the player $(a,b)$ that $ N_{X,G,(a,b)}$ is irreducible for a generic game $X$. Let $X$ be a generic game.
   For $(a,b)=(1,1)$, we have that $N_{X,G,(1,1)}=\mathbb{V}(F^{(1,1)})$ is irreducible  since for generic payoff tensors the polynomial $F^{(1,1)}$ is irreducible. Now, for $(a,b)>(1,1)$ we have that 
   \[
    N_{X,G,(a,b)}= \mathbb{V}(F^{(r,s)}:(r,s)<(a,b))\cap\mathbb{V}(F^{(a,b)}).
   \]
   By induction, $\mathbb{V}(F^{(r,s)}:(r,s)<(a,b))$ is irreducible. Therefore,
    by Bertini's Theorem, $N_{X,G,(a,b)}$ is irreducible away from the singular locus of $ \mathbb{V}(F^{(r,s)}:(r,s)<(a,b))$ and the base locus $B^{(a,b)}:=B_{1,2}^{(a,b)}$ of the linear series  $E^{(a,b)}:=\Lambda_{1,2}^{(a,b)}$.  Again, by Bertini's Theorem, the singular locus of $ \mathbb{V}(F^{(r,s)}:(r,s)<(a,b))$ is contained in the union
    \[
    \bigcup_{(r,s)<(a,b)}B^{(r,s)}.
    \]
    We conclude that $N_{X,G,(a,b)}$ is irreducible away from the union
      \begin{equation}\label{eq:union irred}
    \bigcup_{(r,s)\leq(a,b)}B^{(r,s)}.
    \end{equation}
    Now, we show that $N_{X,G,(a,b)}$ does not have any irreducible component contained in the union \eqref{eq:union irred}. Assume on the contrary that $X$ has an irreducible component $Z$ contained in \eqref{eq:union irred}. Then, there exists $(a',b')\leq (a,b)$ such that $Z$ is contained in $B^{(a',b')}$. Note that $(a',b')$ can not be an isolated vertex of $G$ since in that case $B^{(a',b')}$ is empty.
    In particular, we have that
    \[
    Z\subseteq N_{X,G,(a,b)}\cap B^{(a',b')} = \mathbb{V}(F^{(r,s)}:(r,s)\leq (a,b)\text{ and } (r,s)\neq (a',b')) \cap B^{(a',b')}.
    \]
    Note that the first variety in the above intersection has dimension  $2^{n_1}+\cdots + 2^{n_k}-k-n_{a,b}+1$ where $n_{a,b}$ is the number of players $(r,s)$ with $(r,s)\leq (a,b)$.
    By \cite[Lemma 8]{portakal2024game}, $ B^{(a',b')}$ has three irreducible components:
    \[
\begin{array}{l}
B_1:=   \mathbb{V}(\sigma^{(a')}_{j_{a'1}\cdots 2\cdots j_{a'n_{a'}}} :\text{ for every }\, j_{a'1}\cdots \widehat{j_{a'b'}}\cdots j_{a'n_{a'}}),
\\
B_2:=     \mathbb{V}(\sigma^{(a')}_{j_{a'1}\cdots 1\cdots j_{a'n_{a'}}} :\text{ for every }\, j_{a'1}\cdots \widehat{j_{a'b'}}\cdots j_{a'n_{a'}}),
     \\
 B_3:=    \mathbb{V}(L_1^{(a',b')},L_{2}^{(a',b')}).

     \end{array}
\]

Assume first that $Z$ is contained in $B_1$. This means that we are erasing the second strategy of the player $(a',b')$ from the game. Since the player  $(a',b')$ has binary strategies, the equations are the same as if we erase this player from the game. Therefore, the dimension is
\begin{equation*}
    \label{eq: ineq dim irred}
    \dim\left( \mathbb{V}(F^{(r,s)}:(r,s)\leq (a,b)\text{ and } (r,s)\neq (a',b'))\cap B_1\right) = \sum_{i\in [k]\backslash \{a'\}}2^{n_i}+2^{n_{a'}-1}-k-n_{a,b}+1
\end{equation*}
because $(a',b')$ is not isolated. In particular, $\displaystyle\dim Z\leq  \sum_{i\in [k]\backslash \{a'\}}2^{n_i}+2^{n_{a'}-1}-k-n_{a,b}+1$.

On the other hand, by Corollary \ref{cor:dim Nash}, $\dim(Z)=2^{n_1}+\cdots+ 2^{n_k}-n_{a,b}-k$. Therefore, we get that 
    \[
    \dim Z  \leq  \sum_{i\in [k]\backslash \{a'\}}2^{n_i}+2^{n_{a'}-1}-k-n_{a,b}+1= \sum_{i=1}^k2^{n_i}-2^{n_{a'}-1}-k-n_{a,b}+1=\dim Z+1-2^{n_{a'}-1},
    \]
    and hence, $1-2^{n_{a'}-1}=0$, which implies that $n_{a'}=1$. This is a contradiction since $(a',b')$ is not an isolated vertex of $G$ and $n_{a',b'}\geq 2$. The same argument works if we consider the component $B_2$ of the base locus.

    Assume now that $Z$ is contained in $B_3$. In other words,  $Z$ is contained in 
    \begin{equation}\label{eq: big cap base locus}
    N_{X,G,(a,b)}\cap B_3=\mathbb{V}(F^{(r,s)}:(r,s)\leq (a,b),r\neq a')\cap \mathbb{V}(L_1^{(a',b')},L_2^{(a',b')},F^{(a',s)}:s\neq b').
    \end{equation}
    We now calculate the dimension of \eqref{eq: big cap base locus}.
    We focus on the second term in the intersection \eqref{eq: big cap base locus}. For $s\in[n_{a'}]$, we consider the matrix 
    \[
    \widehat{M}^{(a',s)}:=\begin{pmatrix}
        L_1^{(a',s)}&\mathcal G_1^{(a',s)}\\ L_1^{(a',s)}+L_2^{(a',s)}&\mathcal G_1^{(a',s)}+\mathcal G_2^{(a',s)}
    \end{pmatrix}
    \]
    obtained by summing the two rows of of $\widetilde{M}^{(a',s)}$. Note that 
    $ L_1^{(a',s)}+L_2^{(a',s)}=L_1^{(a',b')}+L_2^{(a',b')}$. Therefore, 
    modulo $L_1^{(a',b')}=L_2^{(a',b')}=0$, the determinant of the matrix $\widehat{M}^{(a',s)}$ is 
    $ L_1^{(a',s)} (\mathcal G_1^{(a',s)}+\mathcal G_2^{(a',s)}) $. We deduce that the second term in the intersection \eqref{eq: big cap base locus} is 
    \begin{equation}
        \label{eq: some intersection}
        \mathbb{V}\left(
        L_1^{(a',b')},L_2^{(a',b')},  L_1^{(a',s)} (\mathcal G_1^{(a',s)}+\mathcal G_2^{(a',s)}): s\in[n_{a'}], s\neq b', \text{ and } (r,s)\leq (a,b)
        \right).
    \end{equation}
    The irreducible components of the variety  \eqref{eq: some intersection} are of the form 
    \begin{equation}
        \label{eq: comp some intersection}
    \mathbb{V}\left(
        L_1^{(a',b')},L_2^{(a',b')},  L_1^{(a',s)}, \mathcal G_1^{(a',t)}+\mathcal G_2^{(a',t)}: s\in S \text{ and }t\in T
        \right),
    \end{equation}
    for a partition $S\sqcup T = \{(a',s):s\neq b' \text{ and } (r,s)\leq (a,b)\}$.
     The linear subspace $   \mathbb{V}(
        L_1^{(a',b')}, L_1^{(a',s)}: s\in S)$ has the expected dimension since each of the linear forms has a variable that does not appear in the other. Mainly, $\sigma^{(a')}_{2\cdots 212 \cdots 2}$ only appears in $L_1^{(a',s)}$ (here the index $2$ appears in the position $s$). Now, since the variable $\sigma^{(a')}_{2\cdots 2}$ only appears in the linear form $L_2^{(a',b')}$, the linear subspace  $   \mathbb{V}(
       L_2^{(a',b')}, L_1^{(a',b')}, L_1^{(a',s)}: s\in S)$ has the expected dimension.
    Note that the polynomial $\mathcal G_1^{(a',t)}+\mathcal G_2^{(a',t)}$ is a generic element in the corresponding complete linear system. Therefore, these polynomials intersect with the expected dimension the first term in the intersection \eqref{eq: big cap base locus}. Now, the base locus of the polynomials appearing in the first term of the intersection \eqref{eq: big cap base locus} are supported in a factor of the Segre variety $\mathcal{M}_G$ distinct than the $a'$--th one. On the other hand, the linear forms $L_1^{(a',s)}$ are supported in the $a'$--th factor of $\mathcal{M}_G$. Therefore, the varieties \eqref{eq: comp some intersection} intersect the first term of the intersection \eqref{eq: big cap base locus} transversely. We deduce that the codimension of such intersection is the total number of equations, which is $n_{a,b}+1$. 
    We conclude that the dimension of \eqref{eq: big cap base locus} is 
    \[
    2^{n_1}+\cdots+2^{n_k}-k-n_{a,b}-1.
    \]
    This is a contradiction since $Z$ has dimension $2^{n_1}+\cdots + 2^{n_k}-n_{a,b}-k$ by  Corollary \ref{cor:dim Nash}. We conclude that $Z$ can not be contained in the base locus $B^{(a',b')}$, and hence, $N_{X,G,(a,b)}$ is irreducible. We conclude that the variety $N_{X,G,\overrightarrow{m}}$ is irreducible for the initial case of the induction.

    Now, we do the induction step. Assume that $\overrightarrow{m}>(2,\ldots,2)$ and $N_{X,G,\overrightarrow{m'}}$ is irreducible for $\overrightarrow{m'}<\overrightarrow{m}$. There, there exists a player $(a,b)$ with $m_{a,b}>2$. In particular, we have that  
    \[
    N_{X,G,\overrightarrow{m}} = N_{X,G,\overrightarrow{m}-\overrightarrow{e}_{a,b}}\cap\mathbb{V}(F^{(a,b)}_{r,m_{a,b}}:\text{ for } 1\leq r<m_{a,b}).
    \] 
    Now, by Bertini's Theorem, $N_{X,G,\overrightarrow{m}}$ is irreducible away from $B$ (see \eqref{eq:def B} for the definition of $B$) and the singular locus of $ N_{X,G,\overrightarrow{m}-\overrightarrow{e}_{a,b}}$. By the proof of Proposition \ref{prop:eq Nash CI}, we know that $N_{X,G,\overrightarrow{m}}$ has no irreducible component contained in $B$. Denote the singular locus of $ N_{X,G,\overrightarrow{m}-\overrightarrow{e}_{a,b}}$ by $S$ and let $Z$ be an irreducible component of $N_{X,G,\overrightarrow{m}}$ contained in $S$. By Bertini's Theorem, every irreducible component of $S$ is contained in the a base locus of the form $B^{(a',b')}_r$. The three irreducible components of this base locus are described in \eqref{eq:base locus nash}. Therefore, $Z$ is contained in one of these three components. 
    The same arguments used in the proof of Proposition \ref{prop:eq Nash CI} show that if $Z$ is contained in one of these components, then $Z$ does not have the correct dimension and hence, $Z$ can not be an irreducible component. We deduce that $N_{X,G,\overrightarrow{m}}$ is irreducible.
\end{proof}

\subsection{Nash meets Chow: degree of Nash CI varieties}\label{sec:Nash and Chow}

In this section we use intersection theory to calculate the degree of Nash CI varieties.
Using the K\"unneth formula (see \cite[Theorem 2.10]{eisenbud20163264}), we have that the Chow ring of $\Mg$, denoted by $A[G]$, is 
\begin{equation}\label{eq:chow segre}
A[G]=\Z[x_1,\ldots,x_k]/\langle x_1^{D_1},\ldots,x_k^{D_k}\rangle.
\end{equation}
By Theorem~\ref{theo:dimension}, the class of the Nash CI variety $[\Vc]$ in $A[G]$ is either zero or a polynomial of degree $\sum_{i=1}^k S_i-n$. Moreover, by Proposition \ref{prop:eq Nash CI}, the Nash CI variety is a determinantal variety. Therefore, we can use Porteous' formula \cite[Theorem 12.4]{eisenbud20163264} to compute the class of $\Vc$ in $A[G]$. In order to apply this formula, we use the following lemma.

\begin{lemma}\label{lemma:matrix}
    Consider the $l\times l$ matrix of the form
    \begin{equation}\label{eq:matrix}
    \mathcal N_l=\begin{pmatrix}
        a+b& ab&0  &       &    0   &0&0      \\
        1 & a+b& ab &\cdots &   0   &0&0     \\
        0 & 1 & a+b &      &    0    &0&0       \\
          & \vdots& & \ddots &      & \vdots   &   \\
        0 & 0    &0&        &  a+b & ab &0 \\
         0 & 0    &0&  \cdots      &  1 & a+b& ab\\
          0 & 0    &0&        & 0  &  1&a+b
    \end{pmatrix}
    \end{equation}
    for $l\geq 1$. Then, the determinant of $\mathcal N_l$ is
    $
    \det \mathcal N_l = \displaystyle \sum_{i=0}^la^ib^{l-i}.
   $
\end{lemma}
\begin{proof}
The statement holds for $l=1,2$. For $l\ge3$, developing the determinant of \eqref{eq:matrix} according to the first column,
we get that 
 $$\det\mathcal N_l=(a+b)\det(\mathcal N_{l-1})-1\cdot ab\cdot\det(\mathcal N_{l-2})$$
Then, the result follows by induction.
\end{proof}


\begin{theorem}\label{theo: class of Nash}
Let $\mathcal S\subset[k]$ be the subset of isolated vertices and let $\widehat{x_i}:=\sum_{u\in[k]\setminus\{i\}}x_u$. Then, 
    for a generic game $X$, the class of $\Vc$ in $A[G]$ is the polynomial
    \begin{equation}
        \label{eq:poly chow}
        [\Vc]= \displaystyle\prod_{i\in\mathcal S}  \widehat{x_i}^{d_i-1} \,\prod_{i\not\in\mathcal S}\prod_{j=1}^{n_i}\left(    
        \sum_{l=0}^{d_{i,j}-1} x_i^l\left(\sum_{u=1}^kx_u\right)^{d_{i,j}-1-l}
        \right).
    \end{equation}
\end{theorem}
\begin{proof}
Let $X$ be a generic game.
We denote by $\mathcal{M}^{(i,l)}$ the variety defined by the $2\times2$ minors of the matrix $\widetilde{M}^{(i,l)}$ (Section~\ref{sec:equations}). Then, the Nash CI variety is the intersection $\mathcal{M}^{(1,1)}\cap\cdots \cap \mathcal{M}^{(k,n_k)}$. The expected codimension of $\mathcal{M}^{(i,l)}$ is $d_{i,l}-1$. By Theorem \ref{theo:dimension}, the codimension of $\Vc$ is $\sum_i\sum_l (d_{i,l}-1)$, which is the expected codimension of $\Vc$. Thus, the codimension of $\mathcal{M}^{(i,l)}$ is $d_{i,l}-1$. Therefore, we can apply  Porteous' formula (see \cite[Theorem 12.4]{eisenbud20163264}). Assume first that $(i,l)$ is not an isolated vertex. Then, the matrix $\widetilde{M}^{(i,l)}$ gives a map between the vector bundles $\mathcal{E}\rightarrow\mathcal{F}$ where
\[
\mathcal{E}=\ko_{\Mg}^{\oplus d_i} \text{ and } \mathcal{F}= \ko_{\Mg}(0,\ldots,\underset{(i)}{1},\ldots, 0)\oplus\ko_{\Mg}(1,\ldots,1).
\]
The total Chern classes in $A[G]$ of these vector bundles are $C(\mathcal{E})=1$ and 
\[\begin{array}{c}
\displaystyle C(\mathcal{F}) = C( \ko_{\Mg}(0,\ldots,\underset{(i)}{1},\ldots, 0))\cdot C(\ko_{\Mg}(1,\ldots,1))=(1+x_i)\left(1+\sum_{u=1}^k x_u\right)=\\ \displaystyle
1+ x_i+\sum_{u=1}^k x_u +x_i\left(\sum_{u=1}^k x_u\right).
\end{array}
\]
Using Porteous' formula (\cite[Theorem 12.4]{eisenbud20163264}) we get that 
\[
\left[\mathcal{M}^{(i,l)}\right]=\Delta_1^{d_{i,l}-1}\left( C(\mathcal{F})\right),
\]
where $\Delta_1^{d_{i,l}-1}\left( C(\mathcal{F})\right)$ is the determinant of $\mathcal N_{d_{i,j}-1}$ (see Equation \eqref{eq:matrix}) for $a=x_i$ and $b=\sum_{u\in[k]} x_u $. Then, by Lemma \ref{lemma:matrix}, we conclude that 
\[
\left[\mathcal{M}^{(i,l)}\right]=\displaystyle \sum_{i=0}^{d_{i,l}-1} x_i^i\left(\sum_{u=1}^k x_u\right)^{d_{i,l}-1-i}.
\]
Assume now that the player $(i,l)$ corresponds to an isolated vertex. We denote this vertex by $i$. Then, the matrix $M^{(i)}$ gives a map between the vector bundles $\mathcal{E}\rightarrow\mathcal{F}$ where
\[
\mathcal{E}=\ko_{\Mg}^{\oplus d_i} \text{ and } \mathcal{F}= \ko_{\Mg}\oplus\ko_{\Mg}(1,\ldots,\underset{(i)}{0},\ldots,1).
\]
The total Chern classes in $A[G]$ of these vector bundles are $C(\mathcal{E})=1$ and 
\[
C(\mathcal{F}) = C( \ko_{\Mg})\cdot C(\ko_{\Mg}(1,\ldots,\underset{(i)}{0},\ldots,1))= \displaystyle
1+\widehat{x_i}.
\]

\noindent Using Porteous' formula (\cite[Theorem 12.4]{eisenbud20163264}) we get that 
\[
\left[\mathcal{M}^{(i)}\right]=\Delta_1^{d_{i,l}-1}\left( C(\mathcal{F})\right),
\]
where $\Delta_1^{d_{i,l}-1}\left( C(\mathcal{F})\right)$ is the determinant of the $(d_i-1)\times (d_i-1)$--matrix
\[
\begin{pmatrix}
    \widehat{x_i}& 0 &\cdots & 0&0\\
    1 & \widehat{x_i}& \cdots &0&0\\
    \vdots & \vdots  & \ddots & \vdots &\vdots\\
    0 & 0 &\cdots &\widehat{x_i}&0 \\
    0&0&\cdots &1 & \widehat{x_i}
    
\end{pmatrix}.
\]
We deduce that 
\[
\left[\mathcal{M}^{(i)}\right]=\displaystyle\widehat{x_i}^{d_i-1}.
\] 
Since $\mathcal{M}^{(1,1)},\ldots,\mathcal{M}^{(k,n_k)}$ intersect in $\Vc$ with the expected codimension, we conclude that $\left[\Vc\right]$ is the polynomial in the statement.
\end{proof}

\noindent A first consequence of Theorem \ref{theo: class of Nash} is the computation of the degree of Nash CI varieties for generic games.

\begin{proposition}\label{prop:degree Nash}
    Let $\mathcal S$ be the set of isolated vertices of the graph $G$. Then, for a generic game $X$, the degree of the Nash CI variety $\Vc$ is the coefficient of the monomial $$x_1^{D_1-1}\cdots x_k^{D_k-1}$$
    in the polynomial 
    \[
    \displaystyle
    \left(
    \sum_{u=1}^kx_u
    \right)^{n-k+\sum_{i=1}^k(D_i-S_i)}
    \prod_{i\in\mathcal S}\widehat{x_i}^{d_i-1} \,\prod_{i\not\in\mathcal S}\prod_{j=1}^{n_i}\left(    
        \sum_{l=0}^{d_{i,j}-1} x_i^l\left(\sum_{u=1}^kx_u\right)^{d_{i,j}-1-l}
        \right)
    \]
\end{proposition}
\begin{proof}
    The degree of a $d$-dimensional variety $Z$ in $\Mg$ is given by the coefficient of 
    \[
    x_1^{D_1-1}\cdots x_k^{D_k-1}
    \]
    in the polynomial corresponding to the class
    \[
    \left[Z\right]\left(\sum_{u=1}^kx_u\right)^d.
    \]
    The proof follows by applying this method for $\Vc$ using Theorem \ref{theo: class of Nash} and Corollary \ref{cor:dim Nash}.
\end{proof}

\begin{remark}
    For binary games, the class of $\Vc$ in $A[G]$ for generic games is 
    \[
    \left[\Vc\right] = \displaystyle  \prod_{i\in\mathcal S}  \widehat{x_i} \,\prod_{i\not\in\mathcal S}\prod_{j=1}^{n_i}\left( x_i+\sum_{u=1}^kx_u
        \right),
    \]
    which coincides with the calculation of this class carried out in \cite{portakal2024game}. In particular, Proposition \ref{prop:degree Nash} coincides with \cite[Proposition 26]{portakal2024game} for binary cases. Similarly, when $G$ is complete, the class of $\Vc$ is 
    \[
    \left[ \Vc\right] = \prod_{j=1}^{n}\left(    
        \sum_{l=0}^{d_{i}-1} x_1^lx_1^{d_{i}-1-l}\right)=\prod_{j=1}^{n}\left(    
        d_i x_1^{d_{i}-1}
        \right)=\left(\prod_{j=1}^{n}    
        d_i \right)x_1^{d_{1}+\cdots+d_n-n}.   
    \]
    Therefore, we obtain that the degree of a generic Nash CI variety of a complete graph is $d_1\cdots d_n$. This coincides with \cite[Theorem 6]{portakal2022geometry}.
\end{remark}

\begin{example}\label{example:Nash3play}
    Consider a generic $d_1 \times d_2 \times d_3$ and let $G$ be a graph on $3$ vertices with one edge between the vertices $2$ and $3$. By Corollary \ref{cor:dim Nash}, if $\Vc$ is nonempty, then its dimension is
    \[
\dim \Vc = d_2d_3-d_2-d_3+1.
    \]
    In this setting, the Chow ring of $\Mg$ is 
    \[A[G]=\Z[x_1,x_2]/\langle x_1^{d_1},x_2^{d_2d_3}\rangle .
    \]
    By Theorem \ref{theo: class of Nash}, the class of $\Vc$ in $A[G]$ is
    \[
    \left[ \Vc\right] = x_2^{d_1-1}\left(    
        \sum_{l=0}^{d_{2}-1} x_2^l\left(x_1+x_2\right)^{d_{2}-1-l}\right)
        \left(\sum_{l=0}^{d_{3}-1} x_2^l\left(x_1+x_2\right)^{d_{3}-1-l}\right).
    \]
    By Proposition \ref{prop:degree Nash}, the degree of $\Vc$ is the coefficient of $x_1^{d_1-1}x_2^{d_2d_3-1}$ in the polynomial 
    \[
        (x_1+x_2)^{d_2d_3-d_2-d_3+1}x_2^{d_1-1}\left(    
        \sum_{l=0}^{d_{2}-1} x_2^l\left(x_1+x_2\right)^{d_{2}-1-l}\right)
        \left(\sum_{l=0}^{d_{3}-1} x_2^l\left(x_1+x_2\right)^{d_{3}-1-l}\right).
    \]
   We deduce the following closed formula (some of whose instances are computed in Table \ref{table:degrees and dims}):
    \begin{equation}\label{eq:explicit degree}
    \binom{d_2d_3+1}{d_1+1}-\binom{d_2d_3-d_2+1}{d_1+1}-\binom{d_2d_3-d_3+1}{d_1+1}+\binom{d_2d_3-d_2-d_3+1}{d_1+1}.
    \end{equation}

\noindent In certain cases, the Nash CI variety turns out to be empty, which specifically when $d_1 > d_2 d_3$ (see Example~\ref{ex: CI beats Nash}). Such a condition is equivalent to the vanishing of \eqref{eq:explicit degree}.
In the next section, we establish the necessary and sufficient conditions for this emptiness to occur in generic games, treating the problem in full generality.

\begin{table}
\begin{tabular}{|c|c|c|c|c|}
\hline
$d_1$ & $d_2$ & $d_3$ & Dimension & Degree \\ \hline
2 & 2 & 2 & 1 & 8 \\ \hline
2 & 2 & 3 & 2 & 21 \\ \hline
3 & 2 & 2 & 1 & 5 \\ \hline
2 & 3 & 3 & 4 & 54 \\ \hline
3 & 2 & 3 & 2 & 29 \\ \hline
3 & 3 & 3 & 4 & 141 \\ \hline
2 & 2 & 4 & 3 & 40 \\ \hline
4 & 2 & 2 & 1 & 1 \\ \hline
2 & 3 & 4 & 6 & 102 \\ \hline
3 & 2 & 4 & 3 & 86 \\ \hline
4 & 2 & 3 & 2 & 20 \\ \hline
2 & 4 & 4 & 9 & 192 \\ \hline
5 & 2 & 2 & empty & empty \\ \hline
2 & 2 & 5 & 4 & 65 \\ \hline
2 & 3 & 5 & 8 & 165 \\ \hline
7 & 2 & 3 & empty & empty \\ \hline
\end{tabular}
\caption{Dimensions and degrees of Nash CI varieties for generic $d_1 \times d_2 \times d_3$ games where the graph's only edge connects players $2$ and $3$.}
\label{table:degrees and dims}
\end{table}
\end{example}
\section{Emptiness of generic Nash CI varieties}\label{sec:Emptyness}

In this section, we provide a characterization of the emptiness of generic Nash CI varieties. The strategy is to use Theorem \ref{theo: class of Nash} and show when the class \eqref{eq:poly chow} vanishes in the Chow ring $A[G]$. To make it so, we first derive several technical lemmas.

\begin{lemma}\label{lemma:product}
    The monomials of the product 
    \begin{equation}\label{eq:prod chow poly}
    \prod_{i\in\mathcal S}\widehat{x_i}^{d_i-1}
    \end{equation}
    are all the monomials $x^\alpha$ of degree $\displaystyle\sum_{i\in\mathcal S}(d_i-1)$ with $\displaystyle\alpha_i\leq \sum_{j\in\mathcal S\setminus\{i\}}(d_j-1)$ for all $i\in\mathcal S$.
\end{lemma}
\begin{proof}
    Without loss of generality, we may assume that $\mathcal S=[s]$ for $s=|\mathcal S|$. We argue by induction on $s$. For $s=1$, \eqref{eq:prod chow poly} equals $\left(\widehat{x_1}\right)^{d_1-1}$ and the statement holds. Assume now that $s\geq 2$. We write the product \eqref{eq:prod chow poly}  as 
    \begin{equation}
        \label{eq: div prod}
         \left(\widehat{x_s}\right)^{d_s-1}\prod_{i=1}^{s-1}\left(\widehat{x_i}\right)^{d_i-1}.
    \end{equation}
    The monomials of the first factor of \eqref{eq: div prod} are all monomials $x^{\tilde{\alpha}}$ of degree $d_s-1$ with $\tilde{\alpha}_s=0$. 
    By induction, the monomials of the second factor of \eqref{eq: div prod} are all monomials $x^{\overline{\alpha}}$ with 
    \begin{equation}\label{eq:mono cond}
    |\overline{\alpha}|=\sum_{i=1}^{s-1}(d_i-1)~~ \text{ and } ~~\overline{\alpha}_i\leq \sum_{j\in[s-1]\setminus\{i\}}(d_j-1) ~~\text{ for every } i\in[s-1]. 
    \end{equation}
    Let $\alpha$ as in the statement. We need to show that every $\alpha$ as in the statement is the sum of $\tilde{\alpha}+\overline{\alpha}$ as above.  For $i\in [k]$, consider the integer vector $\beta$ such that 
    \[
    \beta_i=\left\{
    \begin{array}{ll}
    \displaystyle\min\bigg\{~\alpha_i~,~\sum_{j\in[s-1]\setminus\{i\}}(d_j-1)\bigg\} & \text{ for } i\in[s-1],\\
    \alpha_i & \text{ for } i\geq s.
    \end{array}
    \right.
    \]
    Next, we show that $\displaystyle|\beta|:=\sum_{i=1}^k\beta_i\geq \sum_{i=1}^{s-1}(d_i-1)$. Consider the subset $\mathcal A=\{i\in[s-1]:\beta_i<\alpha_i\}$.
    
    We distinguish $3$ cases:
    \begin{itemize}
        \item If $\mathcal A=\emptyset$, then $$|\beta|=|\alpha|=\sum_{i=1}^s(d_i-1)\geq\sum_{i=1}^{s-1}(d_i-1).$$ 
        \item If $|\mathcal A|=\{i\}$, then 
        \[|\beta|=\sum_{j\in[s-1]\setminus\{i\}}(d_j-1) +\sum_{j\in[k]\setminus\{i\}}\alpha_j = 
        \sum_{j=1}^{s-1}(d_j-1) + \sum_{j=1}^s(d_j-1)-\alpha_i -d_i+1\geq \sum_{j=1}^{s-1}(d_j-1).
        \]
        Here we have used that $\displaystyle|\alpha|=\sum_{j=1}^s(d_j-1)$ then $\displaystyle\alpha_i\leq \sum_{j\in[s]\setminus\{i\}}(d_j-1)$.
        \item If there exists $i_1,i_2\in\mathcal A$ distinct, then
        \[|\beta|\geq \beta_{i_1}+\beta_{i_2} = \sum_{j\in[s-1]\setminus\{i_1\}}(d_j-1) +  \sum_{j\in[s-1]\setminus\{i_2\}}(d_j-1) \geq \sum_{j=1}^{s-1}(d_j-1).
        \]
    \end{itemize}
    We deduce that $|\beta|\geq \sum_{i=1}^{s-1}(d_i-1)$. Therefore, there exists a choice of $0\leq \overline{\alpha}_i\leq \beta_i$ for $i\in[k]$ such that $\overline{\alpha}_s = \beta_s=\alpha_s$ and $|\overline{\alpha}|=\sum_{i=1}^{s-1}(d_i-1)$. Consider also the integer vector $\tilde{\alpha}:=\alpha-\overline{\alpha}$. By construction $\overline{\alpha}$ satisfies \eqref{eq:mono cond} and $|\tilde{\alpha}|=d_s-1$ with $0\leq\tilde{\alpha}_i$. Therefore, $\alpha=\overline{\alpha}+\tilde{\alpha}$ is the desired decomposition.
\end{proof}

\begin{lemma}\label{lemma: monomials}
    The nonzero monomials that appear in the polynomial \eqref{eq:poly chow} are all the monomials $x_1^{\alpha_1}\cdots x_k^{\alpha_k}$ of degree $\displaystyle\sum_{l=1}^k S_l-n$ such that for $i\in\mathcal S$, $\displaystyle\alpha_i\leq \sum_{l=1}^kS_l -n+1-d_i$.
\end{lemma}
\begin{proof}
First, we look at the monomials that appear in each of the factors of the polynomial \eqref{eq:poly chow}. By Lemma \ref{lemma:product}, the nonzero monomials of the factor \eqref{eq:prod chow poly} are all the monomials $x^\alpha$ of degree $\sum_{i\in\mathcal S}(d_i-1)$ with $\alpha_i\leq \sum_{j\in\mathcal S}(d_j-1)-(d_i-1)$.
Now, the nonzero monomials of the factor 
\[ \displaystyle
        \sum_{l=0}^{d_{i,j}-1} x_i^l\left(\sum_{u=1}^kx_u\right)^{d_{i,j}-1-l}
        \]
are all monomials of degree $d_{i,j}$. In particular, the nonzero monomials of the product 
\[
 \displaystyle\prod_{i\not\in\mathcal S}\prod_{j=1}^{n_i}\left(    
        \sum_{l=0}^{d_{i,j}-1} x_i^l\left(\sum_{u=1}^kx_u\right)^{d_{i,j}-1-l}
        \right)
\]
are all monomials of degree $\sum_{i\not\in\mathcal S}(S_i-n_i)$. We conclude that the nonzero monomials of the polynomial \eqref{eq:poly chow} are all monomials of degree $\sum_{i\in[k]}S_i -n$ whose degree in $x_i$ for $i\in\mathcal S$ is smaller than or equal to
\[\sum_{i=1}^kS_i -n-d_i+1.\qedhere\]
\end{proof}

\begin{lemma}\label{lemma: inequality ds}
    For every $i\in[k]$, it holds that $S_i-n_i\le D_i-1$.
\end{lemma}
\begin{proof}
We use induction on $n_i$. For $n_i=1$, then $S_i-n_i=d_i-1=D_i-1$. Now, assume that $n_i>1$. The inequality $S_i-n_i\le D_i-1$ is equivalent to the inequality
    $$\displaystyle1+S_i-n_i=1+\sum_{j=1}^{n_i}(d_{i,j}-1)\le\prod_{j=1}^{n_i}d_{i,j}=D_i.$$
    Then, by induction we get that 
    \[
    \begin{array}{c}
    \displaystyle 1+(d_{i,n_i}-1)+\sum_{j=1}^{n_{i}-1}(d_{i,j}-1)\leq 
    1+(d_{i,n_i}-1)+\sum_{j=1}^{n_i-1}(d_{i,j}-1)+(d_{i,n_i}-1)\left(\sum_{j=1}^{n_i-1}(d_{i,j}-1)\right)\\ \displaystyle
   = d_{i,n_i}\left(\sum_{j=1}^{n_i-1}(d_{i,j}-1)+1\right)\leq d_{i,n_i}\prod_{j=1}^{n_i-1}d_{i,j}=D_i.
    \end{array}\]
\end{proof}

\begin{theorem}\label{theo:emptyness}
    Let $\mathcal S$ be the set of isolated vertices of $G$. Then, the Nash CI variety is nonempty for generic games if and only if
    \begin{equation}\label{eq: final nonempty}
    d_i\leq1+\frac12 \sum_{l=1}^k(D_l-1)
    \end{equation}
    for every $i\in\mathcal S$.
\end{theorem}
\begin{proof}
The Nash CI variety of a generic game $X$ is nonempty if and only if the class $\left[ \Vc\right]$ is nonzero. Thus, it is enough to show that the polynomial \eqref{eq:poly chow} is nonzero if and only if \eqref{eq: final nonempty} holds.
    To do so, we distinguish two cases: either  for all $ i\in\mathcal S$, $2d_i-2\leq \sum_{l=1}^k(S_l-n_l)$, or there exists $ i\in\mathcal S$ such that $2d_i-2>\sum_{l=1}^k(S_l-n_l)$.

Assume first that $2d_i-2\leq \sum_{l=1}^k(S_l-n_l)$ for all $i\in\mathcal S$. By Lemma \ref{lemma: inequality ds},  we deduce that \eqref{eq: final nonempty} holds. Therefore, in this case, to check the equivalence it is enough to check that the class \eqref{eq:poly chow} is nonzero.
Consider the monomial $\prod_{i=1}^kx_i^{S_i-n_i}$. This monomial has degree $\sum_{l=1}^k S_l-n$ and satisfies that for all $i\in\mathcal S$  $$\alpha_i=S_i-n_i=d_i-1\leq \sum_{l=1}^kS_l -n+1-d_i.$$ 
    Hence, by Lemma \ref{lemma: monomials}, this monomial appears in the polynomial \eqref{eq:poly chow}. Moreover,  such monomial is nonzero since $S_i-n_i\le D_i-1$ by Lemma \ref{lemma: inequality ds}. We conclude that the polynomial \eqref{eq:poly chow} is nonzero.

  Assume now that there exists $a\in\mathcal S$ such that $2d_a-2>\sum_{l=1}^k(S_l-n_l)$. Without loss of generality, we assume $d_a$ to be the maximum.  Then, for all $j\in\mathcal S\setminus \{a\}$, we have that 
    \begin{equation}\label{eq: unique da}
   2d_j-2\le d_j-1+d_a-1\le\sum_{l\in\mathcal S}(d_l-1)=\sum_{l\in\mathcal S}(S_l-n_l)\le\sum_{l=1}^k(S_l-n_l).
    \end{equation}
    Now, let $x^{\alpha}=x_1^{\alpha_1}\cdots x_k^{\alpha_k}$ be a monomial of the polynomial \eqref{eq:poly chow}. Then by Lemma \ref{lemma: monomials},
    \begin{equation}
        \label{eq:degree eq}
         \alpha_1+\cdots+\alpha_k=\sum_{l=1}^k S_l-n, \text{ and }\alpha_i\leq \sum_{l=1}^k S_l-n-d_i+1 \text{ for } i\in\mathcal S.
    \end{equation}
    Note that $x^\alpha$ is nonzero if and only if $\alpha_i\leq D_i-1$ for every $i$.
    Moreover, using \eqref{eq:degree eq} we get \begin{equation}\label{eq:alphas}
    \sum_{l\not\in\mathcal S}\alpha_l=\sum_{l=1}^k S_l-n-\sum_{l\in\mathcal S}\alpha_l=\sum_{l\not\in\mathcal S} (S_l-n_l)+\sum_{l\in\mathcal S}(d_l-1)-\sum_{l\in\mathcal S}\alpha_l.
    \end{equation}
    Assume first that \eqref{eq: final nonempty} holds. Then, for $i\in\mathcal S$, we fix the $i$--th coordinate of $\alpha$ to be
    \begin{equation}
        \label{eq: some alpha}
        \alpha_i=\min\big\{d_i-1,\sum_{l=1}^k S_l-n-d_i+1\big\}=\left\{\begin{array}{ll}
             d_i-1&\text{ if }i\ne a\text{ by \eqref{eq: unique da}}\\
             \displaystyle\sum_{l=1}^k S_l-n-d_i+1&\text{ if }i=a.
        \end{array}\right.
    \end{equation}
        Then, we get that  
        \begin{equation}\label{eq: sum of alphas}
\begin{array}{lll}
  \displaystyle  \sum_{l\not\in\mathcal S}\alpha_l&  \displaystyle =\sum_{l\not\in\mathcal S} (S_l-n_l)+\sum_{l\in\mathcal S}(d_l-1)-\sum_{l\in\mathcal S}\alpha_l&\text{ by \eqref{eq:alphas}}\\
  &  \displaystyle =\sum_{l\not\in\mathcal S} (S_l-n_l)+d_a-1-\big(\sum_{l=1}^k S_l-n-d_a+1\big)&\text{ by \eqref{eq: some alpha}}\\
  &  \displaystyle =2d_a-2-\sum_{l\in\mathcal S}(d_l-1).
\end{array}
        \end{equation}
        
\noindent Using \eqref{eq: final nonempty} and \eqref{eq: sum of alphas}, we deduce that
    \begin{equation*}
    \left\{\begin{array}{ll}
    \displaystyle\sum_{l\notin\mathcal S}\alpha_l\ge2d_a-2-\sum_{l=1}^k(S_l-n_l)>0\\
        \displaystyle\sum_{l\not\in\mathcal S}\alpha_l\le\sum_{l=1}^k(D_l-1)-\sum_{l\in\mathcal S}(d_l-1)=\sum_{l\notin\mathcal S}(D_l-1).
    \end{array}\right.
    \end{equation*}
    Hence, there exists a choice of $\alpha_i$ for $i\not\in\mathcal S$ such that $0\leq\alpha_i\leq D_i-1$. Since $\alpha_i\leq d_i-1=D_i-1$ for $i\in\mathcal S$ by \eqref{eq: some alpha}, we deduce that the polynomial \eqref{eq:poly chow} is nonzero since its monomial $x^\alpha$ is a non zero monomial.

Conversely, assume that \eqref{eq: final nonempty} does not hold. Then, \begin{equation}\label{eq:equivalent to final criterion}
        \displaystyle\sum_{l=1}^k(D_l-1)<2d_a-2.
    \end{equation}
    We claim that the polynomial \eqref{eq:poly chow} vanishes. Let $x^\alpha$ be a monomial of \eqref{eq:poly chow}.
    Note that the monomial $x^\alpha$ is zero if there exists $i\in\mathcal S$ such that $\alpha_i\ge d_i$. Hence, we may assume that $\alpha_i\leq d_i-1$ for every $i\in\mathcal S$. By Lemma \ref{lemma: monomials}, we may assume that 
    \begin{equation}\label{eq:ineq alphas}
    \alpha_i\leq \min\{d_i-1,\sum_{l=1}^k S_l-n-d_i+1\}=\left\{\begin{array}{ll}
             d_i-1&\text{ if }i\ne a\text{ by \eqref{eq: unique da}}\\
             \displaystyle\sum_{l=1}^k S_l-n-d_i+1&\text{ if }i=a
        \end{array}\right.
    \end{equation}
    for all $i\in\mathcal S$. Using this inequality we deduce that 
    \begin{align*}
        \sum_{l\not\in\mathcal S}\alpha_l&=\sum_{l\not\in\mathcal S} (S_l-n_l)+\sum_{l\in\mathcal S}(d_l-1)-\sum_{l\in\mathcal S}\alpha_l&\text{ by \eqref{eq:alphas}}\\
        &\geq \sum_{l\not\in\mathcal S} (S_l-n_l)+d_a-1-\big(\sum_{l=1}^kS_l-n-d_a+1\big)&\text{ by \eqref{eq:ineq alphas}}\\
        &=2d_a-2-\sum_{l\in\mathcal S}(d_l-1)>\sum_{l\notin\mathcal S}(D_l-1).&\text{ by \eqref{eq:equivalent to final criterion}}
    \end{align*}
    By Pigeonhole Principle, there exists $i\not\in\mathcal S$ such that $\alpha_i>D_i-1$, and the monomial $x^\alpha$ vanishes. We conclude that the polynomial \eqref{eq:poly chow} is zero.
\end{proof}

\noindent Theorem \ref{theo:emptyness} generalizes the classical result about Nash equilibria explained in Remark~\ref{rem: nash emptyness}. Using Theorem \ref{theo:emptyness}, we can identify games for which there is no totally mixed Nash equilibrium (the Nash CI variety is empty for the graph with no edge), but there exists some totally mixed Nash CI equilibrium. The following example provides an instance of such a situation.

\begin{example}\label{ex: CI beats Nash}
    We remain in the setting of Example \ref{example:Nash3play}: $3$ players, player $2$ and $3$ being connected by the only edge. Using Theorem \ref{theo:emptyness}, we obtain that the Nash CI variety is nonempty for generic games if and only if $d_1\le d_2d_3$. This implies the emptiness results from Table \ref{table:degrees and dims}.
    The classical criterion for totally mixed Nash equilibria (see Remark \ref{rem: nash emptyness}) is more restrictive:
    \[
    d_1\le d_2+d_3-1\text{ and }d_2\le d_1+d_3-1\text{ and }d_3\le d_1+d_2-1.
    \]
    Hence, for certain values of $d_1,d_2,d_3$ we can find a Nash CI equilibrium whereas there is no totally mixed Nash equilibrium. For instance, for $d_1=4$ and $d_2=d_3=2$ and for a generic game, the variety corresponding to the totally mixed Nash equilibria is empty, but the Nash CI variety of the one-edge graph is a curve of degree $1$ (cf. Table \ref{table:degrees and dims}).

In this case, $\mathcal M_G$ is the Segre variety $\PP^3\times\PP^3$. By Proposition \ref{prop:eq Nash CI}, $\Vc$ is defined by $5$ equations: three of them are generic polynomials in $|\mathcal{O}_{\mathcal M_G}(0,1)|$  and the other two have bidegree $(1,2)$. The three polynomials of bidegree $(0,1)$ define a point in the second factor of the Segre variety. Substituting such point in the remaining two equations, we conclude that the Spohn CI variety of a generic game is a line in $\PP^3$. Together with the condition of lying in the probability simplex, we can write these equations as
\begin{equation*}
\left\{\begin{array}{cc}
     & \begin{pmatrix}
    X_{111}^{(1)}-X_{211}^{(1)} & X_{112}^{(1)}-X_{212}^{(1)} & X_{121}^{(1)}-X_{221}^{(1)} & X_{122}^{(1)}-X_{222}^{(1)} \\
    X_{111}^{(1)}-X_{311}^{(1)} & X_{112}^{(1)}-X_{312}^{(1)} & X_{121}^{(1)}-X_{321}^{(1)} & X_{122}^{(1)}-X_{322}^{(1)} \\
    X_{111}^{(1)}-X_{411}^{(1)} & X_{112}^{(1)}-X_{412}^{(1)} & X_{121}^{(1)}-X_{421}^{(1)} & X_{122}^{(1)}-X_{422}^{(1)} \\
    1 & 1 & 1 & 1
\end{pmatrix}\cdot\begin{pmatrix}
    \sigma_{11}^{(2)}\\
    \sigma_{12}^{(2)}\\
    \sigma_{21}^{(2)}\\
    \sigma_{22}^{(2)}
\end{pmatrix}=\begin{pmatrix}
    0\\
    0\\
    0\\
    1
\end{pmatrix} \\
    & \\
     & \begin{pmatrix}
        \big(\sigma_{2+}^{(2)}\cdot(X^{(2)}_{x11}\sigma_{11}^{(2)}+X^{(2)}_{x12}\sigma_{12}^{(2)})-\sigma_{1+}^{(2)}\cdot(X^{(2)}_{x21}\sigma_{21}^{(2)}+X^{(2)}_{x22}\sigma_{22}^{(2)})\big)_{1\le x\le4}\\
        \big(\sigma_{+2}^{(2)}\cdot(X^{(3)}_{x11}\sigma_{11}^{(2)}+X^{(3)}_{x21}\sigma_{21}^{(2)})-\sigma_{+1}^{(2)}\cdot(X^{(3)}_{x12}\sigma_{12}^{(2)}+X^{(3)}_{x22}\sigma_{22}^{(2)})\big)_{1\le x\le4}\\
        (1)_{1\le x\le4}
     \end{pmatrix}\cdot\begin{pmatrix}
    \sigma_{1}^{(1)}\\
    \sigma_{2}^{(1)}\\
    \sigma_{3}^{(1)}\\
    \sigma_{4}^{(1)}
\end{pmatrix}=\begin{pmatrix}
    0\\
    0\\
    1
\end{pmatrix}
\end{array}\right.
\end{equation*}
where $\sigma^{(1)}_i$ for $i\in[4]$ and $\sigma^{(2)}_{ij}$ for $i,j\in[2]$ are the coordinates of the first  and second factor of $\mathcal M_G$ respectively.
The existence of a (totally mixed) Nash CI equilibrium is equivalent to the existence of some positive $(\sigma_i^{(1)})_{1\le i\le4}$ and $(\sigma_{i,j}^{(2)})_{1\le i,j\le2}$ verifying the (linear) equations.
We set the payoff tensors as follows:

\[
\begin{matrix}
    X^{(1)}_{111}=1;&X^{(1)}_{112}=0;&X^{(1)}_{121}=1;&X^{(1)}_{122}=15;&X^{(1)}_{211}=0;&X^{(1)}_{212}=0;&X^{(1)}_{221}=0;&X^{(1)}_{222}=20;\\
    X^{(1)}_{311}=-1;&X^{(1)}_{312}=7;&X^{(1)}_{321}=2;&X^{(1)}_{322}=11;&X^{(1)}_{411}=0;&X^{(1)}_{412}=2;&X^{(1)}_{421}=0;&X^{(1)}_{422}=18;\\
    X^{(2)}_{111}=1;&X^{(2)}_{112}=0;&X^{(2)}_{121}=0;&X^{(2)}_{122}=0;&X^{(2)}_{211}=1;&X^{(2)}_{212}=3;&X^{(2)}_{221}=0;&X^{(2)}_{222}=2;\\
    X^{(2)}_{311}=9;&X^{(2)}_{312}=1;&X^{(2)}_{321}=0;&X^{(2)}_{322}=30;&X^{(2)}_{411}=4;&X^{(2)}_{412}=4;&X^{(2)}_{421}=0;&X^{(2)}_{422}=5;\\
    X^{(3)}_{111}=1;&X^{(3)}_{112}=0;&X^{(3)}_{121}=1;&X^{(3)}_{122}=0;&X^{(3)}_{211}=1;&X^{(3)}_{212}=0;&X^{(3)}_{221}=0;&X^{(3)}_{222}=10;\\
    X^{(3)}_{311}=9;&X^{(3)}_{312}=0;&X^{(3)}_{321}=10;&X^{(3)}_{322}=3;&X^{(3)}_{411}=4;&X^{(3)}_{412}=0;&X^{(3)}_{421}=0;&X^{(3)}_{422}=5
\end{matrix}
\]
For these payoff tensors, the only possible value of $(\sigma^{(2)}_{i,j})_{i,j\in[2]}$ is $(\frac8{21};\frac17;\frac13;\frac17)$.
Note that this implies that there is no totally mixed Nash equilibrium. Indeed, a totally mixed Nash equilibrium must be contained in the set of Nash CI equilibria and it must satisfy the extra condition $\sigma_{11}^{(2)}\sigma_{22}^{(2)}=\sigma_{12}^{(2)}\sigma_{21}^{(2)}$, which is not satisfied by the point $(\frac8{21};\frac17;\frac13;\frac17)$.
The equation of the Nash CI variety is 
\[\left\{\begin{array}{c}
-320 \sigma^{(1)}_1-416\sigma^{(1)}_2+960\sigma^{(1)}_3-1100\sigma^{(1)}_4 = 0,\vspace*{2mm}

\\ 
-360 \sigma^{(1)}_1+1608\sigma^{(1)}_2 -2868\sigma^{(1)}_3+132\sigma^{(1)}_4 = 0.

\end{array}\right.
\]
Adding the equation $\sigma^{(1)}_1+\sigma^{(1)}_2+\sigma^{(1)}_3+\sigma^{(1)}_4=1$, we get the parametric equations 
\[
\left(\sigma^{(1)}_1,\sigma^{(1)}_2,\sigma^{(1)}_3,\sigma^{(1)}_4\right) = \left(\frac{913}{5933} - \frac{106095}{47464}\,t,\,\, \frac{3290}{5933} + \frac{27635}{47464}\,t,\,\,\frac{1730}{5933} + \frac{7749}{11866}\,t,\,\,t \right) \text{ for }
 t \in \mathbb{R}.
\]
The intersection with the interior of the probability simplex is parametrized by the values of $t$ in the interval $\left(0,\tfrac{664}{9645}\right)$. 
\end{example}

Theorem \ref{theo:emptyness} provides a criterion for the emptiness of generic Nash CI varieties. For generic Spohn CI varieties of a general graph, the question remains open. Nevertheless, we derive from Theorem \ref{theo:emptyness} necessary conditions for a generic Spohn CI variety of a general graph to be nonempty. 

\begin{lemma}[{\cite[Lemma $14$]{portakal2024game}}]\label{lem:inclusion} If $G$ is a subgraph of $G'$ with the same vertices, then for any payoff tensors $X$, we have $\mathcal V_{X,G'}=\emptyset\Longrightarrow\mathcal V_{X,G}=\emptyset$.
\end{lemma}

The following example shows that Lemma \ref{lem:inclusion} may be applied to the study of the emptiness of Spohn CI varieties of generic graphs.

\begin{example}\label{last example}
We consider the $4$-vertex graph $G$ with two edges connecting vertex $3$ to the vertices $2$ and $4$, and fix $d_2=d_3=d_4=2$.
Such a graph does not lead to a Nash CI variety since $G$ is not a disjoint union of complete graphs. However, we may use  Lemma \ref{lem:inclusion} combined with Theorem \ref{theo:emptyness} to study the emptyness of a generic Spohn CI variety associated to $G$.

First, for $d_1\leq 5$, we may consider the subgraph $G'$ of $G$ obtained by removing an edge of $G$. Applying Theorem \ref{theo:emptyness}, one deduces that $\mathcal V_{X,G'}$ is nonempty for a generic game $X$. By Lemma \ref{lem:inclusion}, we deduce that  $\mathcal V_{X,G}$ is nonempty for a generic game $X$.

Now, assume that $d_1>8$ and let $G'$ be the graph obtained by adding the edge connecting the vertices $2$ and $4$ to $G$. Then, $G'$ is a cluster graph and by Theorem \ref{theo:emptyness}, the Nash CI variety of a generic game $X$ associated to $G'$ is empty. By Lemma \ref{lem:inclusion}, we deduce that $\mathcal V_{X,G}$ is empty for generic $X$.
For the remaining case $d_1\in\{6,7,8\}$, the emptiness or not emptiness of a generic Spohn CI variety is not known.
\end{example}

\begin{corollary}\label{cor:emptiness}
Let $G$ be a graph with $n$ vertices. Then, the Spohn CI variety of a generic game $X$ and the graph $G$ is nonempty if one of the following conditions hold:
\begin{itemize}
    \item The vertex of $G$ standing for the player with highest number of pure strategies is not isolated.
    \item The graph $G$ has no isolated vertex.
    \item The graph $G$ is connected.
\end{itemize}
\end{corollary}
\begin{proof}
    Note that  the second and third condition implies the first condition. Therefore, we may assume that the first condition holds and let $a$ be the vertex with highest number of strategies. Then, there exists a subgraph $G'$ of $G$ such that $G':=G_1\sqcup G_2\sqcup\ldots\sqcup G_k$ is a cluster graph whose connected components are $G_1,\ldots,G_k$ and $a$ is not an isolated vertex. Without loss of generality, assume that $a$ is a vertex of $G_1$. Let $i$ be an isolated vertex of $G'$. Then, since $a$ is not isolated, $D_1\geq 2d_a$. We deduce that  
    \[
    \sum_{l=1}^k(D_l-1) \geq 2d_a-1\geq 2d_i-1\geq 2d_i-2.
    \]
    In other words, \eqref{eq: final nonempty} holds. By Theorem \ref{theo:emptyness}, we deduce that the generic Nash CI variety of $G'$ is nonempty. By Lemma \ref{lem:inclusion}, we conclude that the generic Spohn CI variety of $G$ is nonempty.
\end{proof}

\begin{example}\label{ex 2 2-2}
    In \cite{spohn2007dependency}, Spohn shows that
    dependency equilibria can be better than Nash equilibria in the sense of Pareto. In this example, we show that this is also the case of (totally mixed) CI equilibria.
    We keep the case of Example \ref{example:Nash3play}: $3$ players, player $2$ and $3$ being connected by the only edge. We assume $d_1=d_2=d_3=2$ and we set the payoff tensors as 
    \begin{align*}
    X^{(1)}_{111}=0;X^{(1)}_{112}=4;X^{(1)}_{121}=X^{(1)}_{122}=0;&~X^{(1)}_{211}=3;X^{(1)}_{212}=X^{(1)}_{221}=0;X^{(1)}_{222}=1;\\
    X^{(2)}_{111}=2;X^{(2)}_{112}=0;X^{(2)}_{121}=3;X^{(2)}_{122}=1;&~X^{(2)}_{211}=3;X^{(2)}_{212}=-1;X^{(2)}_{221}=4;X^{(2)}_{222}=0;\\
    X^{(3)}_{111}=2;X^{(3)}_{112}=3;X^{(3)}_{121}=0;X^{(3)}_{122}=1;&~X^{(3)}_{211}=3;X^{(3)}_{212}=4;X^{(3)}_{221}=-1;X^{(3)}_{222}=0.
    \end{align*}
    
For players $2$ and $3$, strategy $2$ strictly dominates strategy $1$, hence the only Nash equilibrium is each player choosing pure strategy $2$. The expected payoffs are respectively $1,0,0$.

In this case, the variety $N_{X,G}$ is the subvariety of $\mathcal{M}_G=\PP^1\times \PP^3$ given by the minors of the matrices
\begin{equation*}
\begin{array}{c}
\widetilde{M}^{(1)} = \begin{pmatrix}
    1 & 4\sigma^{(2)}_{12}\\
    1 & 3\sigma^{(2)}_{11}+\sigma^{(2)}_{22}
\end{pmatrix},\,\,\,
\widetilde{M}^{(2)} = \begin{pmatrix}
    \sigma^{(2)}_{11}+\sigma^{(2)}_{12} & 2\sigma^{(1)}_1\sigma^{(2)}_{11} + 3\sigma^{(1)}_2\sigma^{(2)}_{11} -\sigma^{(1)}_2\sigma^{(2)}_{12} \\
        \sigma^{(2)}_{21}+\sigma^{(2)}_{22} & 3\sigma^{(1)}_1\sigma^{(2)}_{21} + 4\sigma^{(1)}_2\sigma^{(2)}_{21} +\sigma^{(1)}_1\sigma^{(2)}_{22} \\
\end{pmatrix},\\  \\ \widetilde{M}^{(3)} = \begin{pmatrix}
    \sigma^{(2)}_{11}+\sigma^{(2)}_{21} & 2\sigma^{(1)}_1\sigma^{(2)}_{11} + 3\sigma^{(1)}_2\sigma^{(2)}_{11} -\sigma^{(1)}_2\sigma^{(2)}_{21} \\ \\
        \sigma^{(2)}_{12}+\sigma^{(2)}_{22} & 3\sigma^{(1)}_1\sigma^{(2)}_{12} + 4\sigma^{(1)}_2\sigma^{(2)}_{12} +\sigma^{(1)}_1\sigma^{(2)}_{22} \\
\end{pmatrix}.  
\end{array}
\end{equation*}

Here $\sigma^{(1)}_j$ and $\sigma^{(2)}_{j_1j_2}$ are the coordinates of $\PP^1$ and $\PP^3$ respectively. A \texttt{Macaulay2} (\cite{connelly2025gametheory}) computation shows that $N_{X,G}$ has dimension $1$ and degree $8$. Its irreducible components are 
\begin{equation}\label{eq:irred comp example}
\begin{array}{c}
\vspace*{1mm}
\left\langle \!
    \begin{array}{c} 13\sigma^{(1)}_1(\sigma^{(2)}_{21})^2\!+\!19\sigma^{(1)}_2(\sigma^{(2)}_{21})^2\!-\!2 \sigma^{(1)}_1\sigma^{(2)}_{21}\sigma^{(2)}_{22}\!-\!10\sigma^{(1)}_2\sigma^{(2)}_{21}\sigma^{(2)}_{22}\!+\!\sigma^{(1)}_1\sigma^{(2)}_{22} \!+\! 3 \sigma_2^{(1)} (\sigma_{22}^{(2)})^{2},
    \\ \sigma^{(2)}_{12}-\sigma^{(2)}_{21},3\sigma^{(2)}_{11}-4\sigma^{(2)}_{21}+\sigma^{(2)}_{22} 
   \end{array}\!\right \rangle,
         \\
         \vspace*{1mm}
     \langle     \sigma^{(1)}_1+\sigma^{(2)}_2,3\sigma^{(2)}_{11}-4\sigma^{(2)}_{12}+\sigma^{(2)}_{22},3\sigma^{(2)}_{12}\sigma^{(2)}_{21}-4\sigma^{(2)}_{12}\sigma^{(2)}_{22}+(\sigma^{(2)}_{22})^2
     \rangle ,
    
       \\
     \vspace*{1mm}
     \langle \sigma^{(2)}_{12}-\sigma^{(2)}_{22},\sigma^{(2)}_{11}-\sigma^{(2)}_{22},2\sigma^{(1)}_1\sigma^{(2)}_{21}+3\sigma^{(1)}_2\sigma^{(2)}_{21}-\sigma^{(1)}_2\sigma^{(2)}_{22}\rangle,
    
    \\ \vspace*{1mm}
      \langle \sigma^{(2)}_{11},\sigma^{(2)}_{12},\sigma^{(2)}_{22}\rangle.  
      
\end{array}
\end{equation}

We see that the second and last components in \eqref{eq:irred comp example} of $N_{X,G}$ are contained in the hyperplanes that we remove from the Nash CI variety. We deduce that in this case, $N_{X,G}$ and $\mathcal{V}_{X,G}$ do not coincide since the game it is not generic enough to satisfy Proposition \ref{prop:eq Nash CI}. In this case, $\mathcal{V}_{X,G}$ is obtained by removing the second and last components in \eqref{eq:irred comp example} from $N_{X,G}$. In particular, $\mathcal{V}_{X,G}$ is a Nash CI curve of degree $5$. We refer to \cite{portakal2024nash} for further details on Nash CI curves. 
The first component of \eqref{eq:irred comp example} does not intersect the probability simplex.
The intersection of the third component in \eqref{eq:irred comp example} with the probability simplex describes the totally mixed CI equilibria and it is parametrized by

\[\begin{array}{ccccc}
\sigma_{11}^{(2)}=\sigma_{12}^{(2)}=\sigma_{22}^{(2)}= t, &
\sigma_{21}^{(2)}= 1-3t, &
\sigma^{(1)}_{1}= \frac{10t-3}{4t-1}, & \text{ and } &
\sigma^{(1)}_{2}=\frac{2-6t}{4t-1}
\end{array}
\]
for $\frac{3}{10}<t<\frac{1}{3}$. Evaluating this parametrization at the expected payoffs of the players we get $4t$ for the first player, $1$ for the second, and $2$ for the third player. We deduce that for any value of $t$ we get a strictly better expected payoff than the only Nash equilibrium.

\end{example}

\appendix
\section{Degree formula for certain CI models}\label{appendix: degree CI}

In Section \ref{sec:dim}, the dimension of the Spohn CI variety for generic $X$ is computed for any graph. In Section \ref{sec:nash ci}, the degree of the Spohn CI variety for generic $X$ is computed for cluster graphs. Computing this degree for more general graphs seems to be a hard problem, since even $\deg(\mathcal M_G)$ is not fully determined in general. While $\dim(\mathcal M_G)$ has been computed in full generality in \cite[Corollary 2.7]{hocsten2002grobner}, the only known result about $\deg(\mathcal M_G)$ is a recursion formula in the binary case for trees (\cite[Theorem 5.5]{develin2003markov}). Coming back to Subsection \ref{subsec CI}, this appendix provides a stand-alone demonstration of the computation of $\deg(\mathcal M_G)$ for certain graphs in the nonbinary case, as this insight could aid in understanding the degree of Spohn CI varieties more generally. It has no direct link with the main results about CI equilibria from the previous sections. The authors had first explored this track with the consideration that one could reuse the argument from \cite[Theorem 6]{portakal2022geometry} to study $\mathcal V_{X,G}$ at least when $\mathcal M_G$ and $\mathcal V_X$ intersect transversally and no irreducible component has to be removed for generic $X$ (easiest case).

\subsection{Notations and classical results}
\label{classical}
For an integer matrix $A$, we denote by $X_A$ the corresponding projective toric variety and $\mathcal P(A)$ the associated lattice polytope (convex hull of the columns of $A$). For a graphical model $G$, $n$ is by default its number of vertices and $\mathcal R_{[n]}=[d_1]\times[d_2]\times\cdots\times[d_n]$ its state space. For simplicity, we take the two notations $\dim(G)=\dim(\mathcal M_G)$ and $\deg(G)=\deg(\mathcal M_G)$.

\begin{definition}
\label{Definition $2.11$}
A \textit{parametrization matrix} $A_G$ of a graphical model $G$ is any $0/1$ matrix such that:

$\bullet$ its columns are indexed by the elements of $\mathcal R_{[n]}=[d_1]\times[d_2]\times\cdots\times[d_n]$

$\bullet$ its rows are indexed by the $(a_i)_{i\in C}\in\mathcal R_C$ for each maximal clique $C\in\mathcal C(G)$

$\bullet$ the entry of row $(a_i)_{i\in C}$ and column $(x_i)_{i\in[n]}$ is $1$ if $a_i=x_i$ for all $i\in C$, and $0$ otherwise.
\end{definition}

\noindent By definition, $\mathcal M_G=X_{A_G}$. All the parameterization matrices of a graphical model are the same up to relabeling of the rows and columns. Thus, when we write ``$=$" with such matrices, we mean that there exists a labeling such that the equality holds.

\begin{definition}\label{Definition $2.12$}
For integer matrices $A,B$, we define $A\otimes B$ as a matrix whose columns are $\left(\begin{smallmatrix}
    a\\ b
\end{smallmatrix}\right)$ for all columns $a$ of $A$ and $b$ of $B$. Its number of rows is the sum of the ones of $A$ and $B$, and its number of columns the product of the ones of $A$ and $B$.
\end{definition}
\noindent This corresponds to a special case of toric fiber products, called the Segre product (\cite{sullivant2007toric}).
\begin{corollary}
\label{Proposition $2.13$:} Given two integer matrices $A$ and $B$, then we have the equalities $\mathcal P(A\otimes B)=\mathcal P(A)\times\mathcal P(B)$ and $X_{A\otimes B}=X_A\times X_B$.
\end{corollary}

\begin{proposition}\label{Theorem $2.15$}
    Let $Y\in\PP^m$ and $Z\in \PP^n$ be equidimensional varieties of possibly distinct dimension. Then the degree of $Y\times Z\in\PP^m\times\PP^n$ is 
    \begin{equation}\label{eq: degre prods}
    \displaystyle \deg(Y\times Z) = \binom{\dim(Y)+\dim(Z)}{\dim(Y)}\deg(Y)\deg(Z).
    \end{equation}
\end{proposition}
\begin{proof}
The Chow ring of $\PP^m\times\PP^n$ is $\mathbb{Z}[x_1,x_2]/\langle x_1^{m+1},x_2^{n+1}\rangle$ (see Equation \eqref{eq:chow segre}).
    A Chow ring computation shows that such degree is the coefficient of $x_1^mx_2^n$ of
    \[
    \left[Y\times Z\right]\cdot(x_1+x_2)^{\dim(Y)+\dim(Z)} = (\deg(Y)x_1^{m-\dim(Y)}+\deg(Z)x_2^{n-\dim(Z)}) \cdot(x_1+x_2)^{\dim(Y)+\dim(Z)}.
    \]
    Such coefficient coincides with \eqref{eq: degre prods}.
\end{proof}

\begin{definition}
    The \textit{disjoint union} $G_1\uplus G_2$ of graphical models $G_1$ and $G_2$ of state spaces $F_1$ and $F_2$ is the graphical model on $F_1\times F_2$ whose edges are only the ones of $G_1$ and $G_2$, no edge is added between a vertex of $G_1$ and a vertex of $G_2$.
\end{definition}
The following result is known, but we include a proof here for completeness.
\begin{proposition}\label{Theorem $2.16$}
For graphical models $G_1$ and $G_2$, $\dim(G_1\uplus G_2)=\dim(G_1)+\dim(G_2)$ and $$\deg(G_1\uplus G_2)={\dim(G_1)+\dim(G_2)\choose\dim(G_1)}\cdot\deg(G_1)\cdot\deg(G_2).$$
Moreover, for any graphical models $G_1,\ldots,G_r$ we have that \begin{equation}\label{eq:deg graph}\deg\bigg(\biguplus_{i=1}^rG_i\bigg)=
\frac{\bigg(\displaystyle\sum_{i=1}^r\dim(G_i)\bigg)!}{\displaystyle\prod_{i=1}^r\dim(G_i)!}\cdot\prod_{i=1}^r\deg(G_i).\end{equation}
\end{proposition}
\begin{proof}
    Since no clique of $G_1$ shares a vertex with a clique of $G_2$, $A_{G_1\uplus G_2}=A_{G_1}\otimes A_{G_2}$ by first taking the rows indexed by the elements of $\mathcal C(G_1)$ then the ones indexed by the elements of $\mathcal C(G_2)$ in Definition \ref{Definition $2.11$}. Then the equalities follow from Corollary \ref{Proposition $2.13$:} and Proposition \ref{Theorem $2.15$}. Formula \eqref{eq:deg graph} is derived by induction.
\end{proof}

\begin{remark}
    Proposition \ref{Theorem $2.16$} coincides with \cite[Lemma $5.4$]{develin2003markov} in the nonbinary case.
\end{remark}

\subsection{Degrees of some graphical models}\label{more degrees} We introduce the notion of join of projective varieties (see \cite[Section 1.2]{kuznetsov2021categorical}) to address more degree computations.

\begin{definition}
\label{def cartesian}
    Let $m,n\ge0$ and $Y\subset\mathbb P^m,Z\subset\mathbb P^n$ be projective varieties. We define the \textit{join} $J(Y,Z)\subset\mathbb P^{m+n+1}$ of $Y$ and $Z$ as the projective variety. $$\left\{(\lambda y_0:\lambda y_1:\cdots:\lambda y_m:\mu z_0:\mu z_1:\cdots:\mu z_n)~\Bigg|~\begin{matrix}
    \lambda,\mu\in\mathbb C\\
    (y_0:y_1:\cdots:y_m)\in Y\\
    (z_0:z_1:\cdots:z_n)\in Z
\end{matrix}\right\}.$$
\end{definition}

We highlight that the ideal of $J(Y,Z)$ is generated by the defining equations of $Y$ and $Z$ in their corresponding variables.
 The next result specializes Definition \ref{def cartesian} to toric varieties. We denote by $\textbf{1}$ a row of $1$'s of given size. 

\begin{corollary}\label{coro toric}
    For integer matrices $A$ and $B$: $$\displaystyle J(X_A,X_B)=X_{\left(\begin{smallmatrix}
    \left(\begin{smallmatrix}
        A\\\textbf{1}
    \end{smallmatrix}\right)&0\\0&\left(\begin{smallmatrix}
        B\\\textbf{1}
    \end{smallmatrix}\right)
\end{smallmatrix}\right)}=X_{\left(\begin{smallmatrix}
    A&0\\0&\left(\begin{smallmatrix}
        B\\\textbf{1}
    \end{smallmatrix}\right)
\end{smallmatrix}\right)}=X_{\left(\begin{smallmatrix}
    \left(\begin{smallmatrix}
        A\\\textbf{1}
    \end{smallmatrix}\right)&0\\0&B
\end{smallmatrix}\right)}.$$
\end{corollary}

\noindent The dimension and degree of join of two varieties in our specific case derives from the following result

\begin{corollary}[{\cite[Corollary 2.3]{aadlandsvik1987joins}}]\label{theo cartesian}
    Let $m,n\ge0$. $Y\subset\mathbb P^m$ and $Z\subset\mathbb P^n$ are irreducible varieties. Then:     $$\dim(J(Y,Z))=\dim(Y)+\dim(Z)+1\text{ and }\deg(J(Y,Z))=\deg(Y)\cdot\deg(Z).$$
\end{corollary}

\begin{remark}
   Corollary \ref{theo cartesian} has a geometric interpretation in the case of toric varieties. From Corollary \ref{coro toric}, $J(X_A,X_B)$ is a toric variety and the corresponding full-dimensional polytope is in fact the free sum (see for instance \cite{STAPLEDON201751}) of $\mathcal P(\Tilde A)$ and $\mathcal P\big(\begin{smallmatrix}
    \Tilde B\\\textbf{1}
\end{smallmatrix}\big)\cong\mathcal P(\Tilde B)$. By Kushnirenko's theorem (\cite{kushnirenko1976polyedres}), $\deg(J(X_A, X_B))$ is the normalized volume, i.e. the number of unitary simplices in a triangulation of this free sum. Such a triangulation can be obtained by linking each simplex in a triangulation of $\mathcal P(\tilde{A})$ to one from $\mathcal P(\tilde{B})$. Thus, the number of simplices in the triangulation equals the product of the numbers of simplices in the triangulations of  $\mathcal P(\tilde{A})$ and  $\mathcal P(\tilde{B})$. Hence, using Kushnirenko's theorem again, $\deg(J(X_A, X_B))$ is the product of $\deg(X_A)$ and $\deg(X_B)$.
\end{remark}

\begin{lemma}\label{Proposition $2.23$}
For graphical models $G_1,G_2$, we have that $$\displaystyle J\left(X_{A_{G_1}}, X_{A_{G_2}}\right)=X_{\left(\begin{smallmatrix}
    A_{G_1} & 0\\ 0 & A_{G_1}
\end{smallmatrix}\right)}.$$
\end{lemma}
\begin{proof} The statement is close to Corollary \ref{coro toric}, we just have to show that we can reparametrize the toric variety without the additional row of $1$'s: $$X_{\left(\begin{smallmatrix}
    A_{G_1} & 0\\ 0 & A_{G_2}
\end{smallmatrix}\right)}=X_{\left(\begin{smallmatrix}
    A_{G_1}&0\\0&\left(\begin{smallmatrix}
        A_{G_2}\\\textbf{1}
    \end{smallmatrix}\right)
\end{smallmatrix}\right)}.$$

\noindent For this, we fix $C\in\mathcal C(G_2)$, and let $\big(\sigma_{i_C}^{(C)}\big)_{i_C\in \mathcal R_C}$ be the family of parameters corresponding to the associated rows in $A_{G_2}$. Let's call $\lambda$ the parameter corresponding to the last row. The parameters only appear to the power $1$ ($0/1$-matrix) and one of the $\sigma_{i_C}^{(C)}$ appears if and only if the other ones are to the power $0$. Therefore, replacing the $\sigma_{i_C}^{(C)}$ by $\lambda\cdot\sigma_{i_C}^{(C)}$ gives rise to the same toric variety without the additional parameter $\lambda$. 
\end{proof}

\begin{definition}\label{def: star}
The \textit{join} $G_1\vee G_2$ of graphical models $G_1$ and $G_2$ of state spaces $F_1$ and $F_2$ is the graphical model on $F_1\times F_2$ whose edges are the ones of $G_1$, the ones of $G_2$, and in addition all the edges that connect every vertex of $G_1$ to every vertex of $G_2$.
\end{definition}

\begin{remark}
If $\overline G$ denotes the complement of the graph $G$ (the graph whose edge set is the complement of the one of $G$), then $G_1\vee G_2=\overline{\overline{G_1}\uplus\overline{G_2}}$.
\end{remark}

\begin{proposition}\label{Proposition $2.24$:}
If $K_p$ is a $p$-clique  and $G$ a graph, then $$A_{K_p\vee G}=\left(\begin{smallmatrix}
    A_G & 0 & 0 & . & . & . & 0\\
    0 & A_G & 0 & . & . & . & 0\\
    0 & 0 & A_G & . & . & . & 0\\
    . & . & . & . & & & . \\
    . & . & . & & . & & . \\
    . & . & . & & & . & . \\
    0 & 0 & 0 & . & . & . & A_G
\end{smallmatrix}\right)\text{ with }(\dim(K_p)+1)~~A_G\text{-blocks.}$$
In particular, a block-diagonal matrix whose blocks are all the parametrization matrices of a single graphical model $G$ is the parametrization matrix of the graphical model $K_p\vee G$ for some $p\ge1$. 
\end{proposition}
\begin{proof}Let $F=[d_1]\times\cdots\times[d_n]$ and $F'$ be the state spaces for graphical models $G$ and $K_p$ respectively. First, the state space of $K_p\vee G$ is $F'\times F$ and its maximal cliques are of the form $K_p\vee C$ where $C$ is a maximal clique of $G$. Now, we consider the following order on the rows and columns of $A_{K_p\vee G}$: 

\begin{itemize}
\item We take any order $\prec$ on $T=\{(a_i)_{i\in C}\mid C\text{ maximal clique of }G\text{ and }~\forall i\in C,~a_i\in[d_i]\}$.
The rows are indexed by the set $$F'\times T=\{(a_i)_{i\in C}\mid C\text{ maximal clique of }K_p\vee G\text{ and }~\forall i\in C,~a_i\in[d_i]\}.$$
We order them first according to the lexicographic order on $F'$, then if the ``$F'$-parts" are equal we compare the rest with $\prec$.

\item We take the lexicographic order on $F'\times F$ for the columns.

\end{itemize}

By construction, this parametrization matrix is block-diagonal, each of the $|F'|=1+\dim(K_p)$ blocks being exactly a parametrization matrix of $G$ (with columns in the lexicographic order and rows ordered by $\prec$). Then, the second statement follows from taking $p=1$, and then any dimension for the unique vertex. 
\end{proof}

\begin{example}\label{Example $2.25$}
We consider the graphical model with vertices  set $[6]$ and with the edge set $\{(i,5),(i,6)\mid i\in[4]\}\cup\{(5,6),(1,2)\}$. Such a graph can be seen as $K_2\vee G$ where $G$ has $4$ vertices (in $[4]$) and one edge between $1$ and $2$ (see Figure \ref{fig: example star}). Then, assuming for instance that all the vertices are binary, the matrix associated to the whole graphical model is given by:

$$\left(\begin{matrix}
    A_G & 0 & 0 & 0\\
    0&A_G&0&0\\
    0&0&A_G&0\\
    0&0&0&A_G
\end{matrix}\right)\text{ where }A_G=\left(\begin{smallmatrix}
    1&1&1&1&0&0&0&0&0&0&0&0&0&0&0&0\\
    0&0&0&0&1&1&1&1&0&0&0&0&0&0&0&0\\
    0&0&0&0&0&0&0&0&1&1&1&1&0&0&0&0\\
    0&0&0&0&0&0&0&0&0&0&0&0&1&1&1&1\\
    1&1&0&0&1&1&0&0&1&1&0&0&1&1&0&0\\
    0&0&1&1&0&0&1&1&0&0&1&1&0&0&1&1\\
    1&0&1&0&1&0&1&0&1&0&1&0&1&0&1&0\\
    0&1&0&1&0&1&0&1&0&1&0&1&0&1&0&1
\end{smallmatrix}\right).$$

In the proof of Proposition \ref{Proposition $2.24$:}, the corresponding $\prec$ on $T$ is lexicographic among the $(a_i)_{i\in C}$ for each $C\in\mathcal C(G)=\{\{1,2\};\{3\};\{4\}\}$ and for all $(x,y,z)\in[2]^{\{1,2\}}\times[2]^{\{3\}}\times[2]^{\{4\}}$, $x\prec y\prec z$.

\begin{figure}[H]
\centering
\begin{minipage}{0.35\textwidth} 
\centering
\begin{tikzpicture}[node distance={1.5cm}, thick, main/.style = {draw, circle}]
    \node[main] (1) {1};
    \node[main] (2) [right of=1] {2};
    \node[main] (3) [right of=2] {3};
    \node[main] (4) [right of=3] {4};
    \node[main] (5) [below right of=1] {5};
    \node[main] (6) [right of=5] {6};
    \draw (1) -- (2);
    \draw (5) -- (6);
    \draw (1) -- (5);
    \draw (2) -- (5);
    \draw (3) -- (5);
    \draw (4) -- (5);
    \draw (1) -- (6);
    \draw (2) -- (6);
    \draw (3) -- (6);
    \draw (4) -- (6);
\end{tikzpicture}
\end{minipage}
\hspace{1cm} 
\begin{minipage}{0.35\textwidth} 
\centering
\begin{tikzpicture}[node distance={1.5cm}, thick,
    main/.style = {draw, circle},
    rednode/.style = {draw=red, fill=red!20, circle},
    bluenode/.style = {draw=blue, fill=blue!20, circle}]
    \node[bluenode] (1) {1};
    \node[bluenode] (2) [right of=1] {2};
    \node[bluenode] (3) [right of=2] {3};
    \node[bluenode] (4) [right of=3] {4};
    \node[rednode] (5) [below right of=1] {5};
    \node[rednode] (6) [right of=5] {6};
    \draw[blue] (1) -- (2);
    \draw[red] (5) -- (6);
    \draw[green] (1) -- (5);
    \draw[green] (2) -- (5);
    \draw[green] (3) -- (5);
    \draw[green] (4) -- (5);
    \draw[green] (1) -- (6);
    \draw[green] (2) -- (6);
    \draw[green] (3) -- (6);
    \draw[green] (4) -- (6);
\end{tikzpicture}
\end{minipage}
\caption{Graph of Example \ref{Example $2.25$} with its $\textcolor{red}{K_2}\textcolor{green}{\vee}\textcolor{blue}{G}$ decomposition.}
\label{fig: example star}
\end{figure}
\end{example}

\begin{theorem}
\label{Theorem $2.26$}
If $K_p$ is a $p$-clique and $G$ a graph, then $\deg(K_p\vee G)=\deg(G)^{\dim(K_p)+1}$.  
\end{theorem}
\begin{proof} Since the projective toric variety $\mathcal M_{K_p\vee G}$ can be associated to a block-diagonal matrix with only $A_G$-blocks by Proposition \ref{Proposition $2.24$:}, we show by induction that the degree of a projective toric variety associated to a block-diagonal matrix made of $x$ $A_G$-blocks is $\deg(G)^x$.
First, the statement is trivial for $x=1$. Now let's assume that the statement holds for $k\le x-1$ for some $x\ge2$. By the second statement in Proposition \ref{Proposition $2.24$:}, we can construct a graphical model $G_2$ with a parametrization matrix $A_{G_2}$ made of exactly $(x-1)$ $A_G$-blocks. Then, the induction hypothesis ensures that for  $G_2$, $\deg(X_{A_{G_2}})=\deg(G)^{x-1}$. Finally, Lemma \ref{Proposition $2.23$} and Corollary \ref{theo cartesian} imply that: $$\deg\bigg(X_{\left(\begin{smallmatrix}
    A_G& 0\\
    0& A_{G_2}
\end{smallmatrix}\right)}\bigg)=\deg\left(J\left(X_{A_G},X_{A_{G_2}}\right)\right)=\deg(X_{A_G})\cdot\deg(X_{A_{G_2}})=\deg(G)^x.$$
\end{proof}

 Since a star graph is a $1$-clique linked to every vertex of a $0$-edge graph, we deduce from Proposition \ref{Theorem $2.16$} and Theorem \ref{Theorem $2.26$} the following result.

\begin{corollary}
    \label{Corollary $2.27$:} Let $G$ be a star graph with $n$ vertices where the center of the star is vertex $1$. Then we have that  $$\deg(G)=\left(\frac{\bigg(\displaystyle\sum_{i=2}^n(d_i-1)\bigg)!}{\displaystyle\prod_{i=2}^n(d_i-1)!}\right)^{d_1}.$$
\end{corollary}

\begin{remark}
    In Corollary \ref{Corollary $2.27$:}, if $d_i=2$ for every $i$, we recover the formula $(n-1)!^2$ shown for the binary case in \cite[Corollary 5.6]{develin2003markov}. If $d_1=1$ we recover the Segre embedding: $d_k=1$ is the same as suppressing vertex $k$ and all the subsequent edges.
\end{remark}

\subsection{Examples of computable degrees} We summarize in Table \ref{table: equivalent operations} the equivalent operations as shown in Subsection \ref{classical} and Subsection \ref{more degrees}. This result enables us to compute recursively the degree of any graphical model as long as the graph can be obtained by combining the disjoint union $\uplus$ and the join $\vee$ with a clique. For such graphs, which form a subset of chordal graphs, the degree can be computed recursively by using alternatively Proposition \ref{Theorem $2.16$} and Theorem \ref{Theorem $2.26$}.

\begin{table}[H]
\begin{tabular}{ |c|c|c|c|c|c| } 
 \hline
 & Graphs $G$ & Matrices $A_G$ & Polytopes $\mathcal P(\cdot)$ & Varieties $\mathcal{M}_G$ & Degrees $\deg(\cdot)$\\
 \hline
 \ref{classical} & Disjoint union $\uplus$ & $\textquoteleft\textquoteleft
$Product" $\otimes$ & Product $\times$ & Segre prod. $\times$ & Multinomial coef. \\ 
 \ref{more degrees} & Join $\vee$ & Block-diagonal & Free sum & Join $J(\cdot,\cdot)$ & Multiplication \\ 
 \hline
\end{tabular}
\caption{Summary of equivalent operations on graph.}
\label{table: equivalent operations}
\end{table}

\begin{example}
\label{Example $2.28$}
Let us compute the degree for all $3$-vertices graphs with vertices $1,2,3$. While $0,1,3$ edges were already known, Theorem \ref{Theorem $2.26$} provides a general formula for $2$ edges since, in this case, it is the star graph with $2 $ leaves.
\begin{itemize}
    \item The $0$-edge graph has dimension $d_1+d_2+d_3-3$ and degree $\dfrac{(d_1+d_2+d_3-3)!}{(d_1-1)!(d_2-1)!(d_3-1)!}$.\vspace*{1mm}
    \item The $1$-edge graph with edge $(2,3)$ has dimension $d_1+d_2d_3-2$ and degree $\dfrac{(d_1+d_2d_3-2)!}{(d_1-1)!(d_2d_3-1)!}$.\vspace*{1mm}
    \item  The $2$-edge graph with missing edge $(1,3)$ has dimension $d_1d_2+d_2d_3-d_2-1$ and degree $$\bigg(\dfrac{(d_1+d_3-2)!}{(d_1-1)!(d_3-1)!}\bigg)^{d_2}.$$
    \item The $3$-edge (complete) graph has dimension $d_1d_2d_3-1$ and degree $1$.

\end{itemize}

\end{example}

\begin{example}\label{Example $2.29$}
We consider the graph with vertex set $[4]$ and edges $(i,j)$ for $1\le i<j\le3$ and $(3,4)$. Since the vertex $3$ is linked to every vertex, using Theorem \ref{Theorem $2.26$}, we just have to compute the degree associated to a $1$-edge $3$-vertex graph on $\{1,2,4\}$, which is Example \ref{Example $2.28$} (see Figure \ref{fig: graph 4 vertices}). Hence, the total degree is $$\left(\dfrac{(d_1d_2+d_4-2)!}{(d_1d_2-1)!(d_4-1)!}\right)^{d_3}.$$

\begin{figure}[H]
\centering
\begin{minipage}{0.3\textwidth} 
\centering
\begin{tikzpicture}[node distance={1.5cm}, thick, main/.style = {draw, circle}]
    \node[main] (2) {2};
    \node[main] (1) [below left of=2] {1};
    \node[main] (3) [below right of=2] {3};
    \node[main] (4) [above right of=3] {4};
    \draw (1) -- (2);
    \draw (1) -- (3);
    \draw (2) -- (3);
    \draw (3) -- (4);
\end{tikzpicture}
\end{minipage}
\hspace{1cm}
\begin{minipage}{0.3\textwidth}
\centering
\begin{tikzpicture}[node distance={1.5cm}, thick,
    main/.style = {draw, circle},
    rednode/.style = {draw=red, fill=red!20, circle},
    bluenode/.style = {draw=blue, fill=blue!20, circle}]
    \node[bluenode] (2) {2};
    \node[bluenode] (1) [below left of=2] {1};
    \node[rednode] (3) [below right of=2] {3};
    \node[bluenode] (4) [above right of=3] {4};
    \draw[blue] (1) -- (2);
    \draw[green] (1) -- (3);
    \draw[green] (2) -- (3);
    \draw[green] (3) -- (4);
\end{tikzpicture}
\end{minipage}
\caption{Graph of Example \ref{Example $2.29$} with its $\textcolor{red}{K_1}\textcolor{green}{\vee}(\textcolor{blue}{K_2\uplus K_1})$ decomposition}
\label{fig: graph 4 vertices}
\end{figure}
\end{example}

\begin{example}\label{Example $2.30$}
We consider the graph with vertex set $[7]$ and edges $(1,i)$ for $i\in[\![2,7]\!]$, $(2,i)$ for $i\in[\![3,7]\!]$, $(3,4)$ and $(5,6)$. Such a graph can be written in the form $K_2\vee\big(K_2\uplus K_2\uplus K_1\big)$ (see Figure \ref{fig: 7 vertices}). Using Theorem \ref{Theorem $2.26$} and Proposition \ref{Theorem $2.16$}, we get the degree formula $$\left(\dfrac{(d_3d_4+d_5d_6+d_7-3)!}{(d_3d_4-1)!(d_5d_6-1)!(d_7-1)!}\right)^{d_1d_2}.$$

\begin{figure}[H]
\centering
\begin{minipage}{0.35\textwidth} 
\centering
\begin{tikzpicture}[node distance={1.5cm}, thick, main/.style = {draw, circle}]
    \node[main] (1) {1};
    \node[main] (2) [below left of=1] {2};
    \node[main] (3) [left of=1] {3};
    \node[main] (4) [left of=2] {4};
    \node[main] (5) [right of=2] {5};
    \node[main] (6) [right of=1] {6};
    \node[main] (7) [below left of=5] {7};

    \draw (1) -- (2);
    \draw (1) -- (3);
    \draw (1) -- (4);
    \draw (1) -- (5);
    \draw (1) -- (6);
    \draw (1) -- (7);
    \draw (2) -- (3);
    \draw (2) -- (4);
    \draw (2) -- (5);
    \draw (2) -- (6);
    \draw (2) -- (7);
    \draw (3) -- (4);
    \draw (5) -- (6);
\end{tikzpicture}
\end{minipage}
\hspace{1cm}
\begin{minipage}{0.35\textwidth}
\centering
\begin{tikzpicture}[node distance={1.5cm}, thick,
    main/.style = {draw, circle},
    rednode/.style = {draw=red, fill=red!20, circle},
    bluenode/.style = {draw=blue, fill=blue!20, circle}]
    \node[rednode] (1) {1};
    \node[rednode] (2) [below left of=1] {2};
    \node[bluenode] (3) [left of=1] {3};
    \node[bluenode] (4) [left of=2] {4};
    \node[bluenode] (5) [right of=2] {5};
    \node[bluenode] (6) [right of=1] {6};
    \node[bluenode] (7) [below left of=5] {7};

    \draw[red] (1) -- (2);
    \draw[green] (1) -- (3);
    \draw[green] (1) -- (4);
    \draw[green] (1) -- (5);
    \draw[green] (1) -- (6);
    \draw[green] (1) -- (7);
    \draw[green] (2) -- (3);
    \draw[green] (2) -- (4);
    \draw[green] (2) -- (5);
    \draw[green] (2) -- (6);
    \draw[green] (2) -- (7);
    \draw[blue] (3) -- (4);
    \draw[blue] (5) -- (6);
\end{tikzpicture}
\end{minipage}
\caption{Graph of Example \ref{Example $2.30$} with its $\textcolor{red}{K_2}\textcolor{green}{\vee}\big(\textcolor{blue}{K_2\uplus K_2\uplus K_1}\big)$ decomposition}
\label{fig: 7 vertices}
\end{figure}
\end{example}

\begin{example}\label{Example $2.31$}
We consider the graph with vertex set $[5]$ and edges $(2,3),(3,4)$ and $(i,5)$ for $i\in[4]$. It can be written in the form $K_1\vee\big(K_1\uplus(K_1\vee(K_1\uplus K_1))\big)$ with vertices $5,1,3,2,4$ respectively (see Figure \ref{fig: 5 vertices}). Using Theorem \ref{Theorem $2.26$} and Proposition \ref{Theorem $2.16$}, we get the degree formula: $$\left(\dfrac{(d_1+d_2d_3+d_3d_4-d_3-2)!}{(d_1-1)!(d_2d_3+d_3d_4-d_3-1)!}\cdot\bigg(\dfrac{(d_2+d_4-2)!}{(d_2-1)!(d_4-1)!}\bigg)^{d_3}\right)^{d_5}$$

\begin{figure}[H]
\centering
\begin{minipage}{0.35\textwidth} 
\centering
\begin{tikzpicture}[node distance={1.5cm}, thick, main/.style = {draw, circle}]
    \node[main] (1) {1};
    \node[main] (2) [above right of=1] {2};
    \node[main] (3) [right of=2] {3};
    \node[main] (5) [right of=1] {5};
    \node[main] (4) [right of=5] {4};
    \draw (5) -- (1);
    \draw (5) -- (2);
    \draw (5) -- (3);
    \draw (5) -- (4);
    \draw (2) -- (3);
    \draw (3) -- (4);
\end{tikzpicture}
\end{minipage}
\hspace{1mm}
\begin{minipage}{0.6\textwidth}
\begin{subfigure}
\centering
\begin{tikzpicture}[node distance={1cm}, thick,
    rednode/.style = {draw=red, fill=red!20, circle, minimum size=0.5cm, inner sep=0pt},
    bluenode/.style = {draw=blue, fill=blue!20, circle, minimum size=0.5cm, inner sep=0pt}]
    \node[bluenode] (1) {1};
    \node[bluenode] (2) [above right of=1] {2};
    \node[bluenode] (3) [right of=2] {3};
    \node[rednode] (5) [right of=1] {5};
    \node[bluenode] (4) [right of=5] {4};
    \draw[green] (5) -- (1);
    \draw[green] (5) -- (2);
    \draw[green] (5) -- (3);
    \draw[green] (5) -- (4);
    \draw[blue] (2) -- (3);
    \draw[blue] (3) -- (4);
\end{tikzpicture}
\end{subfigure}
\hspace{3mm}
\begin{subfigure}
\centering
\begin{tikzpicture}[node distance={1cm}, thick,
    erednode/.style = {draw=red!25, fill=red!5, circle, minimum size=0.5cm, inner sep=0pt},
    bluenode/.style = {draw=blue, fill=blue!12, circle, minimum size=0.5cm, inner sep=0pt},
    ebluenode/.style = {draw=brown, fill=brown!20, circle, minimum size=0.5cm, inner sep=0pt},
    egreen/.style={green!25}]
    \node[ebluenode] (1) {1};
    \node[bluenode] (2) [above right of=1] {2};
    \node[bluenode] (3) [right of=2] {3};
    \node[erednode] (5) [right of=1] {5};
    \node[bluenode] (4) [right of=5] {4};
    \draw[egreen] (5) -- (1);
    \draw[egreen] (5) -- (2);
    \draw[egreen] (5) -- (3);
    \draw[egreen] (5) -- (4);
    \draw[blue] (2) -- (3);
    \draw[blue] (3) -- (4);
\end{tikzpicture}
\end{subfigure}
\hspace{3mm}
\begin{subfigure}
\centering
\begin{tikzpicture}[node distance={1cm}, thick,
    rednode/.style = {draw=purple, fill=purple!20, circle, minimum size=0.5cm, inner sep=0pt},
    erednode/.style = {draw=red!25, fill=red!5, circle, minimum size=0.5cm, inner sep=0pt},
    bluenode/.style = {draw=blue!80, fill=blue!16, circle, minimum size=0.5cm, inner sep=0pt},
    ebluenode/.style = {draw=purple!25, fill=purple!5, circle, minimum size=0.5cm, inner sep=0pt},
    egreen/.style={green!25}]
    \node[ebluenode] (1) {1};
    \node[bluenode] (2) [above right of=1] {2};
    \node[rednode] (3) [right of=2] {3};
    \node[erednode] (5) [right of=1] {5};
    \node[bluenode] (4) [right of=5] {4};
    \draw[egreen] (5) -- (1);
    \draw[egreen] (5) -- (2);
    \draw[egreen] (5) -- (3);
    \draw[egreen] (5) -- (4);
    \draw[green!90!black] (2) -- (3);
    \draw[green!90!black] (3) -- (4);
\end{tikzpicture}
\end{subfigure}
\end{minipage}
\caption{Graph of Example \ref{Example $2.31$} with its $\textcolor{red}{K_1}\textcolor{green}{\vee}\big(\textcolor{brown}{K_1}\uplus(\textcolor{purple}{K_1}\textcolor{green!90!black}{\vee}(\textcolor{blue!80}{K_1\uplus K_1}))\big)$ decomposition}
\label{fig: 5 vertices}
\end{figure}
\end{example}

Table \ref{table: equivalent operations} also enables decomposing (thus simplifying) the degree computation for a big graphical model as soon as it has either different connected components (Proposition \ref{Theorem $2.16$}) or a vertex linked to all other vertices (Theorem \ref{Theorem $2.26$}). In some cases, it may avoid using very heavy volume computations based on the lattice polytope (convex hull) associated to the parametrization matrix.

\section*{Acknowledgements} 
This paper grew out of the master's internship of MB at the Max Planck Institute for Mathematics in the Sciences. The work was recognized with the Louis-\'{E}douard Rivot Medal, awarded in 2024 by the \'{E}cole Polytechnique in France. We are grateful to Bernd Sturmfels for insightful comments on the geometric interpretation of Corollary \ref{theo cartesian}, and to Leonie Kayser for her guidance with Porteous' formula. We thank Elke Neuhaus and Luca Sodomaco for their valuable feedback on an earlier version of this paper.
\bibliographystyle{abbrv}
\bibliography{biblio}

@book{hartshorne2013algebraic,
  title={Algebraic geometry},
  author={Hartshorne, Robin},
  volume={52},
  year={2013},
  publisher={Springer Science \& Business Media}
}

@article{hocsten2002grobner,
  title={Gr{\"o}bner bases and polyhedral geometry of reducible and cyclic models},
  author={Ho{\c{s}}ten, Serkan and Sullivant, Seth},
  journal={Journal of Combinatorial Theory, Series A},
  volume={100},
  number={2},
  pages={277--301},
  year={2002},
  publisher={Elsevier}
}

@article{STAPLEDON201751,
title = {Counting lattice points in free sums of polytopes},
journal = {Journal of Combinatorial Theory, Series A},
volume = {151},
pages = {51-60},
year = {2017},
issn = {0097-3165},
doi = {https://doi.org/10.1016/j.jcta.2017.04.004},
url = {https://www.sciencedirect.com/science/article/pii/S0097316517300420},
author = {Alan Stapledon}
}

@misc{stacks-project,
  author       = {The {Stacks project authors}},
  title        = {The Stacks project},
  howpublished = {\url{https://stacks.math.columbia.edu}},
  year         = {2025},
}

@article{develin2003markov,
  title={Markov bases of binary graph models},
  author={Develin, Mike and Sullivant, Seth},
  journal={Annals of Combinatorics},
  volume={7},
  pages={441--466},
  year={2003},
  publisher={Springer}
}

@article{portakal2022geometry,
  author    = {Irem Portakal and Bernd Sturmfels},
  title     = {Geometry of Dependency Equilibria},
  journal   = {Rendiconti dell'Istituto di Matematica dell'Universit\`a di Trieste},
  volume    = {54},
  year      = {2022},
  pages     = {5},
  doi       = {10.13137/2464-8728},
  eprint    = {arXiv:2201.05506},
  url       = {https://arxiv.org/abs/2201.05506}
}

@article{kuznetsov2021categorical,
  title={Categorical joins},
  author={Kuznetsov, Alexander and Perry, Alexander},
  journal={Journal of the American Mathematical Society},
  volume={34},
  number={2},
  pages={505--564},
  year={2021}
}

@article{aadlandsvik1987joins,
  title={Joins and higher secant varieties},
  author={{\AA}dlandsvik, Bj{\o}rn},
  journal={Mathematica Scandinavica},
  volume={61},
  pages={213--222},
  year={1987},
  publisher={JSTOR}
}

@article{portakal2024game,
  author    = {Irem Portakal and Javier Sendra-Arranz},
  title     = {Game Theory of Undirected Graphical Models},
  journal   = {Journal of Algebra},
  volume    = {666},
  pages     = {574--606},
  year      = {2025},
  doi       = {10.1016/j.jalgebra.2023.12.014},
  eprint    = {arXiv:2402.13246},
  url       = {https://arxiv.org/abs/2402.13246}
}

@book{sullivant2018algebraic,
  title={Algebraic statistics},
  author={Sullivant, Seth},
  volume={194},
  year={2018},
  publisher={American Mathematical Soc.}
}

@article{spohn2007dependency,
  title={Dependency equilibria},
  author={Spohn, Wolfgang},
  journal={Philosophy of Science},
  volume={74},
  number={5},
  pages={775--789},
  year={2007},
  publisher={Cambridge University Press}
}

@article{portakal2024nash,
  title={Nash conditional independence curve},
  author={Portakal, Irem and Sendra--Arranz, Javier},
  journal={Journal of Symbolic Computation},
  volume={122},
  pages={102255},
  year={2024},
  publisher={Elsevier}
}

@article{sullivant2007toric,
  title={Toric fiber products},
  author={Sullivant, Seth},
  journal={Journal of Algebra},
  volume={316},
  number={2},
  pages={560--577},
  year={2007},
  publisher={Elsevier}
}

@article{nash1950equilibrium,
  title={Equilibrium points in n-person games},
  author={Nash Jr, John F},
  journal={Proceedings of the national academy of sciences},
  volume={36},
  number={1},
  pages={48--49},
  year={1950},
  publisher={national academy of sciences}
}

@article{connelly2025gametheory,
  title={The {GameTheory} package for Macaulay2},
  author={Connelly, Erin and Galgano, Vincenzo and He, Zhuang and Maletto, Giacomo and Neuhaus, Elke and Portakal, Irem and Tillmann-Morris, Hannah and Zhao, Chenyang},
  journal={arXiv preprint arXiv:2507.16755},
  year={2025}
}

@article{portakal2024dependency,
title = {Dependency equilibria: Boundary cases and their real algebraic geometry},
journal = {Advances in Applied Mathematics},
volume = {168},
pages = {102890},
year = {2025},
issn = {0196-8858},
doi = {https://doi.org/10.1016/j.aam.2025.102890},
url = {https://www.sciencedirect.com/science/article/pii/S0196885825000521},
author = {Irem Portakal and Daniel Windisch}
}

@article{brandenburg2025combinatorics,
  title={Combinatorics of correlated equilibria},
  author={Brandenburg, Marie-Charlotte and Hollering, Benjamin and Portakal, Irem},
  journal={Experimental Mathematics},
  volume={34},
  number={2},
  pages={212--224},
  year={2025},
  publisher={Taylor \& Francis}
}

@article{aumann1974subjectivity,
  author    = {Robert J. Aumann},
  title     = {Subjectivity and correlation in randomized strategies},
  journal   = {Journal of Mathematical Economics},
  volume    = {1},
  number    = {1},
  pages     = {67--96},
  year      = {1974},
  doi       = {10.1016/0304-4068(74)90037-8}
}

@book{Gelfand1994,
  title     = {Discriminants, Resultants, and Multidimensional Determinants},
  author    = {Gelfand, I.M. and Kapranov, M.M. and Zelevinsky, A.V.},
  year      = {1994},
  publisher = {Birkh\"{a}user Boston},
  isbn      = {978-0-8176-4771-1},
  url       = {https://link.springer.com/book/10.1007/978-0-8176-4771-1}
}

@article{abo2025vector,
  title={A vector bundle approach to {N}ash equilibria},
  author={Abo, Hirotachi and Portakal, Irem and Sodomaco, Luca},
  journal={arXiv preprint arXiv:2504.03456},
  year={2025}
}

@article{kushnirenko1976polyedres,
  author    = {A. G. Kushnirenko},
  title     = {Poly\`edres de Newton et nombres de {Milnor}},
  journal   = {Inventiones mathematicae},
  volume    = {32},
  number    = {1},
  pages     = {1--31},
  year      = {1976},
  doi       = {10.1007/BF01389785}
}

@article{spohn2003dependency,
  author    = {Wolfgang Spohn},
  title     = {Dependency Equilibria and the Causal Structure of Decision and Game Situations},
  journal   = {Homo Oeconomicus},
  volume    = {20},
  pages     = {195--255},
  year      = {2003},
  url       = {https://kops.uni-konstanz.de/server/api/core/bitstreams/36c13c9f-be31-41ac-89ef-fde2932d0a3d/content}
}

@book{eisenbud20163264,
  title={3264 and all that: A second course in algebraic geometry},
  author={Eisenbud, David and Harris, Joe},
  year={2016},
  publisher={Cambridge University Press}
}

@book{eisenbud2006geometry,
  title={The geometry of syzygies: a second course in algebraic geometry and commutative algebra},
  author={Eisenbud, David},
  year={2006},
  publisher={Springer}
}

@book{mangolte2020real,
  title={Real algebraic varieties},
  author={Mangolte, Fr{\'e}d{\'e}ric},
  year={2020},
  publisher={Springer}
}

@book{harris2013algebraic,
  title={Algebraic geometry: a first course},
  author={Harris, Joe},
  volume={133},
  year={2013},
  publisher={Springer Science \& Business Media}
}

\end{document}